\documentclass[11pt]{amsart}
\usepackage{fullpage}
\usepackage[utf8]{inputenc}
\usepackage{microtype}

\usepackage{graphicx}

\usepackage{amsmath,amssymb,stmaryrd,array,multirow,dsfont,mathrsfs}
\usepackage{cite}

\usepackage{scalerel,stackengine}
\usepackage{mathtools}

\usepackage{amsthm}
\usepackage{enumerate}
\usepackage{xcolor}
\usepackage{color}
\usepackage{hyperref}
\usepackage{indentfirst}
\usepackage{pifont}
\usepackage[T1]{fontenc}
\usepackage{microtype}
\usepackage{graphics}
\usepackage{amsthm}
\usepackage{extarrows}
\usepackage{stmaryrd}
\graphicspath{{Images/}}
\xdefinecolor{darkgreen}{rgb}{0,0.4,0}
  
\usepackage{tikz}

 \newcommand{\bpm}{\begin{pmatrix}}
 \newcommand{\epm}{\end{pmatrix}}
  \newcommand{\bbm}{\begin{bmatrix}}
 \newcommand{\ebm}{\end{bmatrix}}
\newcommand{\beq}{\begin{equation}}
\newcommand{\eeq}{\end{equation}}
\newcommand{\beqs}{\begin{equation*}}
\newcommand{\eeqs}{\end{equation*}}
\newcommand{\beal}{\begin{align}}
\newcommand{\eeal}{\end{align}}
\newcommand{\beals}{\begin{align*}}
\newcommand{\eeals}{\end{align*}}
\newcommand{\ben}{\begin{eqnarray}}
\newcommand{\een}{\end{eqnarray}}
\newcommand{\beno}{\begin{eqnarray*}}
\newcommand{\eeno}{\end{eqnarray*}}
\renewcommand{\Re}{{\rm Re}\,}
\renewcommand{\Im}{{\rm Im}\,}
\renewcommand{\div}{{\rm div}}

\newcommand{\Id}{{\rm Id}\,}

\newcommand{\Rmnum}[1]{\uppercase\expandafter{\romannumeral #1} }
 \numberwithin{equation}{section}

\allowdisplaybreaks[4]
\newtheorem{thm}{Theorem}[section]
\newtheorem{lem}[thm]{Lemma}
\newtheorem{prop}[thm]{Proposition}
\newtheorem{rmk}[thm]{Remark}
\newtheorem{cor}[thm]{Corollary}

\def\curl{\mathop{\rm curl}\nolimits}

\def \d {\mathrm {d}}
\def\cA{{\mathcal A}}
\def\cB{{\mathcal B}}

\def\cE{{\mathcal E}}

\def\cH{{\mathcal H}}
\def\cI{{\mathcal I}}

\def\cL{{\mathcal L}}
\def\cM{{\mathcal M}}

\def\cO{{\mathcal O}}

\def\cU{{\mathcal U}}

\let\f=\frac
\def \p {\partial}
\def\mR {\mathbb{R}}
\def \eps{\varepsilon}

\def \vep {\varepsilon}
\def \diag {\text{diag}}
\def\vp{\varphi}
\def \kpa {\kappa}

\def \rM {\mathrm{M}}
\def \rD {\mathrm{D}}
\def \rA {\mathrm{A}}
\def \rF {\mathrm{F}}

\def \rP {\mathrm{P}}

\def \rG {\mathrm{G}}
\def \rN {\mathrm{N}}
\def \rQ {\mathrm{Q}}
\def \rV {\mathrm{V}}
\def \rB  {\mathrm{B}}
\def \rU {\mathrm{U}}

\def \tT {\tilde{T}}

\def \pt {\partial_{t}}

\def\na{\nabla}

\def \ep{\epsilon}
\def \per{\mathrm{per}}
\def \i {\mathrm{i}}

\DeclareMathOperator{\tah}{th}
\DeclareMathOperator{\ch}{ch}
\DeclareMathOperator{\sh}{sh}

\date{\today}

\title{On the spectral stability of  periodic capillary-gravity waves}
\begin{document}

\author{Changzhen Sun}
\address{Université Marie-et-Louis-Pasteur, Laboratoire de Mathématiques de Besançon, UMR CNRS 6623, 25000 Besançon, France}
\email{changzhen.sun@univ-fcomte.fr}

\author{Erik Wahlén}
\address{Centre for Mathematical Sciences, Lund University, P.O. Box 118, 22100 Lund, Sweden}
\email{erik.wahlen@math.lu.se}

\begin{abstract}
In this paper, we investigate the spectral stability of periodic traveling waves in the two dimensional capillary-gravity water wave problem. We derive a stability criterion based on an index function, whose sign determines the spectral stability of the waves. This result aligns with earlier formal analyses by Djordjević \& Redekopp \cite{D-R-packets} and Ablowitz \& Segur \cite{A-Segur}, which employed the nonlinear Schrödinger approximation in the modulational regime. In particular, we show that instability is excluded near spectral crossings away from the origin when the surface tension is positive and the inverse square of the Froude number $\alpha\in(0,1]$, which results from the fact that the corresponding Krein signatures are identical. 
It is also shown that there exists $\alpha_1 = (23 - 3\sqrt{41})/8$ and a curve $\beta_0: (\alpha_1, 1]\rightarrow \mR_{+},$ such that for any $\alpha \in (\alpha_1, 1]$, small amplitude periodic waves are spectrally stable when $\beta > \beta_0(\alpha)$. These findings highlight the stabilizing effect of surface tension on periodic capillary-gravity waves.

\vspace{0.5em}


\vspace{0.5em}
\end{abstract}

\maketitle

\tableofcontents 

\section{Introduction}

In this paper, we study the spectral stability of periodic waves of the two-dimensional water-wave system in a channel with finite depth. 
The equations of motions are given by 
\beq\label{o-in}
    \left\{ \begin{array}{l}
         \pt U+U \cdot \na U+\na P +g e_z=0, \\[3pt]
          \div\, U=0,\, \curl U=0.
    \end{array}
    \right. (t,x,z)\in \Omega_t \,.
\eeq
Here the fluid domain $\Omega_t$, which is also  unknown, is defined as 
\begin{align*}
    \Omega_{t}\coloneqq\{(x,z)\in \mR^2 : 0<z<h+\zeta(t,x) \},
\end{align*}
where $h$ is a fixed constant and $\zeta(t,x)>-h$ is the free surface,  which satisfies the kinematic condition 
\beq\label{o-sur}
\p_t \zeta (t,x) +\p_x \zeta  (t,x)\cdot U_1(t,x,  h+\zeta(x) )-U_2(t,x, h+\zeta(x))=0\,.
\eeq
The system \eqref{o-in}, \eqref{o-sur} is closed with the boundary conditions for $U_2$ at the lower boundary and for the pressure $P$ at the upper boundary: 
\begin{align}
& U_2|_{z=0}=0, \\
 &   P|_{z=h+\zeta}=P_{atm}-T \p_x \bigg(\f{\p_x\zeta}{\sqrt{1+(\p_x\zeta)^2}}\bigg), \label{pressure-bd}
\end{align}
where $P_{atm}$ is a constant and can be chosen to be zero, and $T$ is the surface tension parameter. 

As the velocity $U$ is curl-free, we can write $U=\na\phi$, with $\phi$ being the velocity potential. Changing frame $x\rightarrow x-ct$,
equations \eqref{o-in}--\eqref{pressure-bd} reduce to 
\beqs
\left\{\begin{array}{ll}
(\p_x^2+\p_z^2)\phi=0,\, &\mathrm{ in\, } \,\Omega_t, \\
    \p_t \phi-\p_x\phi+\f12(\p_x^2+\p_z^2)\phi+g\,\zeta-T\p_x \bigg(\f{\p_x\zeta}{\sqrt{1+(\p_x\zeta)^2}})=0 \qquad & \mathrm{ on\, }  z=h+\zeta, \\
\pt \zeta-\p_x\zeta+\p_x\zeta\p_x\phi-
\p_z\phi=0 &\mathrm{ on\, }  z=h+\zeta, \\ 
   \p_z\phi=0  & \mathrm{ on\, } z=0.
\end{array}\right.
\eeqs
To make the equations non-dimensional, we perform the  change of variables
\beq\label{changeofvarible} 
\tilde{t}={ct}/{h}, \quad 
\tilde{x}={x}/{h}, \quad \tilde{z}={z}/{h}, \quad \tilde{\zeta}=\zeta/h, \quad  \tilde{\phi}=\phi/(ch). 
\eeq
The new equations then read
\begin{equation}\label{eq-inter}
  \left\{\begin{array}{ll}
   (\p_x^2+\p_z^2)\phi=0\, & \mathrm{ in\, }  0<z<1+\zeta,   \\
    \p_t \phi-\p_x\phi+\f12(\p_x^2+\p_z^2)\phi+\alpha\,\zeta-\beta\p_x \bigg(\f{\p_x\zeta}{\sqrt{1+(\p_x\zeta)^2}}\bigg)=0  &  \mathrm{ on\, }  z=1+\zeta, \\
    \pt \zeta-\p_x\zeta+\p_x\zeta\p_x\phi-
\p_z\phi=0 &\mathrm{ on\, }  z=1+\zeta, \\ 
   \p_z\phi=0  & \mathrm{ on\, } z=0,
\end{array}    \right.
\end{equation}
where 
\begin{align}\label{eq-Froude-Weber}
    \alpha=\f{gh}{c^2}, \qquad \beta=\f{T}{hc^2}
\end{align}
are the inverse square of the Froude number and the Weber number, respectively, and we have dropped the tildes for notational convenience.

Let $\vp$ be the trace of the velocity potential on the surface $1+\zeta$ and define the Dirichlet--Neumann operator $G[\zeta]\cdot$ by
\begin{align*}
    G[\zeta]\vp=(\p_z \phi-\p_x\zeta\cdot\p_x \phi)|_{z=1+\zeta}\, ,
\end{align*}
where $\phi$ is harmonic with $\phi|_{z=1+\zeta}=\vp$ and $\partial_z \phi|_{z=0}=0$.
Then \eqref{eq-inter} can be reduced further to a system in $(\zeta,\vp)$ evaluated on the surface
\beq \label{sur-mov-uni}
    \left\{ \begin{array}{l}
         \pt \zeta =\p_x \zeta+ G[\zeta]\varphi\,, \\
         \pt \varphi=\p_x\vp-\f12 |\p_x\varphi|^2+\f12 \f{(G[\zeta]\varphi+\p_x\vp\cdot\p_x\zeta)^2}{1+|\p_x\zeta|^2}+\alpha\zeta+\beta\p_{x}\big(\f{\p_x \zeta}{1+|\p_x\zeta|^2}\big).
    \end{array}
    \right. 
\eeq

When $\beta=0$, the above system reduces to the pure gravity water-wave problem, which admits a family of periodic stationary solutions for
 $\alpha>1$, commonly known as Stokes waves. The stability of Stokes waves has attracted significant  interest in recent years. We refer the reader to Section \ref{background-stokes} for further background on this topic.
 
 When $\beta>0$, that is, when  surface tension is taken into account, the system is referred to as the capillary-gravity water-wave system. 
The bifurcation of one-dimensional small-amplitude periodic waves to \eqref{sur-mov-uni} from the trivial solution is determined by the limiting wave number, which is the positive root of the linear dispersion relation. 
\begin{align}\label{disp-relation-1}
    D(k)\coloneqq (\alpha+\beta k^2)\sh (k)-k\ch (k)=0.
\end{align}
The number and multiplicity of the positive roots of \eqref{disp-relation-1}  vary, and thus the form of the periodic waves differs, depending on the values of $(\alpha,\beta);$ we refer to the introduction of \cite{Mariana-Tien-Erik} for a summary of known results. 
Specifically, 
it is  known that in the regions 
\beq\label{region-I-III}
\rm I=
\{(\alpha,\beta) : \alpha\in (0,1), \beta>0\}, \qquad \rm III=\{(\alpha,\beta) : \alpha=1, \beta\in (0,1/3)\},
\eeq
the system \eqref{sur-mov-uni} admits a single family of small-amplitude periodic solutions with wave number close to the unique positive root of \eqref{disp-relation-1}. In contrast, in the region $$\rm II=\{(\alpha,\beta) : \alpha>1, 0<\beta<\tilde{\beta}(\alpha) \},$$
there exist two geometrically distinct families of periodic waves with limiting wave numbers  $\kpa_1, \kpa_2$, corresponding to two distinct positive roots of  \eqref{disp-relation-1} under the assumption that  $\f{\kpa_2}{\kpa_1}\notin \mathbb{N}$.
Here $\tilde{\beta}(\alpha)\in (0,\f13)$ is a smooth decreasing curve in 
$\alpha\in (1,+\infty)$.  In the present work, we will mainly focus on the stability of periodic waves in region $\rm I\cup \rm III$ which is of particular interest as they arise as the limiting states of a family of generalized solitary waves \cite{Lombardi-ARMA,Thomas-CPAM} at $\pm\infty$. Investigating the stability of periodic waves in this regime thus constitutes a crucial step in establishing the stability of generalized solitary waves. The stability analysis of periodic waves in region $\rm II$ is more involved due to the presence of non-zero spectral crossings with opposite Krein signatures. A brief discussion of their stability near such crossings is given in Appendix \ref{app-regionII-highcrossing}.

Let $(\tilde{\zeta}_{\vep}, \tilde{\vp}_{\vep})$ be the single family of periodic small-amplitude waves for $(\alpha,\beta)\in \rm I$,
parameterized by the amplitude $\varepsilon$, 
\begin{align*}
    (\tilde{\zeta}_{\vep}, \tilde{\vp}_{\vep})(x)=({\zeta}_{\vep}, {\vp}_{\vep})(k_{\vep} x),
\end{align*}
where $k_{\vep}$ is the wave number and 
$({\zeta}_{\vep}, {\vp}_{\vep})$ are $2\pi$-periodic functions. Note that $ (\tilde{\zeta}_{\vep}, \tilde{\vp}_{\vep})$ and the wave number  also depend on the parameters $(\alpha,\beta)$, but we suppress this dependence for notational simplicity. 
See Lemma \ref{lem-waveprofile} for some useful information on  the wave profiles $({\zeta}_{\vep}, {\vp}_{\vep})$. To study the stability of $(\tilde{\zeta}_{\vep}, \tilde{\vp}_{\vep})$ for the system \eqref{sur-mov-uni}, it is equivalent to study the stability of  $({\zeta}_{\vep}, {\vp}_{\vep})$ for the  system obtained by the scaling  $x\rightarrow k_{\vep} x$.
By using the shape derivative of the Dirichlet--Neumann operator \cite{Lannes}, we find that 
the linearization of the scaled system
around $({\zeta}_{\vep}, {\vp}_{\vep})$ takes the form 
 \begin{align*}
\left\{ \begin{array}{l}
     \displaystyle  \pt \zeta= k_{\vep}\p_x \zeta+ G[\zeta_{\vep}]\varphi-G[\zeta_{\vep}](Z_{\vep}\zeta)-k_{\vep}\p_x(v_{\vep}\zeta),   \\ [5pt]
     \pt \vp =(1-v_{\vep})k_{\vep}\p_x\vp+Z_{\vep}G[\zeta_{\vep}](\vp-Z_{\vep}\zeta)+P[\zeta_{\vep}]\zeta-(\alpha+Z_{\vep}k_{\vep}\p_x v_{\vep})\zeta,
\end{array}\right.   
\end{align*}
with 
\beq \label{defZv}
Z_{\vep}\coloneqq \f{G[\zeta_{\vep}]\varphi_{\vep}+k_{\vep}^2\p_x\zeta_{\vep}\p_x\vp_{\vep}}{1+|k_{\vep}\p_x\zeta_{\vep}|^2}, \qquad \qquad v_{\vep}\coloneqq k_{\vep}(\p_x \vp_{\vep}-Z_{\vep}\p_x\zeta_{\vep})
\eeq   
and 
\beq \label{def-P}
P[\zeta_{\vep}]
\coloneqq \beta k_{\vep}\p_x\bigg(
\f{k_{\vep}\p_x}{(1+(k_{\vep}\p_x\zeta_{\vep})^2)^{\f32}}\bigg).
\eeq
Moreover, the scaled D-N operator $G[\zeta_{\vep}]$ is defined as $G[\zeta_{\vep}]\vp\coloneqq(\p_z\Phi-k_{\vep}^2\p_x\Phi\p_x\zeta_{\vep})|_{z=1+\zeta_{\vep}}$, where $\Phi$ solves the elliptic problem 
\begin{align*}
    (k_{\vep}^2\p_x^2+\p_z^2)\Phi=0 \,\, \text { in } 0< z <1+\zeta_{\vep}; \quad \Phi|_{z=1+\zeta_{\vep}}=\vp, \quad \p_z\Phi|_{z=0}=0.
\end{align*}

Define the new unknowns $V_1=\zeta, V_2=\vp-Z_{\vep}\zeta$. Then the linearized system is changed into 
\beq\label{def-Lep}
\pt {V}=
L^{\vep} {V}, \qquad \qquad L^{\vep}= \left( \begin{array}{cc}
   k_{\vep}\p_x(d_{\vep}\cdot)  &  G[\zeta_{\vep}]\cdot   \\[5pt]
P[\zeta_{\vep}]-  w_{\vep}& k_{\vep}d_{\vep}\p_x
\end{array}\right),
\eeq
 where 
 \beqs 
 d_{\vep}\coloneqq 1- v_{\vep}, \qquad \,w_{\vep}\coloneqq  \alpha-d_{\vep}k_\vep \p_x Z_{\vep}.
 \eeqs
Here we again suppress the dependence of $L^{\vep}$ on the parameters 
 $(\alpha,\beta)$ for simplicity.
 
We are interested in the spectral stability of the periodic waves for the system \eqref{sur-mov-uni} under localized perturbations. This 
is equivalent to studying the spectrum of the linear operator $L^{\vep}$ in the (localized) space 
$$Y=H^1(\mR)\times H_{*}^{\f12}(\mR),\qquad  \big(H_{*}^{\f12}(\mR)\coloneqq\sqrt{G[0]}^{-1}L^2(\mR)\big)$$
with the domain $H^{\f52}(\mR)\times H_{*}^{\f32}(\mR)$. 

An important feature for the operator $L^{\vep}$
is that it is of Hamiltonian form, 
\begin{align}\label{relation-imp}
    L^{\vep}=J A^{\vep},
\end{align}
with 
\begin{align*}
J\,&=\,\begin{pmatrix}0&1\\-1&0\end{pmatrix}\,,
\qquad A^{\vep}=\,\begin{pmatrix}
-P[\zeta_{\vep}]+w_{\vep}&-d_{\vep} k_{\vep}\p_x\\[3pt]
 k_{\vep} \p_x(d_{\vep}\cdot)&G[\zeta_{\vep}]\end{pmatrix}.
\end{align*}
From this one derives 
\begin{align}\label{Hamiltonianstructure}
    (L^{\vep})^{*}=-J L^{\vep} J^{-1}.
\end{align}
Consequently, the spectrum of $L^{\vep}$ is symmetric with respect to the imaginary axis and thus the periodic waves are stable (i.e.,~have no spectrum in the open right half-plane) if and only if the spectrum of $L^{\vep}$ is contained in the imaginary axis.

Our main result is the following. 
\begin{thm}\label{thm-1}
Let 
$(\alpha,\beta)$ belong to the region $\rm I \cup \rm III$ defined in \eqref{region-I-III},
and $\kappa=\kappa(\alpha,\beta)$ be uniquely determined by the dispersion relation 
\beq\label{disp-relation}
\alpha+\beta\kpa^2=\f{\kpa}{\tah(\kpa)}.
\eeq
Let \[e_{*}=\f12\Big(1+\f{2\kpa}{\sh(2\kpa)}\Big)+\beta\kpa\tah(\kpa)>1\geq \sqrt{\alpha}.\]
There exists a number $C=C(\alpha, \beta)$, 
given by 
\begin{align}\label{def-index}
C(\alpha,\beta)= \f{\alpha \kpa^2 \sh(2\kpa)}{\,2(e_{*}^2-\alpha)}\Big(1+\f{\kpa e_{*}}{\alpha \sh(2\kpa)}\Big)^2-\f{2\kpa \,k_2}{c \p_c \kpa}\, , 
\end{align}
such that   sufficiently small-amplitude periodic waves $(\zeta_{\varepsilon}, \varphi_{\varepsilon})$ are spectrally stable if $C > 0$ and unstable if $C < 0$.
\end{thm}

Consequently, the stability analysis of the periodic waves reduces to determining  when the index function $C=C(\alpha,\beta)$ changes sign.
Notably, this is the same index function that governs modulational (in)stability, as it determines whether $L^{\vep}$ exhibits unstable spectrum near the origin—meaning the rest of the spectrum lies along the imaginary axis. This latter fact follows from the observation that the Krein signatures associated with the eigenspaces of the crossing spectrum away from zero are identical, which in turn, stems from the presence of non-negligible surface tension (i.e., $\beta>0$). 

Our modulational stability analysis also applies to the case $(\alpha,\beta)\in \rm  II$,
yielding the following criterion for modulational stability across the full range of $(\alpha,\beta)$.
\begin{thm}\label{thm-fullmodulation}
Let $(\alpha,\beta)\in \rm I \cup II\cup III$ and let $\kappa>0$ be the wave number, satisfying \eqref{disp-relation}. 
Define 
\beq\label{def-index-real}
\tilde{C}(\alpha,\beta)=w_1''(0)C(\alpha,\beta),
\eeq
where $C(\alpha,\beta)$ is given by \eqref{def-index} and
$w_1''(0)$ is the second derivative of the function
\beq\label{def-w1''}
w_1(\xi)=\sqrt{\kpa(1+\xi)\tah(\kpa(1+\xi))(\alpha+\beta\kpa^2(1+\xi)^2)}.
\eeq
 There exists $\ep_0>0$ such that for any $|\vep|\leq \ep_0$, if $\tilde{C}(\alpha,\beta)>0$, the spectrum of $L_0^{\vep}$ in the vicinity of the origin lies on the imaginary axis while if  $\tilde{C}(\alpha,\beta)<0$, there exists unstable spectrum in the vicinity of the origin. 
\end{thm}

Let us remark that when $(\alpha,\beta)\in \rm I \cup III$, it holds that $w_1''(0)>0$ so that the criterion for  modulational stability in the above theorem is consistent with that of Theorem \ref{thm-1}.

\begin{figure}[htbp]\label{figure1}
    \centering
\includegraphics[width=0.6\linewidth,height=0.5\textheight, keepaspectratio]{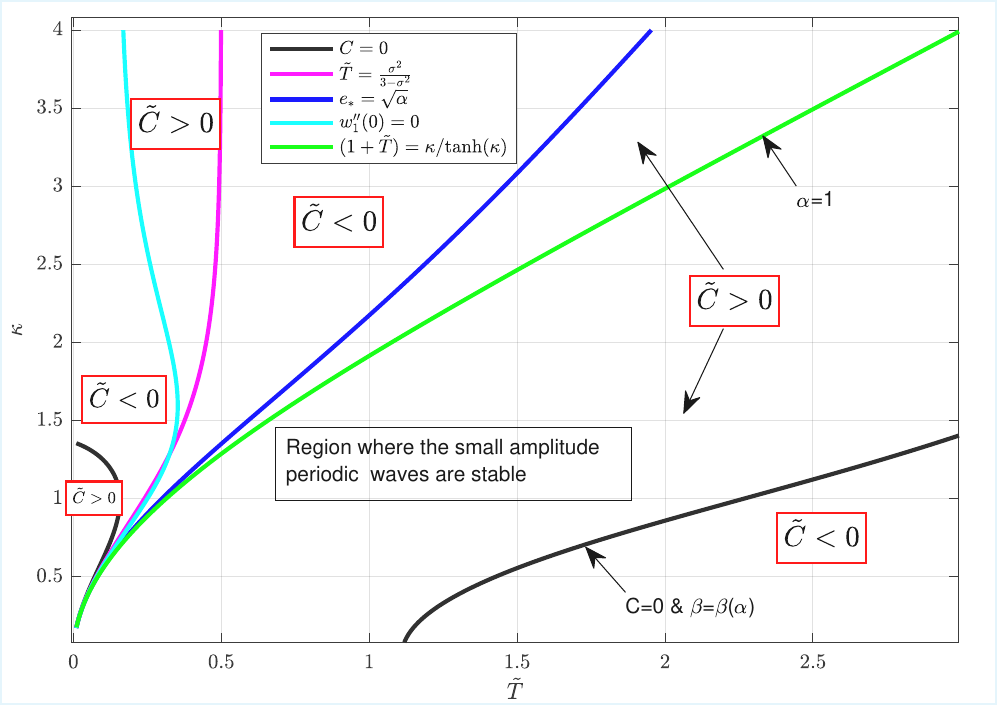} 
    \caption{The stability diagram in the  $(\tilde{T},\kappa)$ plane with $\tilde{T}=\frac{\beta}{\alpha}\kpa^2$.
    The black curves 
    are  such that $C=0$. The cyan curve is 
    where $w_1''(0)=0$. The purple curve is where $\sigma^2-\tilde{T}(3-\sigma^2)=0$ whereas the blue curve represents $e_{*}=\sqrt{\alpha}$.
     The green curve corresponds to $\alpha=1$. The region between the green curve and the right black curve indicates where the small-amplitude periodic waves are (globally) stable.}
\end{figure}

The modulational (in-)stability of periodic capillary-gravity water waves has been studied in \cite{D-R-packets, A-Segur}, where an index function was (formally) derived, using the Schrödinger approximation in the modulational regime, to determine whether such waves are modulationally stable or unstable. 
Through straightforward calculations (see Appendix \ref{appen-indices}), the index 
    $\tilde{C}(\alpha,\beta)$ is found to be proportional to the index function $\lambda(\alpha,\beta)\nu(\alpha,\beta)$ obtained in 
   (2.24b, 2.24h) of \cite{A-Segur} (see also (2.17) of \cite{D-R-packets}): 
    \begin{align*}
\tilde{C}(\alpha,\beta)=
\sqrt{\f{\kpa}{4\alpha}}\kpa^3
\ch^2(\kpa) \, \lambda(\alpha,\beta)\nu(\alpha,\beta). 
\end{align*}
This correspondence allows us to rigorously justify the formal predictions in \cite{A-Segur} and thus provides a rigorous proof of the modulational (in)stability of periodic capillary-gravity water waves. Related results have also been obtained by Hur and Yang \cite{Hur-Yang-capillary}, 
who derive stability index functions for both modulational and non-modulational (nearest two crossings to the origin) stability using an Evans function approach, which yields explicit modulational index functions and explicit but complicated non-modulational ones,  which are studied numerically.
However, on the one hand, when applied to water waves, the Evans function approach tends to be much more computationally intensive. On the other hand, the Evans function approach reduces the analysis of stability near high-frequency crossings to determining the sign of index functions. This contrasts with Kato's perturbation theory, where stability can be characterized by the non-vanishing of index functions, which facilitates the spectral analysis near high-frequency crossings.
We therefore approach the problem using Kato's perturbation theory, which not only yields a simple  criterion for modulational (in-)stability but also provides a more precise description of the spectrum near the origin.  Moreover, we rigorously prove that no instability arises from \textit{any} spectral crossings away from the origin when $$(\alpha,\beta)\in \rm I\cup III=\{(\alpha,\beta): \alpha\in(0,1),\beta>0 \text{\, or\,} \alpha=1,\beta\in(0, 1/3)\},$$
thereby establishing global spectral stability beyond the modulational regime. 
 To determine whether the wave is stable in the regime $(\alpha,\beta)\in\rm I\cup III$ we need to further study the sign of the index function 
$C(\alpha,\beta)$ defined in \eqref{def-index}.
\subsection{On the condition $C(\alpha, \beta)>0$ when $(\alpha,\beta)\in \rm I \cup III$ }

Since $\alpha,\beta$ and $\kappa$ are related through the dispersion relation \eqref{disp-relation}, the index function can also be regarded as a function of $\alpha$ and $\kappa$. Denote 
\begin{align*}
    \underline{C}(\alpha, \kappa)=\frac{8(e_{*}^2-\alpha)}{\kappa^3\ch^2(\kappa)}C\big(\alpha, \kappa^{-2}(\kappa/\textrm{th}(\kappa)-\alpha) \big).
\end{align*}
By \eqref{rewrite-nu} and \eqref{relation-Cnu}, it holds that 
\begin{align*}
 \underline{C}(\alpha, \kappa)={8\alpha} \f{\tah(\kpa)}{\kpa}
 \Big(1+\f{\kpa e_{*}}{\alpha \sh(2\kpa)}\Big)^2+(e_{*}^2-\alpha)\chi
\end{align*}
with $\chi$ defined in \eqref{def-chi}\footnote{$\chi$ is viewed as a function of $\kappa$ and $\alpha$  through the relation 
$\tilde{T}=\frac{\kappa}{\alpha\tah(\kappa)}-1$.}.
Note that $\underline{C}(\alpha, \kappa)$ is smooth for $\alpha\in(0,1], \kappa\in \mR_{+},$ and 
that its sign 
coincides with that of $C(\alpha,\beta)$ when  $(\alpha,\beta)\in \rm I \cup III.$ Direct computations show that 
\begin{align*}
  \lim_{\kappa \rightarrow 0^{+}}  \underline{C}(\alpha, 0)=-4\alpha^2+23\alpha-10,
\end{align*}
which is positive if $\alpha\in(\alpha_1,1]$ where 
$\alpha_1=(23-3\sqrt{41})/8\approx 0.47383$.\footnote{It also corresponds to the intersection point of the black curve in Figure \ref{figure1} with the axis $\kappa=0$. One can  express $C(\alpha,\beta)$ in terms of variables $\kappa$ and $\tilde{T}=\frac{\beta}{\alpha}\kappa^2$ and find that the intersection point corresponds to $\kappa=0$,  $\tilde{T}=\frac{3+3\sqrt{41}}{20}$.}
This, together with the smoothness of $\underline{C}(\alpha, \kappa)$ and the facts that  $(\partial_{\beta}\kappa)(\alpha,\beta)=-
\f{\kappa^2\tah(\kappa)}{2(e_{*}-1)}<0$ and $\lim_{\beta \to \infty} \kappa(\alpha, \beta)=0$ if $0<\alpha<1$, $\lim_{\beta \to 1/3} \kappa(1, \beta)=0$,
implies the following result.
\begin{thm}
  For any $\alpha\in(\alpha_1,1],$ there is $\beta_0(\alpha)$ such that for any $\beta>\beta_0(\alpha)$ ($\beta_0(1)<\beta<\f13$ if $\alpha=1$), it holds that $C(\alpha,\beta)>0.$ As a result, the small amplitude periodic waves are spectrally stable in this parameter regime.
\end{thm}

We thus see that there is an unbounded region in the domain  $(\alpha,\beta)\in \mathrm{I}\cup \mathrm{III}$ where $C(\alpha,\beta)>0$, that is, where the small amplitude periodic waves are spectrally stable.
Numerical simulations indicate that there also exists a region for $\alpha<\alpha_1$ where $C(\alpha,\beta)>0.$ Specifically, it is found that
there exists a constant $\alpha_0\approx 0.39495,$
such that for any $\alpha\in (\alpha_0,\alpha_1)$ the function
$\underline{C}(\alpha,\cdot)$ has two roots denoted by $\kappa_0(\alpha),
\kappa_1(\alpha)$ and satisfies $\underline{C}(\alpha,\cdot)>0$ for any $\kappa\in (\kappa_1(\alpha),
\kappa_0(\alpha)),$ $\underline{C}(\alpha,\cdot)\leq 0$ if $\kappa\notin (\kappa_1(\alpha),
\kappa_0(\alpha)).$ Consequently, for  $\alpha\in(\alpha_0,\alpha_1), $ one has
$C(\alpha,\beta)>0$ when $\beta_0(\alpha)<\beta< \beta_1(\alpha)$ where 
$\beta_j(\alpha)=\beta(\alpha,\kappa_j(\alpha)), j=0,1.$
To summarize, the rigorous analysis above, together with the numerical simulations, suggests that
\begin{align*}
   C(\alpha,\beta) >0 \,\,
   \, 
   \mathrm{ for }\,\,(\alpha, \beta)\in I_1  \subset \mathrm{I}\cup \mathrm{III},
 \end{align*}
where 
$$I_1\coloneqq \{(\alpha,\beta)\in (0,1)\times\mR : \beta>\beta_0(\alpha)  \text{\,when\,}\alpha \in [\alpha_1,1)\text{\,and\,} \beta_0(\alpha)<\beta<\beta_1(\alpha) \text{\,when\,}\alpha \in (\alpha_0,\alpha_1)\}.$$ This region is illustrated in Figure~\ref{fig:figure2}.


The above findings highlight the impact of capillarity on the stability of periodic water waves. Indeed, numerous studies have investigated the stability of periodic waves in the context of gravity water waves (see \cite{BM-BFins,berti-highfreq,BMP-ARMA,Hur-Yang-ARMA,CDT-highfreq}) where surface tension is neglected. 
 Based on these investigations, it is believed that for 
 any depth $h,$ the
$2\pi$-periodic gravity Stokes waves are spectrally unstable. In other words, for the normalized system \eqref{sur-mov-uni} with $\beta=0$, these results indicate that for all  wave numbers $\kappa,$ the corresponding 
$2\pi/\kappa$-periodic waves are unstable.
However, the results above suggest that when surface tension is taken into account and is not too small, the capillary–gravity periodic waves become spectrally stable for parameter values
$(\alpha,\beta)$ lying within the region
$I_1.$  Note that in this regime, the parameter $\alpha$ does not exceed $1,$ and the periodic waves do not exist if $\beta=0.$ Therefore, one of the effects of capillarity is the emergence of a family of periodic waves when $\alpha\leq 1$ and that these waves are spectrally stable provided that $\alpha\geq \alpha_1$ and $\beta$ is not so small. 
Moreover, the region $I_1$ corresponds to the area between the green and black curves in Figure 1. Based on the discussion in Appendix \ref{app-regionII-highcrossing} and the analysis in Section 4 of \cite{Hur-Yang-capillary}, there also exists a region of positive measure between the blue and green curves in Figure 1 where the waves are spectrally stable.\footnote{We thank Zhao Yang for bringing this point to our attention.}

\begin{figure}[htbp]
    \centering
\includegraphics[width=0.6\linewidth,height=0.5\textheight, keepaspectratio]{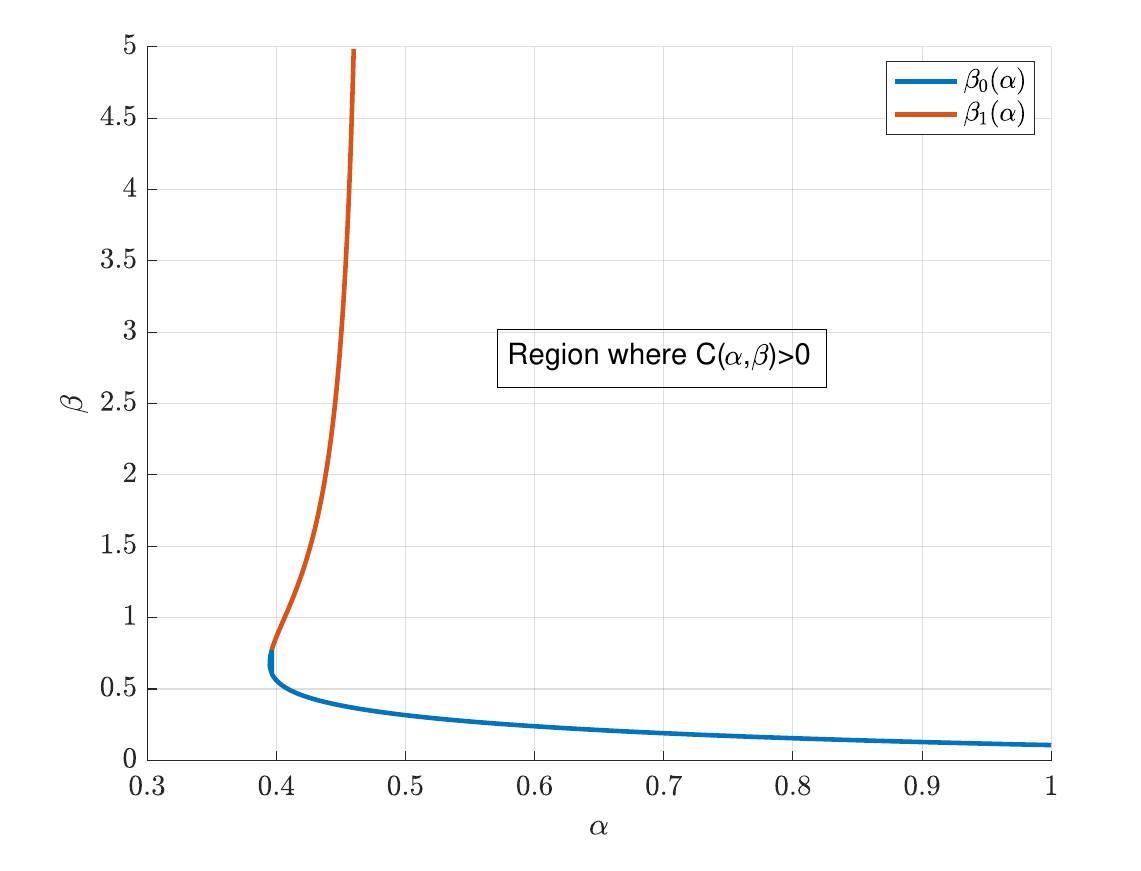}  
    \caption{The curve $\beta=\beta_0(\alpha)$ (in blue)  and 
    the curve $\beta=\beta_1(\alpha)$ (in red) are such that $C(\beta,\alpha(\beta))=0$. The red curve diverges to $+\infty$  as $\alpha$ approaches  $\alpha_1$.
    The region enclosed  by these two curves and lying above
    $\beta_0(\alpha)$ corresponds to the domain where $C(\alpha,\beta)>0,$ which is defined as $I_1$.}
    \label{fig:figure2}
\end{figure}

\subsection{Background on the stability of Stokes waves}\label{background-stokes}
The stability of Stokes waves—periodic traveling wave solutions of the gravity water wave equations—has been an important theme in the study of water waves.
A key phenomenon in this context is modulational instability, where long-wavelength perturbations can cause the wave train to destabilize. This effect, commonly referred to as the Benjamin–-Feir instability, was first observed by Benjamin and Feir through experiments and formal calculations in the 1960s \cite{benjamin1967disintegration}. 
A rigorous mathematical foundation was laid by Bridges and Mielke \cite{BM-BFins}, who analyzed the modulational stability of Stokes waves in finite depth using a Hamiltonian spatial dynamics framework and center-manifold reduction.
Recently, Nguyen and Strauss \cite{Nguyen-Strauss}
proved modulational instability of deep-water Stokes waves, a situation where the 
spatial center-manifold reduction cannot be employed.
In \cite{Hur-Yang-ARMA}, Hur and Yang re-examined the spectral instability of finite-depth Stokes waves using the Evans function approach. They identified a criterion for modulational instability and also detected instability at the next spectral crossing away from the origin. In a series of paper \cite{BMP-ARMA,BMV-cmp,BMP-deep,berti-highfreq}, Berti et al. investigated the stability of Stokes waves based on 
Bloch--Floquet theory by employing the 
analytic perturbation theory of Kato \cite{Kato}. We also refer to \cite{CDT-highfreq} for a formal perturbative analysis and numerical simulations concerning the high-frequency instabilities of Stokes waves. 
Recently, the stability of one-dimensional Stokes waves as solutions to the three-dimensional water-wave system—known as transverse stability—has attracted considerable research interest. See \cite{CNS-deep, CNS-finite} for studies on purely transverse perturbations (i.e., co-periodic perturbations along the longitudinal direction), and \cite{JRSY-cmp} for analyses involving both longitudinal and transverse perturbations.

All of the above results concern the \textit{spectral} (in-)stability of Stokes waves. To our knowledge, the only work so far addressing \textit{nonlinear} instability is due to Chen and Su \cite{chen-su}. By rigorously justifying the NLS approximation and leveraging the modulational instability of periodic NLS waves, they established the nonlinear instability of small amplitude Stokes waves under subharmonic perturbations (i.e., perturbations that are periodic with a period equal to an integer multiple of the Stokes-wave period), without relying directly on the spectral instability of the underlying Stokes waves. 

The study of 
periodic solutions of the capillary-gravity water-wave system is also interesting but has received relatively limited mathematical attention. Apart from the formal analyses presented in \cite{D-R-packets, A-Segur}  and the numerical investigations \cite{DT-numerics,TDW-wiltonripples},
the only existing mathematical proof to date is due to Hur and Yang \cite{Hur-Yang-capillary}, who employed the Evans function approach in the spirit of their earlier work on gravity waves. In this work, we aim to re-examine the problem by using  Kato's perturbation theory, which not only enables us to find 
the criterion for the modulational stability, 
but also enables us to rule out any potential instability near other spectral crossings away from the origin when $(\alpha,\beta)\in \rm I\cup III$. As a result, we obtain a global spectral stability result.

\subsection{Strategy of the proof}
In a broad sense, our work falls within the scope of stability studies of periodic waves for Hamiltonian systems, for which we can refer to 
\cite{Miguel-KdV,ARS-KdV,GL-NLS,Jonhson-fkdV,BHV-Floquettheory,NRS-EP,Corentin-Miguel}.
Since the operator \eqref{def-Lep} has periodic coefficients, one can apply the Bloch--Floquet transform to reduce the spectral analysis of 
 $L^{\vep}$ on the $L^2(\mR)-$based space $Y$ to the spectral study of a family of operators $L_{\xi}^{\vep}$ parametrized by the Floquet parameter $\xi\in (-\f12,\f12]$
on the $L^2(\mathbb{T}_{2\pi})$ $(\mathbb{T}_{2\pi}\coloneqq\mathbb{R}/2\pi\mathbb{Z})$
based
space $Y_{\per}$. Since each $L_{\xi}^{\vep}$ has compact resolvent, the spectrum of $L_{\xi}^{\vep}$ is composed of a family of discrete eigenvalues with finite multiplicities. Moreover, when $\vep=0$, the spectrum of each $L_{\xi}^0$ can be computed via the Fourier transform and lies entirely on the imaginary axis.
Inherited from the Hamiltonian structure \eqref{Hamiltonianstructure}, the spectrum of each $L_{\xi}^{\vep}$ remains symmetric with respect to the imaginary axis. As a result, any potential unstable spectrum of $L^{\vep}$ can only bifurcate from the multiple eigenvalues of $L_{\xi}^{0}$ (which we will refer to as the crossing spectrum) for certain values of the Floquet parameter $\xi$. It is shown in Lemma \ref{lem-crossing} that the multiple eigenvalues consist of
\begin{itemize}
    \item the zero  eigenvalue of $L_0^0$ with multiplicity four, 
    \item countably many other spectral crossings
away from the origin with multiplicity two (arising from the same branch of the spectral curve
when $(\alpha,\beta)\in \rm I\cup III$).
\end{itemize}
The task is then to study how the spectrum bifurcates from these crossing points as $\vep$
becomes non-zero but remains small. The general strategy is to construct a finite-dimensional representation matrix whose eigenvalues coincides with the spectrum of $L^{\vep}$ near these crossing points. This is achieved by constructing appropriate bases and dual bases for the associated eigenspaces. The latter is carried out by employing  Kato's perturbation theory, starting from the bases at $\vep=0$ or $\xi=0$, which can be obtained via the 
Fourier transform or by exploiting the structure of the background waves.

Let us first explain how to examine the spectral behavior near the crossing points away from zero, where the corresponding eigenspaces are two-dimensional. Assume that there is crossing spectrum $\lambda_0$ for the Floquet parameter $\xi_0$. When 
$|(\vep,\xi-\xi_0)|$ is small, 
we can construct the basis $\{q_1^{\vep}(\xi,\cdot), q_2^{\vep}(\xi,\cdot)\}$ by extending the basis $\{q_1^{0}(\xi,\cdot), q_2^{0}(\xi,\cdot)\}$ that is obtained by Fourier transform. Thanks to the Hamiltonian symmetry 
$(L_{\xi}^{\vep})^{*}=-J^{-1}L_{\xi}^{\vep}J$, 
one can construct the dual basis $\{\tilde{q}_1^{\vep}(\xi,\cdot), \tilde{q}_2^{\vep}(\xi,\cdot)\}$ from linear combinations of $\{\i Jq_1^{\vep}(\xi,\cdot), \i J q_2^{\vep}(\xi,\cdot)\}$. It results from the fact that 
$q_1^{0}(\xi,\cdot), q_2^{0}(\xi,\cdot)$ always involves different Fourier modes that the dual basis can be chosen as
\begin{align*}
\tilde{q}_j^{\vep}(\xi,\cdot)=\i \, \omega_{j}^{-1}\,  
J {q}_j^{\vep}(\xi,\cdot),
\qquad \mathrm{ with }\,\,\, \omega_{j}(\xi)=\langle \i \,  J {q}_j^{0}(\xi,\cdot), {q}_j^{0}(\xi,\cdot)\rangle, \,\,  j=1,2 .
\end{align*}
One important property in the current situation is that the associated Krein signatures—that is, the signs of $\alpha_1(\xi),\alpha_2(\xi)-$are always the same.  This fact, combined with the relation  $L_{\xi}^{\vep}=JA_{\xi}^{\vep}$, where $A_{\xi}^{\vep}$
is a symmetric operator, allows us to deduce that the representation matrix takes the form
\begin{align}\label{repmatrix-high}
    D_{\xi}^{\vep}=\,\left(\langle \tilde{q}_j^{\eps}(\xi,\cdot), L_{\xi}^{\vep}q_{\ell}^{\eps}(\xi,\cdot)\rangle\right)_{1\leq j,\ell\leq 2}=\i P_{\xi} \tilde{D}_{\xi}^{\vep} P_{\xi}^{-1}
\end{align}
 where $P_{\xi}$ is an invertible matrix and 
$\tilde{D}_{\xi}^{\vep} $ is a real symmetric matrix. 
Here and throughout the paper, $\langle \cdot, \cdot \rangle$ denotes the canonical inner product on $L^2(\mathbb{T}_{2\pi})$, which is conjugate-linear in its first argument.
It follows from \eqref{repmatrix-high} that the eigenvalues of $D_{\xi}^{\vep}$, hence the spectrum of $L_{\xi}^{\vep}$ near $\lambda_0$, must stay on the imaginary axis.  Consequently, no instability arises from any of the spectral crossings away from zero. This aligns with the classical understanding in Hamiltonian PDE theory that the collision of purely imaginary eigenvalues with identical Krein signatures typically does not result in spectral instability, see \cite{KKS-Krein}. 
Let us remark that such a procedure for rigorously proving stability in this setting was also observed in an earlier work by the first author and collaborators, in the context of periodic waves for the Euler--Poisson system \cite{NRS-EP}.  In that case, however, the Krein signatures are always opposite, leading instead to a conclusion of instability.

We are now in a position to outline the strategy for analyzing the spectrum of $L^{\vep}$ near zero, which corresponds to the modulational stability of periodic waves. Since zero is a spectral crossing of 
 $L_0^0$ with multiplicity four, the resulting matrix will be  $4\times 4$.
 To construct the four basis vectors, one can again start with the eigenfunctions of $L_{\xi}^0$  
 and extend them to the case $\vep\neq0$. However, this approach requires computing the expansions of 
 these basis functions in terms of the amplitude $\vep$ (at least to the second order), making it computationally demanding. Therefore, we instead begin with the basis of $L_0^{\vep}$, whose spectrum near the origin consists solely of $\{0\}$. The associated eigenfunctions can be obtained using the profile equations and the Lyapunov–-Schmidt reduction (see Proposition \ref{prop-kernel}), and are then extended to form a basis for $L_{\xi}^{\vep}$
with small nonzero $\xi$ by  Kato's perturbation theory. The dual bases are then found from linear combinations of $\{\i Jq_1^{\vep}(\xi,\cdot), \i J q_2^{\vep}(\xi,\cdot), \i  Jq_3^{\vep}(\xi,\cdot), \i Jq_4^{\vep}(\xi,\cdot)\}$.
With these in hand, we find the rather explicit expression of the representation matrix
\begin{align*}
    \rD_{\xi}^{\vep}=\,\left(\langle \tilde{q}_j^{\vep}(\xi,\cdot), L_{\xi}^{\vep}q_{\ell}^{\vep}(\xi,\cdot)\rangle\right)_{1\leq j,\ell\leq 4}
    \, .
\end{align*}
To compute the eigenvalues of this matrix, we further exploit the symmetries and parity properties of the wave profiles to derive useful information about its entries. The first expression we obtain is
\begin{align}\label{D-firstexp}
    \rD_{\xi}^{\vep}=\bpm  * \,\i \kpa \xi\, & * & *\,\i \kpa \xi\,\vep  & * \, 
    \vep \\
    * \, (\kpa\xi)^2& * \,\i \kpa \xi & * (\kpa\xi)^2\vep& * \,\i \kpa \xi\,\vep \\ 
 *\,\i \kpa \xi\,\vep & *\,\vep & *  \,\i \kpa \xi & * \\
 * (\kpa\xi)^2\vep & * \,\i \kpa \xi\,\vep & * (\kpa\xi)^2  & *  \,\i \kpa \xi
    \epm 
\end{align}
where $*$ denotes certain \textbf{real} numbers depending smoothly on $\xi$, $\vep$. Motivated by this expression, we anticipate that the eigenvalues of 
$\rD_{\xi}^{\vep}$ are proportional  to $\i\kpa\xi$. This observation prompts the  transformation 
    \begin{align}\label{def-newmatrix-1}
   \mathrm {M}_{\xi}^{\vep}=\frac{1}{\i \,\kpa \, \xi}\rP_{\xi}  \rD_{\xi}^{\vep} \rP_{\xi}^{-1}
   \coloneqq \bpm  \rM_{11} & \vep \rM_{12} \\
    \vep \rM_{21} & \rM_{22}\epm (\vep,\xi), 
\end{align}
where $\rP_{\xi}=\diag\, (\i\,\kpa \xi, 1, \i\, \kpa\xi, 1)$ and $\rM_{j\ell} \,(1\leq j,\ell\leq 2)$ are $2\times 2$ matrices with real coefficients which are smooth in $\xi$ and $\vep$. As is readily seen,  $\sigma(\rD_{\xi}^{\vep})=\i \, \kpa\xi \,\sigma(\rM_{\xi}^{\vep})$, and the problem thus reduces to finding the eigenvalues of the real matrix 
$\rM_{\xi}^{\vep}$, which represents the crucial part of the analysis. Since the off-diagonal $2\times 2$ blocks are of order $\cO(\vep)$, one can expect the eigenvalues of $\rM_{\xi}^{\vep}$ to be close to those of $\rM_{11}$ and $\rM_{22}$. A direct calculation shows that the two eigenvalues of  $\rM_{22}$ remain real, provided $(\vep,\xi)$ is sufficiently small. The eigenvalues of $\rM_{11}$, 
 on the other hand, possess nonzero imaginary parts when $|\xi|$ is much smaller than $\vep$ and their size is of order $\cO(\vep)$. Consequently, the presence of non-vanishing off-diagonal terms can potentially render all eigenvalues of $\rM_{\xi}^{\vep}$ real. 
The idea to get the full description of the eigenvalues of  $\rM_{\xi}^{\vep}$ is to introduce some transformation to make one of the off-diagonal 
$2\times 2$ block $\vep\rM_{12}$ or $\vep\rM_{21}$ vanish, whose existence is ensured 
as long as the eigenvalues of $\rM_{11}$ and $\rM_{22}$ are well-separated (
i.e. $e_{*}\neq \sqrt{\alpha}$), which is the case when $(\alpha,\beta)\in \rm I \cup III$; 
see Proposition \ref{prop-separation} and 
Remark \eqref{rmk-separation-eigen}. More precisely, we first perform a transformation to make the block $\rM_{22}$ diagonal and then find another transformation to make $\vep \rM_{12}$ vanish, which is proven to exist by applying the inverse function theorem. Finally, we arrive at the transformed matrix 
\begin{align*}
\bpm \rM_{11}-\vep^2 \rQ \rG^{-1}\rM_{21} & 0\\ \vep \rG^{-1}\rM_{21} & \bpm X_2^{+} & \\ &  X_2^{-} \epm +\vep^2 \rG^{-1}\rM_{21}\rQ \epm 
\end{align*}
where $X_2^{+}, X_2^{-}$ are two positive real eigenvalues of $\rM_{22}$, and $G$ and $Q$ are  $2\times 2$ real matrices. It is now clear that the eigenvalues of the above matrix consist of the real values $X_{2}^{\pm}+\cO(\vep^2)$
and the two eigenvalues of the matrix $\rM_{11}-\vep^2 \rQ \rG^{-1}\rM_{21}$. These latter two are purely real when the quantity $w_1''(0)C(\alpha,\beta)$, defined in \eqref{criteria}, is positive and they acquire nonzero imaginary parts 
of order $\cO(|\vep\xi|)$ when $|\xi|\ll |\vep|$.

Our proof of modulational stability differs  from \cite{BMP-ARMA}, where Kato's perturbation method is 
employed to study the stability of gravity Stokes waves.
In \cite{BMP-ARMA}, 
 the construction of the basis begins with the operator $L_0^0$ which is then extended to a basis for $L_{\xi}^{\vep}$ for small nonzero values of $\xi$ and $\vep$. Since the Kato extension operator preserves the symplectic inner product, the dual basis can be conveniently constructed via  
\beq\label{intro-dual-expample}
\tilde{q}_j^{\vep}(\xi,\cdot)=\i Jq_j^{\vep}
(\xi,\cdot),\,\qquad  j=1,2,3,4.
\eeq
The subsequent spectral analysis proceeds by performing three steps of Bloch diagonalization on the resulting representation matrix. In contrast, our method directly exploits the structure of the background waves to construct a basis for $L_0^{\vep}$ which we then extend to a basis for $L_{\xi}^{\vep}$ for small $\xi\neq 0$. This approach simplifies the computation by avoiding the need to explicitly compute the Kato extension in terms of the amplitude parameter. The trade-off, however, is that our initial basis $\{q_j^{\vep}(0,\cdot) \} $ does not satisfy perfect orthogonality with respect to the symplectic inner product, and therefore, the dual basis cannot be expressed simply as in \eqref{intro-dual-expample}.
Instead, it must be constructed as a linear combination of $\{Jq_j^{\vep}(\xi,\cdot)\}_{j=1,2,3,4}$. Fortunately, we find that the dual basis $\tilde{q}_j(\xi,\cdot)$ is a small perturbation of $Jq_j^{\vep}$ (see \eqref{def-dualbasis}, \eqref{expan-aij}),
which suffices for our purposes. With the bases and dual bases in hand, we derive the representation matrix as in expression \eqref{D-firstexp}. By observing that the eigenvalues of the diagonal blocks $\rM_{11}$ and $\rM_{22}$ are uniformly separated, we can find directly a similarity transformation to make one of the  off-diagonal blocks vanish and thus facilitate the subsequent spectral study. It is also worth noting that the spectral curve near the origin can be readily identified from the explicit expressions for the eigenvalues (see \eqref{spectrumrelation} and \eqref{realeigen-up}) of the representation matrix\footnote{After the completion of this work, another preprint\cite{Hsiao-Maspero} appeared, deriving the modulational (in)stability and the spectral curve near the origin using an approach similar in spirit to \cite{BMP-ARMA}. }, see Remark \ref{rmk-specurve}.
\\

\noindent\textbf{Notation.} 
For simplicity of notation, we use the following conventions throughout the paper: 
\begin{equation*}
    \sh (\cdot)\coloneqq \sinh(\cdot), \qquad  \ch (\cdot)\coloneqq \cosh(\cdot), \qquad \tah(\cdot)\coloneqq \tanh(\cdot).
\end{equation*}

\section{Spectral preliminaries}

\subsection{Bloch transform}
To analyze the spectrum of operators  with periodic coefficient, it is expedient to introduce Bloch symbols, associated with the (Floquet-)Bloch transform. 

In a similar spirit to the inverse Fourier transform, the inverse Bloch transform provides a way to reconstruct a function from its Bloch components. For any $f\in L^1(\mathbb{R})$ such that $\hat f\in L^1(\mathbb{R})$, it holds that
\begin{align}\label{def-bloch}
\displaystyle
    f(x)=\int_{-1/2}^{1/2} e^{ix\xi} \check{f}(\xi,x)\,\d \xi\, ,
\end{align}
where the direct Bloch transform of $f$, $\cB(f)=\check{f}$, is defined as
\begin{align}\label{def-Bloch}
\cB(f)(\xi,x)=\check{f}(\xi,x)\coloneqq\sum_{j\in \mathbb{Z}} e^{\i\, j\,x } \hat{f}(\xi+j)
.
\end{align}
Note that by design, for each $\xi$, $\check{f}(\xi,\cdot)$ is periodic of period $2\pi$, so that $f$ is written as an integral of functions $g_\xi\,=\,e^{i\xi\,\cdot} \check{f}(\xi,\cdot)$ satisfying $g_\xi(x+2\pi)=e^{i\xi}\,g_\xi(x)$. Such functions are known as Bloch waves, the number $e^{i\xi}$ being a Floquet multiplier, and $\xi$ is classically called Floquet exponent. 

With the help of the inverse Bloch transform \eqref{def-bloch}, one can turn an operator with periodic coefficients acting on a function over $\mR$ into a family of operators acting on periodic functions. Explicitly, it holds that, 
\begin{align}
L^{\vep}(u)(x)=\int_{-1/2}^{1/2} e^{ix\xi} L_{\xi}^{\vep}(\check{u}(\xi, \cdot))(x) \d \xi
\end{align}
where each $L_\xi^\vep\coloneqq L^{\vep}(\p_x+i\xi)$ acts on $Y_{\per}=H^{1}(
\mathbb{T}_{2\pi})\times H^{\f12}(\mathbb{T}_{2\pi})$ with domain $H^{\f52}(\mathbb{T}_{2\pi})\times H^{2}(\mathbb{T}_{2\pi})$. The main gain when replacing the direct analysis of $L^\vep$ with those of $L_\xi^\vep$ is that each $L_\xi^\vep$ has compact resolvent and thus their spectrum is  composed of eigenvalues of finite multiplicity. A key related observation is that
\begin{align*}   \sigma(L^{\vep}|_Y)=\bigcup_{\xi\in[-1/2,1/2]}\sigma(L_{\xi}^{\vep}|_{Y_{\per}}).
\end{align*}
It is also important to notice that the Hamiltonian symmetry is inherited by $L_{\xi}^{\vep}$,
\begin{align}\label{HS-Lxi}
    (L_{\xi}^{\vep})^{*}=-J L_{\xi}^{\vep} J^{-1},
\end{align}
and thus its spectrum is again symmetric with respect to the imaginary axis. 


\subsection{
Study of the constant operator $L_{\xi}^0$}
By definition,
\begin{align*}
   L_{\xi}^{\vep}= \left( \begin{array}{cc}
   k_{\vep}(\p_x+i\xi)(d_{\vep}\cdot)  &  G_{\xi}[\zeta_{\vep}]\cdot   \\[5pt]
 P_{\xi}[\zeta_{\vep}]-  w_{\vep}& k_{\vep}d_{\vep}(\p_x+i\xi)
\end{array}\right),
\end{align*}
where 
\begin{align}\label{def-GP}
    G_{\xi}[\zeta_{\vep}]=e^{\i \,x \,\xi}G[\zeta_{\vep}]e^{-\i\, x\,\xi}; \quad P_{\xi}[\zeta_{\vep}]=e^{\i \,x\,\xi}P[\zeta_{\vep}]e^{-\i\, x\,\xi}
\end{align}
are operators acting on periodic functions.
Specifically, when $\vep=0$,
\begin{align*}
     L_{\xi}^{0}= \left( \begin{array}{cc}
   \kappa(\p_x+i\xi) &  |\kappa(\p_x+i\xi)| \tah (|\kappa(\p_x+i\xi)| ) \\[5pt]
 \beta\kappa^2(\p_x+i\xi)^2- \alpha& \kappa (\p_x+i\xi)
\end{array}\right)\, .
\end{align*}
Being an operator with constant coefficients, its spectrum can be found explicitly by using the Fourier transform:
\beq\label{defspec-0}
\begin{aligned}
   &\sigma(L_{\xi}^0)= \bigcup_{j\in \mathbb{Z}} \lambda_{j}^{\pm}(\xi); \qquad \lambda_{j}^{\pm}(\xi)=\i \big( \kappa (j+\xi)\pm w_j(\xi) \big),\\
   &\text{where } w_j(\xi)\coloneqq \sqrt{\kappa(j+\xi) \tah (\kappa (j+\xi)) \big(\alpha+\beta\kappa^2 (j+\xi)^2\big) }.
\end{aligned}
\eeq
The spectrum of $L_{\xi}^0$ is thus contained in the imaginary axis. By the Hamiltonian symmetry \eqref{HS-Lxi},
the simple eigenvalues of $L_{\xi}^{0}$ remain on the imaginary axis for small $\vep>0$ and possible unstable spectrum bifurcates only from  multiple eigenvalues of $L_{\xi}^0$ when $\vep>0$ is sufficiently small. Therefore, we need to find all the crossing spectrum of $L_{\xi}^0$ with $\xi\in(-\f12,\f12]$.

\begin{lem}[Crossing of $L_{\xi}^0$] \label{lem-crossing}
Let 
$(\alpha,\beta)\in \rm I \cup III$, where
$\rm I, III$ are defined in \eqref{region-I-III}.
The following facts hold:

 \textnormal{(1) [Crossing at the origin].}
The origin is an eigenvalue of $L_0^0$ of algebraic multiplicity four: 
\begin{align}\label{crossing-0}
    \lambda_{-1}^{+}(0)=\lambda_{1}^{-}(0)= \lambda_{0}^{+}(0)= \lambda_{0}^{-}(0)=0.
\end{align}

  \textnormal{(2)  [Crossing away from the origin].} There are 
countably many other spectral crossings of multiplicity two, which occur when
\begin{align}\label{othercrossing}
   \lambda_{j_{\ell}}^{+}(\xi_{\ell})=\lambda_{j'_{\ell}}^{+}(\xi_{\ell})\coloneqq i\sigma_{\ell}; \qquad \lambda_{-j_{\ell}}^{-}(-\xi_{\ell})=\lambda_{-j_{\ell}'}^{-}(-\xi_{\ell})\coloneqq -i\sigma_{\ell}  \qquad
\end{align}
where $ \xi_{\ell}\in (-\f12,\f12],\, j_{\ell}\geq 0, \, j'_{\ell}<0,\,\, \ell=j_{\ell}-j_{\ell}'\geq 2$. Moreover, if $j_{\ell}=0$, it holds that $\xi_{\ell}>0$.
\end{lem}
\begin{proof}
Let
$\sigma_{\pm}(k)=k\pm\sqrt{k\tah(k)(\alpha+\beta k^2)}$. 
It can be shown that, when $(\alpha,\beta)\in \rm I\cup III$,   the function $\sigma_{-}(k)$ 
has only one critical point $\kappa_c \in (0,\kappa)$, and is thus strictly increasing on 
$(-\infty,\kappa_c]$ and strictly decreasing on  
$[\kappa_c,+\infty)$,  with $\sigma_{-}'(k)<0$.
Moreover, $\sigma_{-}(k)$ is strictly concave with $\sigma_{-}''(k)<0$ on $[\kappa_c,+\infty)$. We refer for instance to \cite[Lemma 3.5]{Hur-Yang-capillary} for a proof  (see also \cite[p.~865, (205)--(207)]{Dull-ARMA}). 
It follows from the identity $\lim_{k\rightarrow \pm \infty}\sigma_{-}(k)=-\infty$ that $\sigma_{-}(k)<0$ for $k\in(-\infty,0)\cup (\kappa,+\infty)$ and 
$\sigma_{-}(k)\geq 0$ for $k\in [0,\kappa]$.
One gets the corresponding property for $\sigma_{+}(k)$ from the relation 
$\sigma_{+}(k)=-\sigma_{-}(-k)$. 
This proves the crossing \eqref{crossing-0} and that $\lambda^\pm_j(\xi)\ne 0$ for all other $j$ and $\xi$. Note that $j=\pm 1$ give rise to one eigenvector each, whereas for $\xi=0$ and $j=0$ (i.e., constant functions) 
the action of $L_0^0$ is given by
  $$  \left( \begin{array}{cc}
   0 &  0 \\[5pt]
 -\alpha & 0
\end{array}\right),$$
resulting in the eigenvector $(0,1)^t$ and the generalized eigenvector $(\frac1{\alpha}, 0)^t$.



 To study the crossings away from the origin, note that for any $k\in( 0, \kpa)$, it holds that
$\sigma_{+}(k)>\sigma_{-}(k), 
$ and $\sigma_{-}(k)\sigma_{+}(k-\kpa)<0$
and thus there is no crossing $\lambda_{j}^{+}(\xi)=\lambda_{j'}^{-}(\xi)$ 
for $0\neq \xi\in (-1/2,1/2]$ and $|j-j'|\leq 1$. 
From the concavity 
of $\sigma_{-}(k)$ on $[\kappa_c,+\infty)$, we also derive that there is no $k_1, k_2$ such that 
$k_1\in(0,\kpa), k_1-k_2\geq \ell \kpa\, (\ell \geq 2)$ and $\sigma_{-}(k_1)=\sigma_{+}(k_2)$. This, together  with the fact $\sigma_{+}(k)=-\sigma_{-}(-k)$, will enable us to exclude the possibility of $\lambda_{j}^{+}(\xi)=\lambda_{j'}^{-}(\xi)$ with $|j-j'|\geq 2$.
Indeed, on the one hand,
if $k_1\in [\kpa_c, \kpa)$, then 
\begin{align*}
    \sigma_{+}(k_1-\ell\kpa)=- \sigma_{-}(\ell\kpa-k_1)\geq -\sigma_{-}(2\kpa-k_1)> \sigma_{-}(k_1),
\end{align*}
where in the last inequality we have used the strict concavity of $\sigma_{-}(k)$ when $k\geq \kappa_c$.
On the other hand, if $k_1\in (0,k_c)$, we have by using the facts that 
$-\sigma_{-}(k)$ is strictly increasing on $[\kappa_c,+\infty)$ and $\sigma_-(k)$ is strictly increasing on $(-\infty, \kappa_c]$ that
\begin{align*}
    \sigma_{+}(k_1-\ell\kpa)=- \sigma_{-}(\ell\kpa-k_1) >-\sigma_{-}(\ell\kpa-\kappa_c) 
    \geq -\sigma_{-}(2\kpa-\kappa_c)> \sigma_{-}(\kappa_c)>\sigma_{-}(k_1).
\end{align*}

The above facts indicate that the non-zero crossing  spectrum, if it exists, must arise from the same spectral curve, that is \eqref{othercrossing}. 
Moreover, from the sign and monotonicity properties of $\sigma_{\pm}$, it follows that $jj'\leq 0$ and $|j-j'|\geq 2$. 
 Note that all these crossings are bounded away from the origin by a uniform distance.
\end{proof}

\begin{figure}[htbp]
    \centering
\includegraphics[width=0.5\linewidth,height=0.4\textheight, keepaspectratio]{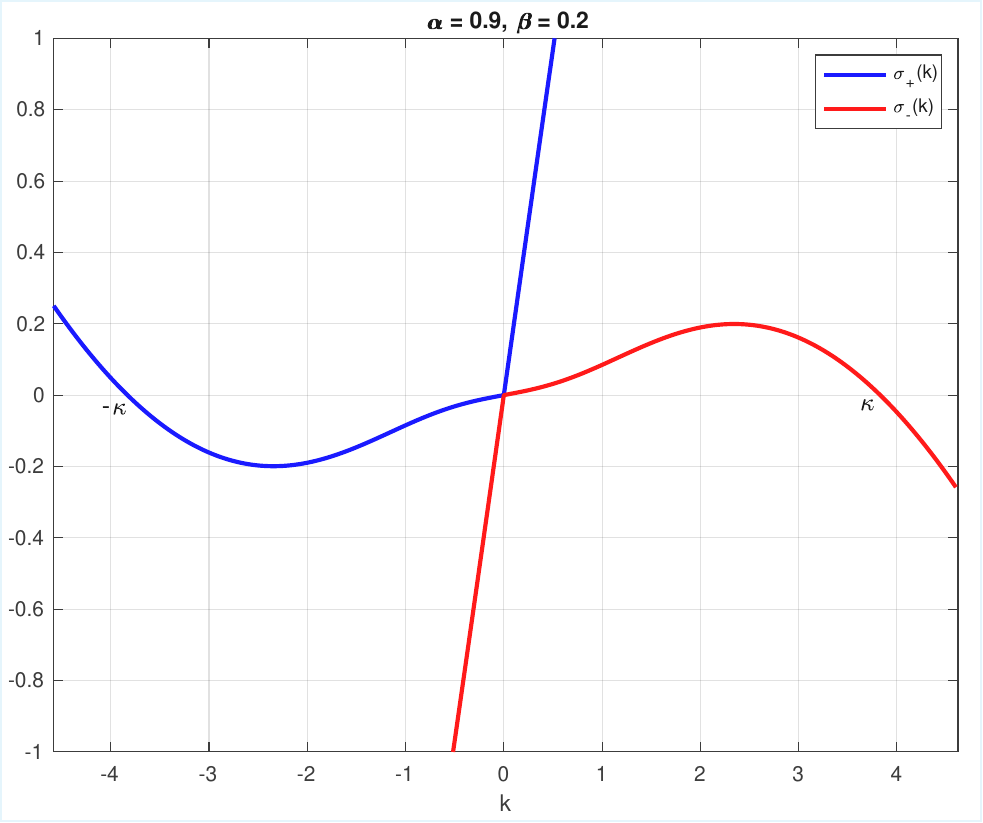}  
    \caption{The picture of curves $\sigma_{+}(k)$ (in blue) and $\sigma_{-}(k)$ (in red) when $\alpha=0.9,\beta=0.2$. }
    \label{fig:figure3}
\end{figure}

\section{Exclusion of  unstable spectrum with large imaginary part}

In this section, we aim to show that the spectrum of 
$L^{\vep}$ lies on the imaginary axis when the imaginary part is large enough. 

\begin{thm}\label{thm-exclusionveryhigh}
There exists $\tilde{\vep}_1>0$ and a constant $M \in (1,+\infty)$ such that for any $|\vep|\le \tilde{\vep}_1$ and any $\xi\in(-\frac{1}{2},\f12]$ it holds that 
\begin{align*}
\sigma(L_{\xi}^{\vep}) \, \cap \{\lambda : |\Im \lambda|\geq M\} \subset \i \mR.
\end{align*}
\end{thm}
\begin{proof}
Let $L_{\xi}^{\vep}=J A_{\xi}^{\vep}$, with the symmetric operator $A_{\xi}^{\vep}$ defined by 
\begin{align*}
A_{\xi}^{\vep}\coloneqq \left( \begin{array}{cc}
   -P_{\xi}[\zeta_{\vep}]+  w_{\vep}& -d_{\vep}k_{\vep}(\p_x+i\xi)  \\[5pt]
   k_{\vep}(\p_x+i\xi)(d_{\vep}\cdot)  &  G_{\xi}[\zeta_{\vep}]\,\cdot  
\end{array}\right),
\end{align*}
where $P_{\xi}[\zeta_{\vep}],\, G_{\xi}[\zeta_{\vep}]$ are defined in \eqref{def-GP}. Since $[-\f12,\f12]$ is compact, it suffices to prove the desired result for $\xi$ lying in a small interval $(\xi_0-\delta,\xi_0+\delta)$, where $\xi_0\in [-\f12,\f12]$ and $\delta>0$ is small enough.

The proof relies on the fact that the symmetric operator $A_{\xi}^{\vep}$ has finitely many negative eigenvalues for sufficiently small $\vep$, which we will now show. When $\vep=0$, the eigenvalues of $A_{\xi}^0$ can be computed via Fourier transform
\begin{align*}
\Lambda_j^{\pm}(\xi)=
\f12 \bigg({(\alpha+\beta\eta_j^2)+|\eta_j|\tah(|\eta_j|)}\pm \sqrt{\big((\alpha+\beta\eta_j^2)-|\eta_j|\tah(|\eta_j|)\big)^2+4\eta_j^2}\bigg)\in \mR
\end{align*}
where $\eta_j\coloneqq\kpa (j+\xi)$. It 
is easily verified that $\Lambda_j^{+}(\xi)>0$ for any $j\in \mathbb{Z}, \xi\in (-\f12,\f12]$, and $\Lambda_j^{-}(\xi)<0$ if and only if
\begin{align*}
   \alpha+\beta\eta_j^2< \f{\eta_j}{\tah(\eta_j)}.
\end{align*}
Consequently, for each $\xi\in (-\f12,\f12]$, there are only finitely many $j\in \mathbb{Z}$ such that $\Lambda_j^{-}(\xi)<0$. For  convenience, we label the real eigenvalues of $A_{\xi}^0$ by 
$(\tilde{\sigma}_j(\xi))_{j=-j_0(\xi)}^{+\infty}$
where $$j_0(\xi)\in (0,+\infty), \qquad 
\tilde{\sigma}_{\ell}(\xi)\leq \tilde{\sigma}_{\ell+1}(\xi),$$
and $\tilde{\sigma}_0(\xi)\geq 0$ is the smallest nonnegative eigenvalue of $A_{\xi}^0$. It is also worth noting that each eigenvalue has finite multiplicity. Take a $k_0
>0$ large enough such that $\tilde{\sigma}_{k_0+1}(\xi)-\tilde{\sigma}_{k_0}(\xi)>0$ for any $\xi\in (\xi_0-\delta, \xi_0+\delta)$.
We can find a ball $B_{R}$ around the origin with radius large enough such that 
$\tilde{\sigma}_{\ell}\in B_R$ for any $\ell\in (-j_0(\xi), k_0)$ and any $\xi\in (\xi_0-\delta, \xi_0+\delta)$, 
and $\p B_R \subset \rho(A_{\xi}^0; (L^2(\mathbb{T}_{2\pi}))^2)$. Moreover, the uniform resolvent estimate
\begin{align*}
 \sup_{|\xi-\xi_0|\leq \delta } \sup_{\lambda\in \partial B_R} \| (\lambda I-A_{\xi}^0)^{-1}\|_{B((L^2(\mathbb{T}_{2\pi}))^2)}<+\infty
\end{align*}
holds.
Since 
\begin{align*}
   \| A_{\xi}^{\vep}-A_{\xi}^0\|_{B( {\rm Dom}(A_{\xi}^0),\,(L^2(\mathbb{T}_{2\pi}))^2)}\lesssim \vep, 
\end{align*}
there exists $\tilde{\vep}_2$, such that for any 
$|\vep|\leq \tilde{\vep}_2, \lambda \in \rho(A_{\xi}^0, Y_{\per})$.
Then we define the projector
\begin{align*}
\tilde{\Pi}_{\xi}^{\vep}\coloneqq\oint_{B_R} (\lambda I-A_{\xi}^{\vep})^{-1}\, \d \lambda\,.
\end{align*}
Since $\|\tilde{\Pi}_{\xi}^{\vep}-\tilde{\Pi}_{\xi}^{0}\|\lesssim \vep$, we know that 
$B_R$ contains the same number of eigenvalues of $A_{\xi}^{\vep}$ and $A_{\xi}^0$, and thus $A_{\xi}^{\vep}$ also has finitely many negative eigenvalues. 

Since $A_{\xi}^{\vep}$ is self-adjoint on $(L^2(\mathbb{T}_{2\pi}))^2$ and has compact resolvent, there is a corresponding orthonormal eigenbasis $\{\theta_j(\xi,\cdot)\}$ of $(L^2(\mathbb{T}_{2\pi}))^2$. Let 
$\theta_{-j_0}(\xi,\cdot),\cdots \theta_{k_0}(\xi,\cdot)$ be the normalized eigenfunctions 
(in $Y_{\per}$)
associated with the negative and first $k_0+1$ non-negative eigenvalues of $A_{\xi}^{\vep}$. Furthermore, for any $u^{\perp}\in \{\theta_{-j_0},\cdots \theta_{k_0}\}^{\perp}$, there exists $c_0>0$ such that 
\begin{align*}
    \langle u^{\perp}, A_{\xi}^{\vep}u^{\perp}\rangle\geq c_0 \|u^{\perp}\|_{Y_{\per}}^2, \qquad  \xi\in(\xi_0-\delta,\xi_0+\delta). 
\end{align*}

Now for $\lambda$ with $\Re \lambda>0$ consider the resolvent problem in $Y_{\per}$
\begin{align}\label{resolventpb}
    (\lambda I-J A_{\xi}^{\vep})u=f.
\end{align}
Writing
\begin{align*}
    u=\sum_{j=-j_0}^{k_0} \gamma_j \theta_j + u^{\perp}
\end{align*}
and taking real part of the $(L^2(\mathbb{T}_{2\pi}))^2$ inner product of the above identity with $A_{\xi}^{\vep}u$, 
we find that 
\begin{align*}
   c_0 \,\Re \lambda  \|u^{\perp}\|_{Y_{\per}}^2\leq \Re \lambda \langle u^{\perp}, A_{\xi}^{\vep}u^{\perp} \rangle &\leq -\Re \lambda \sum_{j=-j_0}^{k_0} |\gamma_j|^2  \langle \theta_j, A_{\xi}^{\vep}\theta_j\rangle+\Re \langle f, A_{\xi}^{\vep}u \rangle\\
   & \leq R\, \Re \lambda  \sum_{j=-j_0}^{k_0} |\gamma_j|^2 + \|f\|_{Y_{\per}}\|u\|_{Y_{\per}}\,,
\end{align*}
where we have used the fact that $\theta_j$ are  eigenfunctions of $A_{\xi}^{\vep}$, associated with the eigenvalues lying inside the ball $B_R$.
It thus follows from  Young's inequality that
\begin{align}\label{es-1}
   \Re \lambda \,  \|u^{\perp}\|_{Y_{\per}}^2\leq C \bigg(\Re \lambda \sum_{j=-j_0}^{k_0} |\gamma_j|^2+(\Re \lambda)^{-1}\|f\|_{Y_{\per}}^2\bigg)\,.
\end{align}
On the other hand, we take the inner product of 
\eqref{resolventpb} with $\theta_j$,  integrate by parts and use the fact that $\theta_j$ is smooth to obtain that 
\begin{align}\label{es-2}
    |\lambda| \sum_{j=-j_0}^{k_0} |\gamma_j|\lesssim \|u\|_{Y_{\per}}+\|f\|_{Y_{\per}}.
\end{align}
Combining \eqref{es-1} and \eqref{es-2}, we find that 
\begin{align*}
\|u\|_{Y_{\per}}\leq \|u^{\perp}\|_{Y_{\per}}+\sum_j |\gamma_j|\|\theta_j\|_{Y_{\per}}& \lesssim \sum_j |\gamma_j|+(\Re \lambda)^{-1}\|f\|_{Y_{\per}}\\
&\lesssim |\lambda|^{-1}\|u\|_{Y_{\per}}+(\Re \lambda)^{-1}\|f\|_{Y_{\per}}\,.
\end{align*}
Consequently, as long as $|\Im \lambda|$ (and hence $|\lambda|$) is
large enough and $\Re \lambda>0$ then 
\begin{align*}
\|u\|_{Y_{\per}}
\lesssim (\Re \lambda)^{-1}\|f\|_{Y_{\per}}\,,
\end{align*}
so that $\lambda\in \rho(L_{\xi}^{\vep}; Y_{\per})$.
\end{proof}

\begin{prop} \label{prop-balls}
Let $(\alpha,\beta)\in \rm I\cup III$ and let $M>0$ be as in the statement of Theorem \ref{thm-exclusionveryhigh}.
There exists $\tilde{\vep}_0>0$  
such that for any $|\eps| \leq \tilde{\vep}_0$ and 
any $\xi\in (-\f12,\f12]$, the spectrum 
$$\sigma(L_{\xi}^{\vep})\cap \{\lambda : |\Im \lambda| \leq M\}$$
is contained in a finite union of
pairwise disjoint balls centered on the imaginary axis, denoted by 
$$B_j= B_{c_j}(z_j(\xi)), \qquad j\in \mathbb{Z}, -\ell_0(\xi)\leq j\leq \ell_1(\xi)$$ with $0<c_j=\cO(|j|^{\f12}), 0<\ell_0(\xi),\ell_1(\xi)=\cO(M^{2/3})$ and $z_j(\xi)\in \i \mR$. 
Moreover, each ball contains the same number of eigenvalues of $L_{\xi}^{\vep}$ as of $L_{\xi}^0$, and when that number is one, the eigenvalue lies on the imaginary axis. 
\end{prop}
\begin{proof}
We first prove the result for $\vep=0$. Recall that 
by \eqref{defspec-0}, the spectrum of $L_{\xi}^0$ is 
composed of $\{\lambda_{j}^{\pm}(\xi)\}_{j\in \mathbb{Z}}$.  By the identity $L_{\xi}^0=\overline{L_{-\xi}^0}$, it suffices to prove the result for $L_{\xi}^0$ with $\xi\in [0,\f12]$.

 Let $\xi_0>0$ be a  fixed sufficiently small number independent of  $\vep$.
We consider separately the case $0\leq \xi\leq \xi_0$ and $\xi_0< \xi\leq \f12$. 
For $0\leq \xi\leq \xi_0$, there is $c_0=\cO(\xi_0)>0$ such that
\begin{align*}
  \lambda_{-1}^{+}(\xi),\, \lambda_{1}^{-}(\xi),\, \lambda_{0}^{\pm}(\xi)  \in B_{c_0}(0),\\
  \Big(\bigcup_{|j|\geq 2} \{\lambda_{j}^{\pm}(\xi) \} \Big)\bigcup \{\lambda_{-1}^{-}(\xi),\lambda_{1}^{+}(\xi) \}\subset (B_{2c_0}(0))^c.
\end{align*}
It suffices to show that $\bigcup_{|j|\geq 2} \{\lambda_{j}^{\pm}(\xi)\}$ is contained in a countable union of  disjoint balls. For any  $|j|\geq 2$, 
there is $\delta_1>0$ small, such that
 $|\lambda_{j}^{\pm}(\xi)- \lambda_{j+1}^{\pm}(\xi)|\geq \delta_1\sqrt{\beta}|j|^{\f12}>0$. 
By monotonicity of $\sigma_{+}(k)$ 
on $[0,+\infty)$ and $[-\infty,-\kappa]$, 
$\sigma_{-}(k)$ on $(-\infty,0]$ and $[\kpa,+\infty)$,  any two eigenvalues $\lambda_k^{\pm}(\xi)$ and $\lambda_{m}^{\pm}(\xi)$ (say $k>m$) are well-separated by a distance of order $\cO\big((|k|+|m|)^{\f12}
\big)$, as long as 
$$km\geq 0 \quad \text{\,or\,} \quad  km\leq 0 \,\, \text{\,\,but\,} \,\,  (k,m)\neq (j_{\ell}, j_{\ell}'); (-j_{\ell}', -j_{\ell}),$$
where $j_{\ell}, j_{\ell}'$ are defined in Lemma \ref{lem-crossing} (2). On the other hand, if 
$(k,m)=(j_{\ell}, j_{\ell}')$ but $|\xi-\xi_{\ell}|\geq  \delta_2$, with $\delta_2 \in (0,\delta_1)$ a small parameter to be chosen, then
\begin{align*}
  |\lambda_{k}^{+}(\xi)-\lambda_{m}^{+}(\xi)|= |\lambda_{j_{\ell}}^{+}(\xi)-\lambda_{{j_{\ell}}}^{+}(\xi_{\ell})|+|\lambda_{{j'_{\ell}}}^{+}(\xi)-\lambda_{{j'_{\ell}}}^{+}(\xi_{\ell})|\gtrsim_{\alpha}\sqrt{\beta} \delta_2 \, (|k|+|m|)^{\f12},
\end{align*}
and thus the two eigenvalues remain separated by a distance of order
$\cO(|k|+|m|)^{\f12}$.
Similar facts hold for $(k,m)=(-j_{\ell}', -j_{\ell})$ and $|\xi+\xi_{\ell}|\geq  \delta_2$. 
However, if $(k,m)=(j_{\ell}, j_{\ell}')$ but $|\xi-\xi_{\ell}|\leq  \delta_2 $, 
$\lambda_{k}^{+}(\xi)$ and $\lambda_m^{+}(\xi)$ 
may be arbitrarily close. If this occurs, we proceed as follows: on the one hand, we choose a sufficiently large number $K$ such that 
$$\lambda_{k}^{+}(\xi),\, \lambda_{m}^{+}(\xi)\in B_{K\delta_2\sqrt{\beta |k|}}(\i\sigma_{\ell}), \qquad (\,\i\sigma_\ell=\lambda^{+}_{j_{\ell}}(\xi_{\ell}))\,.$$
On the other hand, we select
$\delta_2$ small enough so that $K\delta_2\leq {\delta_1}/{4}$, which ensures  that the ball is separated from the rest of the spectrum of $L_{\xi}^0$ by a positive distance of order $\cO(|k|^{\f12})$. This completes the proof for $0\leq \xi\leq \xi_0$. The case of 
$\xi_0<\xi\leq \f12$ is similar, the only difference being that there is no spectral crossing near zero, and the eigenvalues
$\lambda_{-1}^{+}(\xi),\, \lambda_{1}^{-}(\xi),\, \lambda_{0}^{\pm}(\xi)$ are separated. It suffices to construct four disjoint balls to isolate them. The construction of disjoint balls covering the rest of the spectrum proceeds as in the previous arguments, and we therefore omit the details.

We now turn to the case of non-zero $\vep$. 
For convenience, we relabel the disjoint balls found above for $\vep=0$ as 
\begin{align*}
   B_j= B_{c_j}(z_j(\xi)), \quad c_j=\cO(1+|j|^{\f12}), \quad \inf_{j\in \mathbb{Z}} c_j\geq 2\delta_3>0, \, \qquad (j\in \mathbb{Z}, -\ell_0(\xi)\leq j\leq \ell_1(\xi)).
\end{align*}
Each ball contains either four, two or one spectral values of $L_{\xi}^0$. 
Moreover, denote 
\begin{align*}
  \mathcal{E}(\xi)=  (\cup_j B_j(\xi))^c\cap \{|\Im \lambda |\leq M\} ,
\end{align*}
 we can ensure, by the construction,  the uniform resolvent estimates
\begin{align}\label{resov-1}
\sup_{\xi\in (-\f12,\f12]} 
\sup_{ \lambda \in\mathcal{E}(\xi)}
\|(\lambda I-L_{\xi}^0)^{-1}\|_{B(Y_{\per})}<+\infty,
\end{align}
and 
\beq\label{resov-2}
\begin{aligned}
&\sup_{\xi\in (-\f12,\f12]} 
\sup_{\lambda \in 
\mathcal{E}(\xi)} 
\|(\lambda I-L_{\xi}^0)^{-1}\|_{B(Y_{\per}, {\rm Dom}(L_{\xi}^0))}\lesssim 1+M^{2/3}. 
\end{aligned}
\eeq
On the other hand, it holds that 
\begin{align*}
   \| L_{\xi}^{\vep}-L_{\xi}^0\|_{B( {\rm Dom}(L_{\xi}^0),\,Y_{\per})}\lesssim \vep.
\end{align*}
Consequently, there is $\tilde{\vep}_0>0$ such that  for any $|\vep|\leq \tilde{\vep}_0$ and any $\xi\in (-\f12,\f12]$
\begin{align*}
    \cE(\xi)\subset \rho(L_{\xi}^{\vep}; Y_{\per}), 
\end{align*}
and 
 the spectral projections 
 \begin{align}\label{def-specproj}   \Pi_{\xi,j}^{\vep}\coloneqq\oint_{\p B_j} (\lambda I-L_{\xi}^{\vep})^{-1}\d \lambda
 \end{align}
satisfy 
\begin{align*}
    \|\Pi_{\xi,j}^{\vep}-\Pi_{\xi,j}^0\|_{B(Y_{\per})}\lesssim \tilde{\vep}_0, \qquad  -\ell_0(\xi)\leq j\leq \ell_1(\xi).
\end{align*}
The last estimate ensures that $B_j$ contains the same number of eigenvalues of $L_{\xi}^{\vep}$ as of $L_{\xi}^0$.  Due to Hamiltonian symmetry, if a ball $B_j$ contains only one eigenvalue of $L_{\xi}^{\vep}$, then this eigenvalue must lie on the imaginary axis. 
\end{proof}

To determine whether $L_{\xi}^{\vep}$ has an unstable spectrum, it suffices to examine its spectrum within the balls that contain two or four eigenvalues of $L_{\xi}^0$, close to the crossing spectrum identified in Lemma \ref{lem-crossing}. Let us remark that from Theorem \ref{thm-exclusionveryhigh} and
the proof of Lemma \ref{lem-crossing}, for 
$|\vep|\leq \min \{\tilde{\vep}_1,\tilde{\vep}_0\},$
 the possible unstable spectrum of $L^{\vep}$ is contained in a finite union of disjoint balls. For clarity, we denote them by 
 $B_{c_0}(0)$ and $B_{r_{\ell}}(\i \sigma_{\ell})$, $\ell=2,\ldots, N_0$, where the points $\i\sigma_{\ell}$ are spectral crossings identified in Lemma \ref{lem-crossing} that satisfy $|\sigma_{\ell}|\leq M$. 
In the next two sections, we analyze the spectrum of  $L_{\xi}^{\vep}$ insides these balls, focusing separately on the parts away from zero (in $B_{r_{\ell}}(\i \sigma_{\ell})$) and near zero (in $B_{c_0}(0)$). The result stated in Theorem \ref{thm-1} will then be a consequence of Theorem \ref{thm-highfreq} and Theorem \ref{thm-modulation}.

\section{Stability near the crossings away from the origin} 
In this section, we show that the spectrum of $L^{\vep}$ near the crossing spectrum 
away from the origin remains on the imaginary axis when $\vep$ is sufficiently small. Letting $B_{r_{\ell}}(\i \sigma_{\ell})$, $\ell=2,\ldots N_0$, be defined in the previous section, we prove the following.

\begin{thm}\label{thm-highfreq}
Let $(\alpha,\beta)\in \rm I\cup III$. 
There exists $\vep_1>0$ such that
for any $|\vep|\leq \vep_1$ and any $\ell=2,\ldots, N_0$, the spectrum of $L^{\vep}$ within the ball   $B_{r_{\ell}}(\i\sigma_{\ell})$ lies on the imaginary axis. 
\end{thm}

Since $\i\sigma_{\ell}$ is 
crossing spectrum of $L_{\xi_{\ell}}^{\vep}$,
we only need to examine the spectrum of $L_{\xi}^{\vep}$ for $\xi$ close to $\xi_{\ell}$ inside the ball $B_{r_{\ell}}(\i\sigma_{\ell})$ which contains two discrete eigenvalues (up to multiplicity) of $L_{\xi}^{\vep}$ for any $|\vep|\leq \tilde{\vep}_0$, with $\p B_{r_{\ell}}(\i\sigma_{\ell})\in \rho(L_{\xi}^{\vep})$  and $|r_{\ell}|=\cO(|\ell|^{\f12})$. 
To get more precise information about these two eigenvalues, we will compute the representation matrix of the operator $L_{\xi}^{\vep}$ acting on the invariant subspace spanned by generalized eigenvectors corresponding to these eigenvalues by constructing a basis and a dual basis for the subspace.

\begin{proof}

The following computation applies for any $2\leq \ell\leq N_0$, but for notational simplicity, we will denote 
\begin{align*}
    j'\coloneqq j'_{\ell}, \quad j\coloneqq j_{\ell}. 
\end{align*}
We start from a basis of eigenvectors of $L_{\xi}^0$ for the spectral subspace corresponding to spectrum close to $\lambda_j^{+}(\xi_{\ell}),$
which can be computed explicitly thanks to the Fourier transform
\begin{align*}
  q_1^0(\xi,x)=  e^{\i j'  x} \begin{pmatrix}1\\\,\i\,\omega_{j'}(\xi)\end{pmatrix} , \qquad  q_2^0(\xi,x)=  e^{\i j  x} \begin{pmatrix}1\\\,\i\,\omega_{j}(\xi)\end{pmatrix} ,
\end{align*}
where 
\begin{align} \label{defomegaj}
\omega_{j}(\xi)\coloneqq\sqrt{\frac{\alpha+\beta \kappa^2(j+\xi)^2}{\kappa (j+\xi) \tah (\kappa (j+\xi))}}>0 , \qquad (j,\xi)\neq (0,0).
\end{align}
 To  extend $q_1^0$ and $q_2^0$ for $\vep\neq 0$, 
we define the spectral projector $\Pi_{\xi}^{\vep}$ 
\[
\Pi_{\xi}^{\vep}\coloneqq\f{1}{2i \pi} \oint_{\p B_{r_{\ell}}(\lambda_{\ell})}
(\lambda I-L_{\xi}^{\vep})^{-1}\, \d \lambda, \qquad 
\]
which is valid if $|\vep|\leq \tilde{\vep}_0$, and $|\xi-\xi_{\ell}|\leq \delta_2$  where $\tilde{\vep}_0>0$ and $\delta_2>0$ appear in the statement of Proposition \ref{prop-balls} and its proof.  Moreover, thanks to \eqref{resov-1}, \eqref{resov-2}, there is $\vep_1\in(0, \tilde{\vep}_0]$ such that
it holds that 
\begin{align}\label{unies-proj-vep}
    \sup_{|\vep|\leq \vep_1 } \sup_{|\xi-\xi_{\ell}|\leq \delta_2} \|\Pi_{\xi}^{\vep}\|_{B(Y_{\per})}+\|\p_{\vep}\Pi_{\xi}^{\vep}\|_{B(Y_{\per})}<+\infty.
\end{align}

With this in hand, we define the extension operator $\cU_{\xi}(\vep)$ by solving the ODE
 \begin{align}\label{def-extop-ep}
\p_{\vep}\mathcal{U}_\xi(\vep)
=[\p_{\vep}\Pi_{\xi}^{\vep},\Pi_{\xi}^{\vep}]\ \mathcal{U}_\xi(\vep)\,, \qquad\qquad
\mathcal{U}_{\xi}(0)=\Id.
\end{align}
Then we set
\begin{align*}
    q_1^{\vep}(\xi,\cdot)= \mathcal{U}_\xi(\vep)q_1^0(\xi,\cdot), \quad q_2^{\vep}(\xi,\cdot)=\mathcal{U}_\xi(\vep)q_2^0(\xi,\cdot)
\end{align*}
which, by Kato's perturbation theory \cite[Chapter 2, Section 4]{Kato}, 
form a basis for the spectral subspace of $L_{\xi}^{\vep}$ associated to the spectrum inside of $B_{r_{\ell}}(\lambda_{\ell})$. 
Let us remark that by \eqref{unies-proj-vep} and the Cauchy--Lipschitz theorem for ODE in Banach spaces, \eqref{def-extop-ep} admits a solution in the interval $\vep\in [-{\vep}_1, {\vep}_1]$, uniformly in $|\xi-\xi_{\ell}|\leq \delta_2$. We refer to Proposition \ref{prop-extop} for properties of this extension operator.

By the Hamiltonian symmetry \eqref{HS-Lxi}, $\{Jq_1^{\vep}(\xi,\cdot), Jq_2^{\vep}(\xi,\cdot)\}$ forms a basis for the spectral subspace of $(L_{\xi}^{\vep})^{*}$ associated with eigenvalues inside of $B_{r}(\bar{\lambda}_0)$.
Since $j\neq j'$, it holds that 
\begin{align}\label{id-ortho}
\langle J\,q_{1}^{0}(\xi,\cdot),q_{2}^{0}(\xi,\cdot)\rangle&=0,&
\langle J\,q_{2}^{0}(\xi,\cdot),q_{1}^{0}(\xi,\cdot)\rangle&=0.&
\end{align}
Moreover, it follows from direct computations that
\begin{align}\label{Krein}
\langle J q_{1}^{0}(\xi,\cdot), q_{1}^{0}(\xi,\cdot)\rangle&=- 4\,\pi\,\i\,\omega_{j'}(\xi),&
\langle Jq_{2}^{0}(\xi,\cdot), q_{2}^{0}(\xi,\cdot)\rangle&=- 4\,\pi \, \i  \, \omega_{j}(\xi).
\end{align}
By \eqref{preseve-syminner}, the extension operator $\cU_{\xi}(\vep)$ preserves the 
symplectic inner product, thus
the identities \eqref{id-ortho} and \eqref{Krein} hold true when $\vep=0$ is replaced by $\vep\neq 0$.
This motivates us to define the following dual basis when $\vep\neq 0$:
\begin{align*}
   \tilde{q}_1^{\vep}(\xi,\cdot)&=-\f{\i\,  Jq_{1}^{\vep}(\xi, \cdot)}{4\,\pi\,\omega_{j'}(\xi)}\,,&
\tilde{q}_{2}^{\vep}(\xi, \cdot)=-\f{\i\, Jq_{2}^{\vep}(\xi, \cdot)}{4\,\pi\, \omega_{j}(\xi)}, 
\end{align*}
which satisfies $\langle \tilde{q}_j^{\vep}(\xi,\cdot), q_k^{\vep}(\xi,\cdot) \rangle=\delta_{jk}$. 

With these in hand, we  find that
\begin{align*}
    L_{\xi}^{\vep}(q_1^{\vep}(\xi,\cdot),q_2^{\vep}(\xi,\cdot))=(q_1^{\vep}(\xi,\cdot),q_2^{\vep}(\xi,\cdot)) D_{\xi}^{\vep}
\end{align*}
where 
\begin{align*}
 D_{\xi}^{\vep}= \left( \begin{array}{cc}
\langle \tilde{q}_{1}^{\vep}(\xi, \cdot), L_{\xi}^{\vep}  {q}_{1}^{\vep}(\xi, \cdot)\rangle & \langle \tilde{q}_{1}^{\vep}(\xi, \cdot), L_{\xi}^{\vep}  {q}_{2}^{\vep}(\xi, \cdot)\rangle\\[3pt]
 \langle \tilde{q}_{2}^{\vep}(\xi, \cdot), L_{\xi}^{\vep}  {q}_{1}^{\vep}(\xi, \cdot)\rangle & \langle \tilde{q}_{2}^{\vep}(\xi, \cdot), L_{\xi}^{\vep} q_{2}^{\vep}(\xi, \cdot)\rangle
    \end{array}\right).
\end{align*}
The spectrum of $L_{\xi}^{\vep}$ inside of $B_{\delta}(\lambda_0)$ is thus the eigenvalues of 
$D_{\xi}^{\vep}$. In view of the relation \eqref{relation-imp},
\begin{align*}
   D_{\xi}^{\vep}&=
   \f{\i}{{4\,\pi}}\left( \begin{array}{cc}
\f{\langle {q}_{1}^{\vep}(\xi, \cdot), A_{\xi}^{\vep}  {q}_{1}^{\vep}(\xi, \cdot)\rangle}{\omega_{j'}(\xi)} & \f{\langle {q}_{1}^{\vep}(\xi, \cdot), A_{\xi}^{\vep}  {q}_{2}^{\vep}(\xi, \cdot)\rangle}{\omega_{j'}(\xi)}\\[5pt]
 \f{\langle {q}_{2}^{\vep}(\xi, \cdot), A_{\xi}^{\vep}  {q}_{1}^{\vep}(\xi, \cdot)\rangle}{{\omega_{j}(\xi)}} & \f{\langle {q}_{2}^{\vep}(\xi, \cdot), A_{\xi}^{\vep} q_{2}^{\vep}(\xi, \cdot)\rangle}{{\omega_{j}(\xi)}}
    \end{array}\right) 
    = \f{\i}{ 4\,\pi} P_{\xi}\tilde{D}_{\xi}^{\vep}P_{\xi}^{-1}
\end{align*}
with 
\begin{align*}
P_{\xi}\coloneqq\left( \begin{array}{cc}
    \sqrt{\omega_{j}(\xi)} & 0\\
    0 &  \sqrt{\omega_{j'}(\xi)}
     \end{array}\right),\qquad \tilde{D}_{\xi}^{\vep}\coloneqq\left( \begin{array}{cc}
\f{\langle {q}_{1}^{\vep}(\xi, \cdot), A_{\xi}^{\vep}  {q}_{1}^{\vep}(\xi, \cdot)\rangle}{\omega_{j'}(\xi)} & \f{\langle {q}_{1}^{\vep}(\xi, \cdot), A_{\xi}^{\vep}  {q}_{2}^{\vep}(\xi, \cdot)\rangle}{\sqrt{\omega_{j'}(\xi)\,\omega_{j}(\xi)}} \\[5pt]
 \f{\langle {q}_{2}^{\vep}(\xi, \cdot), A_{\xi}^{\vep}  {q}_{1}^{\vep}(\xi, \cdot)\rangle}{\sqrt{\omega_{j'}(\xi)\,\omega_{j}(\xi)}} & \f{\langle {q}_{2}^{\vep}(\xi, \cdot), A_{\xi}^{\vep} q_{2}^{\vep}(\xi, \cdot)\rangle}{{\omega_{j}(\xi)}}\end{array}\right).
\end{align*}
Since $A_{\xi}^{\vep}$ is a symmetric operator, it holds that $(\tilde{D}_{\xi}^{\vep})^{*}=\tilde{D}_{\xi}^{\vep}$. Therefore, $\sigma(\tilde{D}_{\xi}^{\vep})\subset \mR$ and hence  $\sigma({D}_{\xi}^{\vep})\subset \i \mR$. The proof is thus finished.
\end{proof}

\section{Modulational stability}
In this section, we investigate the spectral properties of $L^{\vep}$ near the origin, which corresponds to the modulational (in-)stability of the periodic waves. The  result stated in Theorem \ref{thm-fullmodulation} is a consequence of the 
following theorem.
\begin{thm}\label{thm-modulation}
Let $(\alpha,\beta)\in \rm I\cup II \cup III$, 
and let
$\xi_0>0$ $c_0>0$ be as in the proof of Proposition \ref{prop-balls}. 
If the index  $\tilde{C}(\alpha,\beta)$ defined in \eqref{def-index-real} is positive, then there 
exists $\vep_2>0,\,\xi_2\in(0,\xi_0]$ such that for any $|\vep|\leq {\vep}_2, |\xi|\leq \xi_2, $ it holds that
\[
\displaystyle
\sigma\left(L_{\xi}^{\vep}\right)\cap B_{c_0}(0)\subset \i\,\mathbb{R}\,.
\]
On the other hand, if  $\tilde{C}(\alpha,\beta)<0$,
then there is a constant $A=A(\alpha,\beta)$, such that if $|\xi|<A|\vep|\ll 1$, then
$\sigma\big(L_{\xi}^{\vep}\big) \cap B_{c_0}(0)$ contains two purely imaginary eigenvalues and two eigenvalues with nonzero real parts, whose absolute values are of order $\mathcal{O}(|\vep\xi|)$. 

\end{thm}
The general strategy for proving this result is, once again, to construct a representation matrix whose eigenvalues coincide with the spectrum of 
 $L_{\xi}^{\vep}$ near $0$. This will be achieved by constructing an appropriate basis and dual basis. Since zero is a spectral crossing of 
 $L_0^0$ with multiplicity four, the resulting matrix will be  $4\times 4$.
 To construct the  basis, one can start with the eigenfunctions of $L_{\xi}^0$  
 and extend them to the case $\vep\neq0$. However, this approach requires computing the expansions of 
 these basis functions in terms of the amplitude $\vep$ (at least to the second order), making it computationally demanding. 
 Therefore, we instead begin with a basis of $L_0^{\vep}$,  whose spectrum near the origin consists solely of $\{0\}$  with algebraic multiplicity 4 and geometric multiplicity 2. 
 Moreover, as is classical in modulation theory, the associated eigenfunctions are intimately connected to the structure of the background waves.
 

\begin{prop}[Generalized kernel of $L_0^{\vep}$] \label{prop-kernel}
Zero is an eigenvalue of $L_0^{\vep}$ with algebraic multiplicity four  and geometric multiplicity two, and there exist four (generalized) eigenfunctions 
${ q_1^{\vep}(0,\cdot), \, q_2^{\vep}(0,\cdot),\, q_3^{\vep}(0,\cdot),\, q_4^{\vep}(0,\cdot)}$ such that 
\begin{align*}
    L_{0}^{\vep} q_1^{\vep}(0,\cdot)&=0, \qquad  L_{0}^{\vep} q_2^{\vep}(0,\cdot)=\vep k_{\vep}\frac{\p_{\vep}k_{\vep}}{c\,\p_c k_{\vep}} q_1^{\vep}(0,\cdot); \\
     L_{0}^{\vep} q_3^{\vep}(0,\cdot)&=0, \qquad  L_{0}^{\vep} q_4^{\vep}(0,\cdot)=-q_3^{\vep}(0,\cdot).
\end{align*}
Moreover, $q_1^{\vep}(0,\cdot), q_2^{\vep}(0,\cdot)$ are  explicitly given by
\begin{align*}
    q_1^{\vep}(0,\cdot)&=\vep^{-1}
\begin{pmatrix}  \p_x \zeta_{\vep}\\\p_x \vp_{\vep}- Z_{\vep}\p_x \zeta_{\vep} \end{pmatrix}, \quad q_2^{\vep}(0,\cdot)=
    \begin{pmatrix}
     1 & 0 \\
      -Z_{\vep} & 1
   \end{pmatrix}
    \left(\p_{\vep}+\frac{\p_{\vep}k_{\vep}}{\p_c k_{\vep}} \begin{pmatrix} \p_c &0 \\0& c^{-1}\p_c(c\,\cdot)
    \end{pmatrix} \right)
    \begin{pmatrix}
      \zeta_{\vep} \\
      \vp_{\vep}
    \end{pmatrix}, 
    \end{align*}
while $q_3^{\vep}(0,\cdot), q_4^{\vep}(0,\cdot)$ are smooth in terms of $\vep$ and  expand as
\begin{align} \label{q3q4}
     q_3^{\vep}(0,\cdot)&=  \begin{pmatrix}
      0 \\
   1
   \end{pmatrix}+\vep \Lambda_0\, q_1^{\vep}+\vep^2\Lambda_1(\vep)q_1^{\vep}, \\
  q_4^{\vep}(0,\cdot)&= 
 \begin{pmatrix}
   \f{1}{\alpha} \\ 0
  \end{pmatrix}+ \vep \frac{\kappa}{2\alpha}  \begin{pmatrix}
   \ch (\kappa) \cos (x)\\ -\sh (\kappa) \sin (x)
   \end{pmatrix}+\vep^2  q_{42}^{\vep}(x) \,, 
   \label{exp-q40}
\end{align}
where $\Lambda_0=\frac{\kappa^2}{\alpha \sh(2\kappa)}$, $ \Lambda_1(\vep)$ is a smooth function in $\vep$. Moreover, $q_{42}^{\vep}(x)$ is smooth in $x$ with its first element being even and the second element being odd in $x$.
\end{prop}
\begin{rmk}
    Note that the background waves $(\zeta_{\vep},\vp_{\vep})$ depend on the wave speed $c$ 
   through the parameters 
    $\alpha=\f{gh}{c^2}$ and $\beta=\f{T}{hc^2}$ 
    since the wave numbers $k_{\vep}$ depend on these two parameters.
\end{rmk}
The kernel $q_1^{\vep}(0,\cdot)$ can be found by using the translational invariance of the profile equation. Since $G_0[\zeta_{\vep}] (1)=0$, the profile equation is invariant under constant shifts of the velocity potential. This results in another kernel element $(0,1)^t$. 
Being the linear combination of $q_1^{\vep}(0,\cdot)$
 and $(0,1)^t$, we see that $q_3^{\vep}(0,\cdot)$ also belongs to the kernel of $L_0^{\vep}$.
The generalized eigenfunction  $q_2^{\vep}(0, \cdot)$
 can be found by differentiating the original profile equation before normalization in terms of the wave speed $c$ and the amplitude $\vep$, while $q_4^{\vep}(0, \cdot)$ is constructed by using the Lyapunov--Schmidt method. We will give the proof in Appendix \ref{appen-kernel}.

From the above proposition  and 
expansions of the profiles in \eqref{expan-profile}, we derive the following basis expansion.

\begin{cor}[Expansion of the basis in $\vep$]\label{cor-exp-0}
It holds that 
\begin{align*}
q_1^{\vep}(0,\cdot)&= \bpm -\sh(\kappa)\sin (x)\\ \ch(\kappa) \cos (x)\epm +\vep \bpm \p_x\zeta_2\\\p_x\vp_2-\kappa \sh(\kappa) \sin (x)\, \p_x\zeta_1 \epm  +\vep^2  \begin{pmatrix}
  \mathrm{ odd }\\ \mathrm{ even }
 \end{pmatrix}, \\ 
 q_2^{\vep}(0,\cdot)&=\bpm \sh(\kappa)\cos(x)\\ \ch(\kappa) \sin (x)\epm +\vep \bpm 2\,\zeta_2\\ 2\vp_2-\kappa \sh(\kappa) \sin (x)\, \zeta_1 \epm
 +\vep^2   \begin{pmatrix}
  \mathrm{ even }\\ \mathrm{ odd }
 \end{pmatrix}, \\
 q_3^{\vep}(0,\cdot)&=\bpm 0\\ 1 \epm +\vep \Lambda_0  \bpm -\sh(\kappa)\sin (x)\\ \ch(\kappa) \cos (x)\epm +
 \vep^2  \begin{pmatrix}
   \mathrm{ odd }\\ \mathrm{ even }
 \end{pmatrix},\\ 
\end{align*}
and $q_4^{\vep}(0,\cdot)$ expands as in \eqref{exp-q40}. We refer to  \eqref{def-expzetavp} for the definition of $\zeta_1,\zeta_2,\vp_2$.
\end{cor}
\begin{cor}\label{corD0ep}
 It holds that 
 \begin{align*}
 L_0^{\vep} (q_1^{\vep}, q_2^{\vep}, q_3^{\vep}, q_4^{\vep})(0,\cdot)=(q_1^{\vep}, q_2^{\vep}, q_3^{\vep}, q_4^{\vep})(0,\cdot)\rD_0^{\vep}
 \end{align*}
 with the constant-coefficient $(4\times 4)$ matrix 
\begin{align*}
    \rD_0^{\vep}= \bpm \rA_0^{\vep}  & 0 \\  0 &  F_0^{\vep}\epm  \quad \text { with }
\rA_0^{\vep}=\bpm  0 & \vep k_{\vep}\frac{\p_{\vep}k_{\vep}}{c\,\p_c k_{\vep}} \\ 0& 0\epm, \quad 
\rF_0^{\vep}= \bpm 0 & -1\\ 0& 0\epm . 
\end{align*}

\end{cor}

\subsection{Construction of the basis and dual basis}

With a basis for the generalized kernel of 
$L_0^{\vep}$ in hand, we proceed by employing the Kato extension operator in the Floquet parameter  $\xi$ 
to construct a basis for the eigenspaces of $L_{\xi}^{\vep}$ associated with the four eigenvalues near zero. 

To do so, let us first define the spectral projector 
\begin{align*}
    \Pi_{\xi}^{\vep}= \oint_{\p B_{c_0}(0)} \big(\lambda I-L_{\xi}^{\vep}\big)^{-1}\, \d \lambda
\end{align*}
where $|\xi|\leq \xi_0, |\vep|\leq \tilde{\vep}_0$. By the proof of Proposition \ref{prop-balls}, $\p B_{c_0}(0) \cap \sigma(L_{\xi}^{\vep})=\emptyset$ and there are exactly four elements of $\sigma(L_{\xi}^{\vep})$ inside of $B_{c_0}(0)$.

We then define the Kato extension operator $\cU^{\vep}(\xi)$ by solving the ODE
\begin{align}\label{def-extop-xi}
\p_\xi\mathcal{U}^{\vep}(\xi)
=[\p_\xi\Pi_{\xi}^{\vep}, \Pi_{\xi}^{\vep}]\ \mathcal{U}^{\vep}(\xi)\,, \qquad
\mathcal{U}^{\vep}(0)=\Id.
\end{align}
Thanks to \eqref{resov-1}, \eqref{resov-2},
it can be verified that there is $\tilde{\vep}_2\in (0,\tilde{\vep}_0]$ ($\tilde{\vep}_0$ being the one appearing in the statement of Proposition \ref{prop-balls}),
such that the following estimates holds
\begin{align*}
    \sup_{|\vep|\leq \tilde{\vep}_2 } \sup_{|\xi|\leq \xi_0} \|\Pi_{\xi}^{\vep}\|_{B(Y_{\per})}+\|\p_{\xi}\Pi_{\xi}^{\vep}\|_{B(Y_{\per})}<+\infty.
\end{align*}
This ensures that $\p_\xi\mathcal{U}^{\vep}(\xi)$ exists on $\xi\in[-\xi_0,\xi_0]$ uniformly in $\vep\in [-\tilde{\vep}_2,\tilde{\vep}_2]$.
We also refer to Proposition \ref{prop-extop} for some important properties of the extension operator $\cU^{\vep}(\xi)$.

Define 
\beq \label{def-basis}
q_j^{\vep}(\xi,\cdot)=\cU^{\vep}(\xi) q_j^{\vep}(0,\cdot), \quad \, j=1,2,3,4,
\eeq
which, by Proposition \ref{prop-extop} (1), form a basis of $\mathrm{Ran}(\Pi_{\xi}^{\vep})$, which are smooth in $\xi$ and $\vep$. 
To obtain the matrix representation of $L_{\xi}^{\vep}$ relative to this basis, we need to find the dual basis. By the Hamiltonian symmetry \eqref{HS-Lxi}, 
$\{ Jq_1^{\vep}, Jq_2^{\vep}, Jq_3^{\vep}, J q_4^{\vep}\}(\xi,\cdot)$ spans the eigenspaces of $(L_{\xi}^{\vep})^{*}$ corresponding to the eigenvalues inside  $B_{\ep_0}(0)$ and we can therefore find the dual basis from linear combinations of them. 
Thanks to Corollary \ref{cor-exp-0}, we have
\begin{align*}
    \langle Jq_j^{\vep}(0,\cdot), q_{\ell}^{\vep}(0,\cdot)\rangle =0, \qquad (j,\ell)\in \{(1,3), (3,1), (2,4), (4,2) \}.
\end{align*}
By \eqref{preseve-syminner}, the same identity holds for sufficiently small nonzero $\xi$:
\beq \label{orthogonality}
    \langle Jq_j^{\vep}(\xi,\cdot), q_{\ell}^{\vep}(\xi,\cdot)\rangle =0, \qquad (j,\ell)\in \{(1,3), (3,1), (2,4), (4,2) \}.
\eeq
Therefore, we define the dual basis in the following way: 
\beq \label{def-dualbasis}
\begin{aligned}
\tilde{q}_1^{\vep}(\xi,\cdot)=a_{11}({\vep})Jq_2^{\vep}(\xi,\cdot) + a_{12}({\vep})Jq_4^{\vep}(\xi,\cdot), & \qquad \tilde{q}_2^{\vep}(\xi,\cdot)=a_{21}({\vep})Jq_1^{\vep}(\xi,\cdot) + a_{22}({\vep}) Jq_3^{\vep}(\xi,\cdot),\\
\tilde{q}_3^{\vep}(\xi,\cdot)=a_{31}({\vep})Jq_2^{\vep}(\xi,\cdot) + a_{32}({\vep})Jq_4^{\vep}(\xi,\cdot), & \qquad \tilde{q}_4^{\vep}(\xi,\cdot)=a_{41}({\vep})Jq_1^{\vep}(\xi,\cdot) + a_{42}({\vep})Jq_3^{\vep}(\xi,\cdot).
\end{aligned}
\eeq
The ($\vep$-dependent) coefficients $a_{jk}({\vep})$ ($j=1,2,3,4;k=1,2$) are determined from the conditions 
\begin{align}
    \langle \tilde{q}_j^{\vep}, q_j^{\vep} \rangle (\xi) 
    =1, \quad 
    \langle \tilde{q}_j^{\vep}, q_{\ell}^{\vep} \rangle(\xi) =0, \qquad (j,\ell)\in \{(1,3), (3,1), (2,4), (4,2) \}
\end{align}
 (the other combinations vanish by \eqref{orthogonality} and the skew-symmetry of $J$). 
By \eqref{preseve-syminner} and Corollary \ref{cor-exp-0},
we can compute these coefficients explicitly. 
Define 
\begin{align*}
 b_{j\ell}({\vep})=\langle Jq_j^{\vep}(\xi,\cdot), q_{\ell}^{\vep}(\xi,\cdot)\rangle=\langle Jq_j^{\vep}(0,\cdot), q_{\ell}^{\vep}(0,\cdot)\rangle\in \mR. 
\end{align*}
direct computations show that 
\begin{align}\label{compaij}
\bpm a_{11} & a_{31} \\
a_{12}& a_{32}\epm(\vep) =\gamma({\vep}) \bpm -b_{34}& b_{14} \\ b_{32} & -b_{12}\epm (\vep), \qquad \bpm a_{21} & a_{41} \\
a_{22}& a_{42}\epm(\vep)=\gamma({\vep}) \bpm b_{34}& -b_{32} \\ -b_{14} & b_{12}\epm (\vep)
\end{align}
where $\gamma(\vep)=\big(b_{12}b_{34}-b_{14}b_{32}\big)^{-1}(\vep)$. 

We summarize their expansions in $\vep$ in the following lemma, which
is needed for the latter proof.
\begin{lem}
It holds that 
\beq\label{expan-bjk}
\begin{aligned}
    & b_{12}(\vep)=\pi \sh(2\kpa)+\cO(\vep^2),  \quad   b_{14}(\vep)= \frac{\pi\kappa\ch(2\kpa)} {2\alpha}\vep+\cO(\vep^2),\\
   &  b_{32}(\vep)=\cO(\vep^2), \qquad\qquad\qquad b_{34}(\vep)=\frac{2\pi}{\alpha}+\cO(\vep^2). \\
\end{aligned}
\eeq
  As a consequence, $a_{jk}(\vep)\in\mR$ satisfy
  \beq\label{expan-aij}
\begin{aligned}
 & a_{11}(\vep)=-a_{21}(\vep)=-\frac{1}{\pi \sh(2\kpa)}+\cO(\vep^2), \qquad 
a_{12}(\vep)=-a_{41}(\vep)=\cO(\vep^2),\\
& a_{31}(\vep)=-a_{22}(\vep)=\frac{\kappa}{4\pi \tah(2\kpa)}\vep+\cO(\vep^2), \qquad a_{32}(\vep)=-a_{42}(\vep)=-\f{\alpha}{2\pi}+\cO(\vep^2).
\end{aligned}
\eeq
\end{lem}
\begin{proof}
The expansions \eqref{expan-aij} follow from  relations \eqref{compaij} and \eqref{expan-bjk}, 
which  in turn follow from  
direct computations and the expansions of $q_{j}^{\vep}(0,\cdot)$ in Corollary \ref{cor-exp-0}. Let us only verify that $b_{32}(\vep)=-b_{23}(\vep)=\cO(\vep^2)$ which will be important in our subsequent analysis: 
\begin{align*}
    b_{23}(\vep)&=\langle Jq_2^{\vep}(0,\cdot), q_{3}^{\vep}(0,\cdot)\rangle\\
    &=\vep\left( \bigg\langle \bpm \ch(\kappa)\sin(x) \\ -\sh(\kappa)\cos(x)\epm  , \frac{\kappa^2}{\alpha \sh(2\kappa)}\bpm -\sh(\kappa)\sin (x)\\ \ch(\kappa) \cos (x)\epm    \bigg\rangle+\bigg\langle  \bpm * \\\f{\kpa^2}{2\alpha}\epm, \bpm 0 \\1 \epm \bigg\rangle \right)+\cO(\vep^2)\\
    &=\cO(\vep^2)\, .
\end{align*}
\end{proof}

\subsection{Analysis of the representation matrix}
Having established the basis and its dual, we are now in a position to compute the representation matrix in this basis. 
Define the $(4\times 4)$ matrix
$\rD_{\xi}^{\vep}$ such that
\begin{align*}
 L_{\xi}^{\vep} (q_1^{\vep}, q_2^{\vep}, q_3^{\vep}, q_4^{\vep})(\xi,\cdot)=(q_1^{\vep}, q_2^{\vep}, q_3^{\vep}, q_4^{\vep})(\xi,\cdot)\rD_{\xi}^{\vep}.
 \end{align*}
By taking suitable inner products of the above identity with the dual basis $\{\tilde{q}_{j}^{\vep}(\xi,\cdot)\}$, it is found that 
\begin{align*}
  \rD_{\xi}^{\vep}= \big(\langle  \tilde{q}_{j}^{\vep}(\xi,\cdot),  L_{\xi}^{\vep} q_{\ell}^{\vep}(\xi,\cdot)\rangle \big)_{{1 \leq j,\ell \leq 4}}\coloneqq (d_{j\ell}(\vep, \xi))_{{1 \leq j,\ell \leq 4}}.  
\end{align*}

We can, in fact, find the expansions of the entries $d_{j\ell}(\vep, \xi)$ in $\xi$ and $\vep$, using the definitions \eqref{def-basis}, \eqref{def-dualbasis}. However, it is not necessary to compute all of these expansions to determine whether the eigenvalues of $\rD_{\xi}^{\vep}$ are purely imaginary. 

The first step is to get the the expression of 
$\rD_{\xi}^{\vep}$ when $\xi$ or $\vep$ vanishes. 
We have already know $\rD_{0}^{\vep}$ from 
Corollary \ref{corD0ep}, so we now focus on determining $\rD_{\xi}^{0}$. 

\subsubsection{Computation of $\rm D_{\xi}^{0}$ }
When $\vep=0$, the operator $L_0^{\vep}$ has constant coefficient, so we can take benefit of the Fourier transform. Denote 
\begin{align*}
p_{\pm 1}^{\mp}(\xi,x)=e^{\pm \i\, x} \bpm \pm \i \\ \omega_{\pm1}(\xi) \epm, 
\qquad p_{0}^{\pm}(\xi,x)=\bpm  \mp \i \sqrt{\kappa\xi \tah(\kpa\xi)}\\ \sqrt{\alpha+\beta \kpa^2\xi^2}\epm 
\end{align*}
where $\xi\neq 0$ and $\omega_j(\xi)$ is defined in \eqref{defomegaj}.
Then 
\begin{align*}
L_{\xi}^{0} \big(p_1^{-}, p_{-1}^{+}, p_{0}^{-}, p_{0}^{+}\big)=\big(p_1^{-}, p_{-1}^{+}, p_{0}^{-}, p_{0}^{+}\big)\,\diag\, (\lambda_1^{-}, \lambda_{-1}^{+}, \lambda_0^{-}, \lambda_0^{+}\big)(\xi)
\end{align*}
 where $\lambda_{j}^{\pm}(\xi)=\i \big(\kappa(j+\xi)\pm w_j(\xi)\big)$ are the eigenvalues of 
 $L_{\xi}^0$ defined in \eqref{defspec-0}. To get
 $\rD_{\xi}^0$, we need to relate $q_j^0(\xi,\cdot)$ ($j=1,2,3,4$) with $p_1^{-}(\xi,\cdot),p_{-1}^{+}(\xi,\cdot),p_{0}^{\pm}(\xi,\cdot)$. 
 \begin{lem}\label{transp-basis}
 There are $(2\times 2)$ matrices $T_0, T_1$ such that 
\begin{align*}
  (q_3^0,q_4^0)(\xi,\cdot)=(p_{0}^{-},p_{0}^{+})(\xi,\cdot)\,T_0\, , \qquad   (q_1^0,q_2^0)(\xi,\cdot)=(p_{1}^{-},p_{-1}^{+})(\xi,\cdot)\,T_1,
\end{align*}
  where 
\begin{align*}
T_0= \f12 \bpm \f{1}{\sqrt{\alpha+\beta\kpa^2\xi^2}} & \frac{-\i}{\alpha\sqrt{\kpa\xi\tah(\kpa\xi)}}\\ \f{1}{\sqrt{\alpha+\beta\kpa^2\xi^2}} & \frac{\i}{\alpha\sqrt{\kpa\xi\tah(\kpa\xi)}}\epm, \quad T_1= \mathrm{diag}\big(a_1(\xi), a_{-1}(\xi)\big)\frac12 \bpm 1 & -\i\\ 1& \i \epm 
\end{align*}
for some smooth coefficients $a_1(\xi), a_{-1}(\xi)$. 
\end{lem}
\begin{proof}
To start, we first compute the expressions of $q_j^0(\xi, \cdot)=\cU^0(\xi)q_j^0(0, \cdot)$. 
By Lemma \ref{lem-exp-extopxi} in Appendix \ref{appen-extension}, we see that $\cU^0(\xi)$ does not shift the eigenmodes and reduces to the identity when applied to the constant vectors
  \begin{align*}
    q_3^0(\xi,x)= \bpm 0 \\1 \epm ,   \quad 
     q_4^0(\xi,\cdot)= \bpm 1/\alpha \\0 \epm .
    \end{align*}
The form of $T_0$ then follows. 

    Define 
    \beqs 
\tilde{p}_{\pm 1}(\xi,\cdot)= \cU^0(\xi) \left( e^{\pm\i x} \bpm \pm \i\sh(\kpa)\\ \ch(\kpa) \epm \right).
    \eeqs
Applying \eqref{exp-extopxi}, we readily compute that 
\begin{align*}
   \tilde{p}_{\pm1}(\xi,x)=e^{\pm \i x}
   \sh(\kpa) \sqrt{\f{\omega_{\pm1}(0)}{\omega_{\pm1}(\xi)}}  \bpm \pm \i \\ \f{\omega_{\pm1}(\xi
   )}{\omega_{\pm1}(0)\,\tah(\kpa)}\epm .
\end{align*}
It follows from the dispersion relation $\alpha+\beta\kpa^2=\f{\kpa}{\tah (\kpa)}$ that 
$\omega_{\pm 1}(0)=\f{1}{\tah(\kpa)}$. Therefore,
we derive that 
\begin{align*}
    \tilde{p}_{\pm 1}(\xi,\cdot)=  \sh(\kpa) \sqrt{\f{\omega_{\pm1}(0)}{\omega_{\pm1}(\xi)}} \,
    p_{\pm 1} ^{\mp}(\xi,\cdot)\,.
\end{align*}
The form of $T_1$ then follows from the fact that
\beqs
 (q_1^0,q_2^0)(\xi,\cdot)=(\tilde{p}_{ 1},\tilde{p}_{ -1})(\xi,\cdot)\frac12 \bpm 1 & -\i\\ 1& \i \epm .
\eeqs
\end{proof}
\begin{cor}\label{cor-Dxi0}
When $\vep=0$, the representation matrix reads
 \begin{align*}
\rD_{\xi}^0=\big(d_{j\ell}(0,\xi)\big)_{1\leq j,\ell\leq 4}   
\end{align*}
where $d_{j\ell}(0,\xi)=0$ for 
 $(j,\ell)\in (\{1,2\}, \{3,4\})$ or $( \{3,4\}, \{1,2\})$
and
\beq\label{entry-xi=0}
\begin{aligned}
  &  d_{11}(0,\xi)=d_{22}(0,\xi)=\f{\i}{2} \bigg(2\kpa\xi+w_1(-\xi)-w_1(\xi)\bigg)=\i \kpa\xi\big(1-w_1'(0)/\kpa\big)+\cO(\xi^2),\\
    & d_{12}(0,\xi)= -d_{21}(0,\xi)=\f12\big(2\kappa-w_1(\xi)-w_{1}(-\xi)\big)=
     - \f{1}{2}w_1''(0)\,\xi^2+\cO(\xi^4), \\
 &   d_{33}(0,\xi)=d_{44}(0,\xi)=\i \kpa \xi,  \quad  d_{34}(0,\xi)=-(1+\beta\kpa^2\xi^2/\alpha), \quad d_{43}(0,\xi)=\alpha \kpa \xi \tah(\kpa\xi).
\end{aligned}
\eeq
Here $w_1(\xi)\coloneqq\sqrt{\kpa(1+\xi)\tah(\kpa(1+\xi))(\alpha+\beta\kpa^2(1+\xi)^2)}$ and when $\alpha\in(0,1], \beta>0$, the relevant coefficients 
\beq\label{impor-coeff}
e_{*}\coloneqq\kpa^{-1}w_1'(0)>1, \qquad w_1''(0)>0. 
\eeq
\end{cor}
\begin{proof}
It follows from Lemma \ref{transp-basis} and 
\begin{align*}
 \rD_{\xi}^0= \bpm T_{1}^{-1} & 0 \\
    0& T_0^{-1}\epm  \mathrm{diag} \big(\lambda_{1}^{-}, \lambda_{-1}^{+}, \lambda_{0}^{-}, \lambda_{0}^{+}\big) \bpm T_{1} & 0 \\
    0& T_0\epm
\end{align*}
that $\rD_{\xi}^0=\bpm \rA_{\xi}^0 & 0\\ 0 & \rF_{\xi}^0 \epm, $ with 
\begin{align*}
    \rA_{\xi}^0&=\f12 \bpm \lambda_1^{-}+ \lambda_{-1}^{+}& \i\,\big(\lambda_{-1}^{+}-\lambda_{1}^{-}\big) \\
 \i\,\big(\lambda_{1}^{-}-\lambda_{-1}^{+}\big) & \lambda_1^{-}+ \lambda_{-1}^{+} \\   \epm\\
 &= \bpm \f{\i}{2} \big(2\kpa\xi+w_1(-\xi)-w_1(\xi)\big) & \f{1}{2}\big(2\kappa- w_1(\xi)-w_{1}(-\xi)\big) \\
 -\f{1}{2}\big(2\kappa- w_1(\xi)-w_{1}(-\xi)\big) & \f{\i}{2} \big(2\kpa\xi+w_1(-\xi)-w_1(\xi)\big) \\   \epm, 
\end{align*}
\begin{align*}
    \rF_{\xi}^0&=\f12 \bpm \lambda_0^{-}+ \lambda_{0}^{+} & -\i\sqrt{\frac{\alpha+\beta\kpa^2\xi^2}{\kpa\xi\tah(\kpa\xi)}}\big(\lambda_0^{-}-\lambda_{0}^{+}\big)/\alpha\\
 \i \alpha \sqrt{{\frac{\kpa\xi\tah(\kpa\xi)}{\alpha+\beta\kpa^2\xi^2} }} \big(\lambda_0^{-}-\lambda_{0}^{+}\big)& \lambda_0^{-}+ \lambda_{0}^{+}\epm \\
 &= \bpm \i\kpa\, \xi & -\big(1+\beta\kpa^2\xi^2/\alpha\big) \\
 \alpha \kpa \,\xi \tah(\kpa \xi)
 & \i\kpa\, \xi \epm. 
\end{align*}
From these identities we find \eqref{entry-xi=0}. 

Finally, the properties in \eqref{impor-coeff}
follow by noting that $w_1(\xi)=\kappa(1+\xi)-\sigma_{-}(\kappa(1+\xi))$, with $\sigma_{-}$ as in the proof of Lemma \ref{lem-crossing}. Indeed, this implies that $\kappa^{-1} w_1'(0)=1-\sigma_{-}'(\kappa)>1$ and $w_1''(0)=-\kappa^2 \sigma_{-}''(\kappa)>0$.
\end{proof}

\subsubsection{Some useful properties of the entries}
We now aim to gather some useful properties of the entries, from the symmetries of the system and the parity of the basis.  

\begin{lem}\label{lem-entry-impor}
For $|\xi|\leq \xi_0,\, |\vep|\leq \tilde{\vep}_2,$
let \begin{align*}
\rD_{\xi}^{\vep}=\big(d_{j\ell}(\vep,\xi)\big)_{1\leq j,\ell\leq 4}  \, .
\end{align*}
It holds that 
\beq\label{propofentry}
\begin{aligned}
&d_{j\ell}\in \i\,\mR && {\textrm{if } j+\ell \textrm{ is even}},\\
&d_{j\ell}\in \mR && { \textrm{if } j+\ell \textrm{ is odd}}.
\end{aligned}
\eeq
In particular, it holds that $ \mathrm{ for }\,\, j,\ell\in \{1,3\}$ or $j,\ell\in \{2,4\} ,$
\begin{align}\label{entryimaginary}
    d_{j\ell}(\vep,\xi)=\i \,\kpa\, \xi \, m_{j\ell}(\vep,\xi) 
\end{align}
and for $(j,\ell)\in( \{2,4\}, \{1,3\})$, 
\begin{align}\label{entryxi2}
d_{j\ell}=-\kpa^2\xi^2 m_{j\ell}(\vep,\xi).
\end{align}
Here $m_{j\ell}(\vep,\xi)\in \mR$ are smooth with
\beq \label{m-off-diag}
m_{j\ell}(\vep,\xi)=\cO(\vep),  \qquad  (j,\ell)\in (\{1,2\}, \{3,4\} ) \, \mathrm{ or } \, ( \{3,4\}, \{1,2\}).
\eeq 
\end{lem}
\begin{proof}
    By  construction (see Proposition \ref{prop-kernel}), 
    \begin{align*}
         q_1^{\vep}(0,\cdot), \,  q_3^{\vep}(0,\cdot) =
        \bpm \mathrm{ odd } \\ \mathrm{ even }\epm \in \mR^2 , \quad \,  q_2^{\vep}(0,\cdot),\, q_4^{\vep}(0,\cdot) =\bpm \mathrm{ even } \\ \mathrm{ odd }\epm \in \mR^2
    \end{align*}
    Therefore,
    \begin{align*}
       \cI q_j^{\vep}(0,\cdot)=(-1)^j\, q_j^{\vep}(0,\cdot)\, ,
    \end{align*}
where the reversibility operator $\cI$ is defined in \eqref{reversibilityop}. Thanks to \eqref{reversibilitycom}, this property is preserved under the action of the extension operator so that  
\begin{align*}
       \cI q_j^{\vep}(\xi,\cdot)=(-1)^j q_j^{\vep}(\xi,\cdot), \qquad \forall \, |\xi|\leq \ep_0\,.
    \end{align*}
The other piece of information is that 
\beqs 
A_{\xi}^{\vep} \cI=\cI A_{\xi}^{\vep}
\eeqs
which is readily checked from the parity of the background waves. 
Thanks to the above two identities, we compute 
\beq\label{identity-entries}
\begin{aligned}
    \langle Jq_j^{\vep}(\xi,\cdot), L_{\xi}^{\vep} q_{\ell}^{\vep}(\xi,\cdot) \rangle= \langle q_j^{\vep}(\xi,\cdot), A_{\xi}^{\vep} q_{\ell}^{\vep}(\xi,\cdot) \rangle
    &=(-1)^{j} \langle \cI q_j^{\vep}(\xi,\cdot), A_{\xi}^{\vep} q_{\ell}^{\vep}(\xi,\cdot) \rangle\\
    &=(-1)^{j} \, \overline{\langle  q_j^{\vep}(\xi,\cdot), \cI A_{\xi}^{\vep} q_{\ell}^{\vep}(\xi,\cdot)  \rangle}\\
    &= (-1)^{j+\ell} \, \overline{\langle  q_j^{\vep}(\xi,\cdot),  A_{\xi}^{\vep} q_{\ell}^{\vep}(\xi,\cdot)  \rangle}.
\end{aligned}
\eeq
Consequently, 
\beqs 
\langle Jq_j^{\vep}(\xi,\cdot), L_{\xi}^{\vep} q_{\ell}^{\vep}(\xi,\cdot) \rangle \in \mR \quad \mathrm{ for } \,\, j+\ell \,\,\mathrm{ even };\,\, \in \i\mR \,\,
  \mathrm{ for } \, j+\ell \, \,\mathrm{ odd }.
\eeqs
The property \eqref{propofentry} then follows from 
the identity $d_{j\ell}=\langle \tilde{q}_j^{\vep}(\xi,\cdot), L_{\xi}^{\vep} q_{\ell}^{\vep}(\xi,\cdot) \rangle$ and the definitions of $\tilde{q}_j^{\vep}$
in \eqref{def-dualbasis}.

   Next, since $L_0^{\vep}q_{1}^{\vep}(0,\cdot)=L_0^{\vep}q_{3}^{\vep}(0,\cdot)=0$, we deduce from \eqref{identity-entries} that 
\beqs 
  \langle Jq_j^{\vep}(0,\cdot), L_{0}^{\vep} q_{\ell}^{\vep}(0,\cdot) \rangle=0 \quad \mathrm{ if } \,\, j\, \mathrm{ or }\, \,\ell \in \{1,3\},
\eeqs
which, together with \eqref{propofentry}, leads to \eqref{entryimaginary}. 
Moreover, when $j\in \{1,3\}$ and $\ell\in \{1,3\}$, by using $L_0^{\vep}q_{\ell}^{\vep}(0,\cdot)=0$,
$(L_0^{\vep})^*Jq_{j}^{\vep}(0,\cdot)=0$,
it holds that
\begin{align*}
    \f{1}{\i \kpa}\f{\d}{\d \xi} \big\langle Jq_j^{\vep}(\xi,\cdot), L_{\xi}^{\vep} q_{\ell}^{\vep}(\xi,\cdot) \big\rangle|_{\xi=0}&= \big\langle Jq_j^{\vep}(0,\cdot), 
    L_{1}^{\vep} q_{\ell}^{\vep}(0,\cdot) \big\rangle,
\end{align*}
 where
\begin{align*}
 L_1^{\vep}\coloneqq\f{1}{\i\kpa}\p_{\xi}L_{\xi}^{\vep}\,\big|_{\xi=0}\,.
\end{align*}
Since $L_1^{\vep}q_{\ell}^{\vep}=(\mathrm{odd}, \mathrm{even})^t,\, Jq_j^{\vep}(0,\cdot)=(\mathrm{ even }, \mathrm{ odd })^t, $
we conclude that 
\begin{align*}
    \big\langle Jq_j^{\vep}(\xi,\cdot), L_{\xi}^{\vep} q_{\ell}^{\vep}(\xi,\cdot) \big\rangle=\cO(\kpa^2\xi^2), \qquad j,\ell\in \{1,3\},
\end{align*}
which leads to \eqref{entryxi2}.  Finally,  by Corollary \ref{cor-Dxi0}, it holds that 
\begin{align*}
    d_{j\ell}(\vep,\xi)=\cO(\vep), \quad \forall\, (j,\ell)\in (\{1,3\}, \{2,4\} ) \, \mathrm{ or } \, ( \{2,4\}, \{1,3\}).
\end{align*}
This leads to \eqref{m-off-diag}. 
\end{proof}

The following lemma will be useful later. 
\begin{lem}\label{lemd1423}
    It holds that 
\begin{align*}
    d_{14}(\vep,\xi)=\cO(\vep\kpa^2\xi^2), \quad    d_{32}(\vep,\xi)=\cO(\vep\kpa^2\xi^2).
 \end{align*}
\end{lem}
\begin{proof}
 By Corollaries \ref{corD0ep} and \ref{cor-Dxi0},
 we have
  \begin{align*}
     d_{14}(0,\xi) =  d_{14}(\vep,0)= d_{32}(0,\xi) =  d_{32}(\vep,0)=0.
  \end{align*}
  Consequently, it suffices to show that 
  \begin{align*}
  (\p_{\xi} d_{14}) (\vep,0)
  =(\p_{\xi} d_{32})(\vep,0)=0.
  \end{align*}
 We prove the first identity; the second follows by a similar argument.

Recall that
  \begin{align*}
    d_{14}(\vep,\xi)&=a_{11}(\vep)\,\langle J q_2^{\vep}(\xi,\cdot), L_{\xi}^{\vep}q_4^{\vep}(\xi,\cdot)  \rangle+a_{12}(\vep)\,\langle Jq_4^{\vep}(\xi,\cdot), L_{\xi}^{\vep}q_4^{\vep}(\xi,\cdot)  \rangle.
  \end{align*}
We compute 
\begin{align*}
   & \f{\d}{\d \xi}\big\langle J q_2^{\vep}(\xi,\cdot), L_{\xi}^{\vep}q_4^{\vep}(\xi,\cdot)  \big\rangle\big|_{\xi=0}\\
    &=\i\, \kpa \bigg(\big\langle J q_2^{\vep}(0,\cdot), L_{1}^{\vep}q_4^{\vep}(0,\cdot)  \big\rangle+ \big\langle J \cU_1^{\vep}q_2^{\vep}(0,\cdot), L_{0}^{\vep}q_4^{\vep}(0,\cdot)  \big\rangle+\big\langle J\cU_1^{\vep} q_2^{\vep}(0,\cdot), L_{0}^{\vep}q_4^{\vep}(0,\cdot)  \big\rangle\bigg),
\end{align*}
 where $\cU_1^{\vep}\coloneqq \f{1}{\i \kpa} \p_{\xi}\cU^{\vep}|_{\xi=0}$.
Since $q_2^{\vep}(0,\cdot), q_4^{\vep}(0,\cdot)$ are real valued and both $L_1^{\vep},\, \cU_1^{\vep}$ preserve real-valuedness { (see Lemma \ref{lem-U1-real})}, it follows that
\begin{align*}
    \f{\d}{\d \xi}\big\langle J q_2^{\vep}(\xi,\cdot), L_{\xi}^{\vep}q_4^{\vep}(\xi,\cdot)  \big\rangle\big|_{\xi=0}\in \i\,\mR.
\end{align*}
The same property holds for $ \f{\d}{\d \xi}\big\langle J q_4^{\vep}(\xi,\cdot), L_{\xi}^{\vep}q_4^{\vep}(\xi,\cdot)  \big\rangle\big|_{\xi=0}$. Therefore, 
since $a_{11}(\vep), a_{12}(\vep)\in \mR$, we 
conclude that $(\p_{\xi} d_{14})(\vep, 0)\in \i\,\mR$.
However, according to \eqref{propofentry}, we also have 
$(\p_{\xi} d_{14})(\vep, 0)\in \mR$, and thus it must be that $(\p_{\xi} d_{14})(\vep, 0)=0$.
\end{proof}

\subsubsection{Transformation of the representation matrix}

In light of the properties of the entries  $d_{j\ell}({\vep},\xi)$ established in Lemma \ref{lem-entry-impor}, we begin by applying a transformation which reduces the problem to the study of a real-valued matrix.
Define 
\begin{align}\label{def-newmatrix}
   \mathrm {M}_{\xi}^{\vep}\coloneqq \frac{1}{\i \,\kpa \, \xi}\rP_{\xi}  \rD_{\xi}^{\vep} \rP_{\xi}^{-1}=(m_{j\ell}(\vep,\xi))_{1\leq j, \ell\leq 4}, \qquad \mathrm{\,with\,}\,\, \rP_{\xi}=\diag\, (\i\,\kpa \xi, 1, \i\, \kpa\xi, 1).
\end{align}
It is easy to see that 
\beq \label{spectrumrelation}
\sigma(\rD_{\xi}^{\vep})=\i \, \kpa\xi \,\sigma(M_{\xi}^{\vep}).
\eeq
By Lemma \ref{lem-entry-impor}, it holds that all the new entries $m_{j\ell}(\vep,\xi)$ are real. 
We first collect  some properties for them in the next two lemmas.
\begin{lem}[Entries of the diagonal blocks]\label{lem-diagentry}$ $

\textnormal{(1).}
It holds that 
\begin{align}\label{diff-digentry}
 m_{11}-m_{22}= \frac{d_{11}-d_{22}}{\i\,\kpa\,\xi}=\cO(\xi^2+\vep^2), \qquad  m_{33}-m_{44}=\frac{d_{33}-d_{44}}{\i\,\kpa\,\xi}=\cO(\xi^2+\vep^2).
\end{align}

\textnormal{(2).}
There is $c_1, c_2\in\mR$, such that
\begin{align}\label{exp-d1133}
    m_{11}=
    1-e_{*}+c_1\vep+\cO(\vep^2+\xi^2),
    \quad m_{33}=
    1+c_2\vep+\cO(\vep^2+\xi^2).
\end{align}

\textnormal{(3).}
For the other entries in the diagonal blocks,
\beq\label{m12-21}
\begin{aligned}
& m_{12}=d_{12}=- \f12w_1''(0)\xi^2+ \f{2\kpa \,k_2}{c \p_c \kpa}\vep^2
+\cO(\xi^2(\xi^2+\vep)), \quad m_{34}=d_{34}=-1+\cO(\xi^2+\vep^2), \\
&m_{21}=-\frac{d_{21}}{\kpa^2\xi^2} =- \f{1}{2\kpa^2}w_1''(0)+\cO(\xi^2+|\vep|),
\qquad\qquad  m_{43}=-\frac{d_{43}}{\kpa^2\xi^2}=-\alpha+\cO(\xi^2+|\vep|).
\end{aligned}
\eeq
\end{lem}
\begin{proof}
(1). By the definition of the dual basis in \eqref{def-dualbasis}, 
we have 
\begin{align*}
    d_{11}(\vep,\xi)&=a_{11}\,\langle q_2^{\vep}(\xi,\cdot), A_{\xi}^{\vep}q_1^{\vep}(\xi,\cdot)  \rangle+a_{12}\,\langle q_4^{\vep}(\xi,\cdot), A_{\xi}^{\vep}q_1^{\vep}(\xi,\cdot)  \rangle,\\
   d_{22}(\vep,\xi) &=a_{21}\,\langle q_1^{\vep}(\xi,\cdot), A_{\xi}^{\vep}q_2^{\vep}(\xi,\cdot)  \rangle+ a_{22}\,\langle q_3^{\vep}(\xi,\cdot), A_{\xi}^{\vep}q_2^{\vep}(\xi,\cdot)  \rangle.
\end{align*}
  In view of relations \eqref{identity-entries}
and \eqref{expan-aij}, it follows that
\begin{align*}
    d_{11}-d_{22}&=a_{12}\langle q_4^{\vep}(\xi,\cdot), A_{\xi}^{\vep}q_1^{\vep}(\xi,\cdot)  \rangle+a_{22}\langle q_2^{\vep}(\xi,\cdot), A_{\xi}^{\vep}q_3^{\vep}(\xi,\cdot)  \rangle . 
\end{align*}
The first identity in \eqref{diff-digentry} then stems from the facts that $$ A_{0}^{\vep}q_1^{\vep}(0,\cdot)=A_{0}^{\vep}q_3^{\vep}(0,\cdot)=0,\,\quad a_{12}=\cO(\vep^2), \quad  a_{22}=\cO(\vep)$$ 
together with
\begin{align*}
\f{1}{\i \kpa}\f{\d}{\d \xi} \big\langle q_2^{\vep}(\xi,\cdot), A_{\xi}^{\vep}q_3^{\vep}(\xi,\cdot) \big \rangle\big|_{\vep=\xi=0}=\big\langle q_2^0(0,\cdot), \big( A_0^0  \cU_1^{0}+A_1^0\big)q_3^0(0,\cdot)\big\rangle =0,
\end{align*}
whose vanishing follows from the fact that 
$ A_1^\vep\coloneqq\f{1}{\i\kpa}\p_{\xi}A_{\xi}^{\vep}\big|_{\xi=0}$
and $\cU_1^{\vep}\coloneqq\f{1}{\i\kpa}\p_{\xi}\cU^{\vep}(0)$ 
 do not shift the Fourier modes.

The second identity in \eqref{diff-digentry} can be shown in a similar manner, we omit the proof. 

(2).
The expansions in  \eqref{exp-d1133} follow from  direct computations and \eqref{entry-xi=0}.

(3). These identities can be derived from Corollaries \ref{corD0ep} and \ref{cor-Dxi0}, and the fact that the relevant coefficients are all real.
\end{proof}

\begin{lem}[Entries of the off-diagonal blocks]
All the entries of the off-diagonal $2\times 2 $ blocks are of order $\cO(\vep)$ and in particular
\begin{align}\label{diag-offdiag}
m_{14}(\vep,\xi),\, m_{32}(\vep,\xi)\, =\cO(\vep \kpa^2\xi^2), 
\end{align}
\beq \label{entry-important}
\begin{aligned}
m_{13}=-\kpa \big(1+\f{\kpa \, e_{*}}{\alpha\, \sh(2\kpa)} \big) \vep+\cO(\vep^2+\kpa^2\xi^2), \qquad m_{42}=\f{\alpha\,\sh(2\kpa)}{2} m_{13}+\cO(\vep^2).
\end{aligned}
\eeq
\begin{proof}
The fact that all entries of the off-diagonal $2 \times 2$ blocks are of order $\cO(\vep)$ follows directly from \eqref{m-off-diag}. 
By definition \eqref{def-newmatrix}, we have $(m_{14}, m_{32}) = (d_{14}, d_{32})$.
Therefore, \eqref{diag-offdiag} follows as a consequence of 
Lemma \ref{lemd1423}.
By \eqref{def-newmatrix} and the fact 
that $a_{41}=-a_{12}=\cO(\vep^2)$ (see \eqref{expan-aij}),
\begin{align*}
m_{42}=(\i\kpa\xi)^{-1}{d_{42}}&=(\i\kpa\xi)^{-1}\big\langle a_{42}(\vep) Jq_3^{\vep}(\xi,\cdot) +a_{41}(\vep) J q_1^{\vep}(\xi,\cdot), L_{\xi}^{\vep}q_2^{\vep}(\xi,\cdot) \big \rangle\\
&=-a_{42}\,\overline{(\i\kpa\xi)^{-1}\langle  J q_2^{\vep}(\xi,\cdot),  L_{\xi}^{\vep}q_3^{\vep}(\xi,\cdot) \rangle }+\cO(\vep^2)\\
&=-\f{a_{42}}{a_{11}} m_{13}+\cO(\vep^2)
\end{align*}
which, by \eqref{expan-aij},  leads to the second identity in \eqref{entry-important}. 
It now remains to compute $m_{13}$. Using the 
expansions \eqref{expan-aij}, the fact  $(L_0^{\vep})^{*}J q_2^{\vep}=\cO(\vep^2),\, L_0^{\vep}q_3^{\vep}(0,\vep)=0,$
we conclude that
\beq\label{comp-m13}
\begin{aligned}
    m_{13}&=(\i\kpa\xi)^{-1}{d_{13}}= \f{1}{\i \kpa} \f{\d}{\d{\xi}}\big\langle a_{11}(\vep) Jq_2^{\vep}(\xi,\cdot),  L_{\xi}^{\vep}q_3^{\vep}(\xi,\cdot) \big\rangle\big|_{\xi=0} +\cO(\kpa^2\xi^2+\vep^2)\\
    &= a_{11}(\vep) \big\langle  Jq_2^{\vep}(0,\cdot),  L_{1}^{\vep}q_3^{\vep}(0,\cdot) \big\rangle+\cO(\kpa^2\xi^2+\vep^2)\\
    &= \frac{-1}{\pi \sh(2\kpa)}\bigg( \langle Jq_2^0(0,\cdot), L_1^1 q_3^0(0,\cdot)+L_1^0q_3^1(0,\cdot)\rangle+\langle Jq_2^1(0,\cdot), L_1^0q_3^0(0,\cdot)\rangle\bigg)\vep+\cO(\kpa^2\xi^2+\vep^2)
\end{aligned}
\eeq
where 
\beqs 
q_j^{1}=\p_{\vep} q_j^{\vep}|_{\vep=0} ,  \qquad 
L_{\ell}^k=\p_{\vep}^k \big(\f{1}{\i\kpa}\p_{\xi}\big)^{\ell} L_{\xi}^{\vep}\big|_{\xi=\vep=0}. 
\eeqs
By the expansion of the Dirichlet--Neumann operator (see Proposition \ref{prop-DN}), we find
\begin{align*}
    L_1^0=\bpm 1 & G_1^0\\ 2\beta\kpa \p_x & 1 \epm, \quad L_1^1= \bpm -v^1 & G_1^1\\ 0 & -v^1 \epm 
\end{align*}
where $v^1=\kappa \p_{\vep} \big(\p_x\vp_{\vep}-Z_{\vep}\p_x\zeta_{\vep}\big)|_{\vep=0}=\kpa \ch(\kpa) \cos(x)$ and 
\begin{align*}
  G_1^0=-\i \tah(\kpa D)-\f{\kpa\p_x}{\ch^2(\kpa D)} , \qquad G_1^1=-\big(G_0^0\,\zeta^1G_1^0+G_1^0\,\zeta^1 G_0^0\big)-\kpa \big(\p_x(\zeta^1\cdot)+\zeta^1\p_x\big)
\end{align*}
with $G_0^0=G_0[0]=|\kpa D| \tah(|\kpa D|), $  and $\zeta^1=\p_{\vep} \zeta_{\vep}|_{\vep=0}=\sh(\kpa)\cos (x)$.  
Using the expansion of the basis $q_2^{\vep}(0,\cdot),q_3^{\vep}(0,\cdot)$ in Corollary \ref{cor-exp-0},
it then follows from 
direct computations that 
\begin{align*}
\langle Jq_2^0(0,\cdot), L_1^1 q_3^0(0,\cdot)\rangle& = \left\langle \bpm \ch(\kpa) \sin(x)\\-\sh(\kpa)\cos (x) \epm, \, \kpa \bpm \sh(\kpa)\sin(x) \\ -\ch(\kpa)\cos(x) \epm  \right\rangle \\
&=\pi \kpa \sh(2\kpa), \\
\langle Jq_2^0(0,\cdot), L_1^0 q_3^1(0,\cdot)\rangle& =  \left\langle  
\bpm \ch(\kpa) \sin(x)\\-\sh(\kpa)\cos (x) \epm, \,
\Lambda_0 
\bpm \f{\kpa}{\ch(\kpa)}\sin(x)\\
(-2\beta\kpa\sh(\kpa)+\ch(\kpa)) \cos (x) \epm \right\rangle \\
&=\pi \f{\kpa^2}{2\alpha}\bigg( \f{2\kpa}{\sh(2\kpa)} +2\beta\kpa\tah(\kpa)-1\bigg)\, ,\\
\langle Jq_2^1(0,\cdot), L_1^0q_3^0(0,\cdot)\rangle&=\left\langle  
\bpm * \\  \f{\kpa^2}{2\alpha}+ *\cos(x) \epm , 
\bpm 0 \\ 1\epm \right\rangle=\f{\pi\kpa^2}{\alpha},
\end{align*}
where we use $*$ to denote irrelevant functions or constants in the final computation.
Substituting the above three expressions into \eqref{comp-m13}, we find \eqref{entry-important}.
\end{proof}
    
\end{lem}
\subsubsection{Eigenvalues of the transformed matrix $\rM_{\xi}^{\vep}$--proof of Theorem \ref{thm-modulation}}

Given the properties of the entries $\{m_{j\ell}\}$ detailed in the preceding subsection, we are now in a position to analyze the eigenvalues of the transformed matrix $\rM_{\xi}^{\vep}$. 
Denote 
\begin{align}
    \rM_{\xi}^{\vep}=\bpm  \rM_{11} & \vep \rM_{12} \\
    \vep \rM_{21} & \rM_{22}\epm (\vep,\xi) 
\end{align}
where $\rM_{jk}(\vep,\xi),\, 1\leq j, k\leq 2$ are 
$2\times 2$ matrices. 

By \eqref{spectrumrelation}, to prove Theorem \ref{thm-modulation}, it suffices to show the following:
\begin{thm}\label{thm-modulation-matrix}
Let $(\alpha,\beta)\in \rm I\cup II\cup III$. Let $\xi_0, c_0$ be the same as in Theorem \ref{thm-modulation}.
If $(\alpha,\beta)$ is such that the index  $\tilde{C}(\alpha,\beta)$ defined in \eqref{def-index-real} is positive, then there exists $\vep_{2}>0, \xi_2\in(0,\xi_0]$ such that 
for any $|\vep|\leq \vep_2, |\xi|\leq \xi_2$, it holds that
\[
\displaystyle
\sigma\left(\rM_{\xi}^{\vep}\right)\cap B_{c_0}(0)\subset \mathbb{R}\,.
\]
On the other hand, if  $\tilde{C}(\alpha,\beta)<0$ then there is a constant $A=A(\alpha,\beta)$, such that if $|\xi|<A|\vep|\ll 1$, then
$\sigma\big(\rM_{\xi}^{\vep}\big) \cap B_{c_{0}}(0)$ contains two  real eigenvalues and two eigenvalues with nonzero imaginary parts, whose absolute values are of order $\mathcal{O}(|\vep|)$. 

\end{thm}

The idea to get the full description of the eigenvalues of  $\rM_{\xi}^{\vep}$ is to introduce a transformation to make one of the off-diagonal 
$2\times 2$ blocks $\vep\rM_{12}$ or $\vep\rM_{21}$ vanish, whose existence is ensured 
as long as the eigenvalues of $\rM_{11}$ and $\rM_{22}$ are well-separated. 
\begin{prop}\label{prop-separation}
  Let ${X}_1^{\pm}$ and ${X}_2^{\pm}$  be the eigenvalues of $\rM_{11}(\vep,\xi)$ and $\rM_{22}(\vep,\xi)$. Let $(\alpha,\beta)\in \rm I\cup III$
  or $(\alpha,\beta)\in\rm II$ be such that $0<e_{*}\neq \sqrt{\alpha}$ where $e_{*}=w_1'(0)/\kpa$. 
  There is $\ep_2>0$ such that for any 
   $|(\vep,\xi)|\leq \ep_2$, 
   there exists $\iota=\iota(\alpha,\beta,\ep_2)>0$
   for which 
   \begin{align}\label{separation}
 |X_1^{\mu}-X_{2}^{\pm}|>\iota \qquad (\mu\in \{+, -\} ).
\end{align}
Moreover, $X_{2}^{\pm}\in \mR$ and $$\Im X_1^{\pm}\neq 0, \qquad \mathrm{ when } \, \,|\xi|
< \f{2\kpa \, k_2}{c\p_c\kpa\, w_1''(0)}|\vep|\ll 1 .$$
\end{prop}
\begin{proof}
Direct computations show that 
\begin{align*}
    X_1^{\pm}=\f12 \big({m_{11}+m_{22}}\pm \sqrt{\Delta_1}\big), \qquad X_2^{\pm}=\f12 \big({m_{33}+m_{44}}\pm \sqrt{\Delta_2}\big)
\end{align*}
with 
\begin{align}
     \Delta_1=\big({m_{11}-m_{22}}\big)^2+4m_{12}\,m_{21}, \qquad 
   \Delta_2= \big({m_{33}-m_{44}}\big)^2+4 \, m_{34}\,m_{43}\,.
\end{align}
It then follows from Lemma \ref{lem-diagentry}
that 
\begin{align*}
   X_1^{\pm}-(1-e_{*})= {\cO\big(|\vep|+|\xi|\big)},
   \qquad X_2^{\pm}-\big(1\pm \sqrt{\alpha}\big)=\cO\big(|\vep|+\xi^2\big)\in \mR
\end{align*}
so that 
\begin{align*}
   X_1^{\mu}-X_{2}^{\pm}=-\big(e_{*}\pm\sqrt{\alpha}\big)+
   \cO\big(|\vep|+\xi^2\big), \quad (\mu\in \{+, -\} ).
\end{align*}
Therefore, \eqref{separation} holds for sufficiently small  $|\vep|$ and $|\xi|$ as long as $e_{*}\neq \sqrt{\alpha}$.
As shown in Corollary \ref{cor-Dxi0}, 
this condition holds 
for any $(\alpha,\beta)\in \rm I\cup II$, (it holds that $e_{*}>1\geq \sqrt{\alpha}$ ).

Finally, it follows from \eqref{diff-digentry} and \eqref{diag-offdiag} that 
\begin{align*}
    \Delta_1=2\kpa^{-2} w_1''(0) \bigg( \f12 w_1''(0)\xi^2-\f{2\kpa \, k_2}{c\p_c\kpa}\vep^2\bigg)+\cO(\xi^4+ |\vep|^3+|\vep|\xi^2)<0
\end{align*}
for those $(\vep,\xi)$ satisfying
$|\xi|
< \f{4\kpa \, k_2}{c\p_c\kpa\, w_1''(0)}|\vep|\ll 1$. 
\end{proof}
\begin{rmk}\label{rmk-separation-eigen}
    From the above proof, we see easily that if $e_{*}\neq \pm \sqrt{\alpha}$
    \begin{align}\label{entrysep}
       m_{jj}-X_2^{\pm}=
        -(e_{*}\pm\sqrt{\alpha})+\cO(|\vep|+\xi^2)\neq0, \qquad (j=1,2)
    \end{align}
    and thus $\rM_{11}-X_{2}^{\pm}\,\mathrm{Id}_{2\times2}$ is invertible 
    for $|(\vep,\xi)|\leq \ep_2$ by choosing $\ep_2$ smaller if necessary.
\end{rmk}

\begin{rmk}
Recall that the off-diagonal matrices $\vep \rM_{12}$ and $\vep\rM_{21}$ are of order $\cO(\vep)$. Therefore, we can expect the eigenvalues of $\rM_{\xi}^{\vep}$ to differ from $X_1^{\pm}$ and $X_2^{\pm}$ by terms of order $\cO(\vep)$. Since the imaginary part of $X_1^{\pm}$ is also of order $\cO(\vep)$, it is possible that the eigenvalues of $\rM_{\xi}^{\vep}$ are purely real. Our subsequent analysis shows that this can indeed happen. 
\end{rmk}

We now present the proof of Theorem \ref{thm-modulation-matrix}, which requires a detailed analysis of the eigenvalues of the matrix $\rM_{\xi}^{\vep}$.

\begin{proof}[Proof of Theorem \ref{thm-modulation-matrix}]

The analysis consists of two steps. In the first step, we apply a transformation to diagonalize $\rM_{22}$, and in the second step, we perform another transformation to eliminate the upper off-diagonal part $\rM_{12}$. 

\textbf{Step 1: Diagonalization of $\rM_{22}$.}
Let $\rG$ be the similarity transformation matrix
associated to $\rM_{22}$, that is 
\begin{align*}
    \rM_{22}=\rG \bpm X_2^{+} & 0\\ 0 & X_2^{-}\epm \rG^{-1}\, .
\end{align*}
Direct computations show that 
\begin{align*}
    \rG= \bpm 1 & 1 \\ -\sqrt{\alpha} & \sqrt{\alpha}\epm +\cO\big(|\vep|+\xi^2\big), \qquad  \rG^{-1}= \f12 \bpm 1 & -\f1{\sqrt{\alpha}} \\ 1 & \f1{\sqrt{\alpha}}\epm +\cO\big(|\vep|+\xi^2\big)\,.
\end{align*}
We introduce 
\begin{align*}
    \rN_{\xi}^{\vep}=\bpm \mathrm{Id}_{2\times2} & 0 \\ 0 & \rG^{-1} \epm \rM_{\xi}^{\vep} \bpm \mathrm{Id}_{2\times2} & 0 \\ 0 & \rG \epm = \bpm \rN_{11}& \vep \rN_{12}\\\vep \rN_{21} &  \rN_{22}\epm (\vep, \xi) 
\end{align*}
with 
\beq \label{def-entry-N}
\rN_{11}=\rM_{11}, \,  \quad  \rN_{12}=\rM_{12}\rG\, , \, \quad \rN_{21}= \rG^{-1}\rM_{21},  \quad 
\rN_{22}=\diag \, (X_2^{+}, X_2^{-}).
\eeq 

\begin{rmk}
   Motivated by Proposition \ref{prop-separation}, it is natural to consider diagonalizing both $\rM_{11}$ and $\rM_{22}$ simultaneously, in order to take advantage of the separation of their eigenvalues. However, since the diagonal entries of $\rM_{11}$ are already close to $X_1^{\pm}$ within an error of $\cO(|\vep| + |\xi|^2)$, it is not necessary to further diagonalize $\rM_{11}$. See Remark \ref{rmk-separation-eigen}.
\end{rmk}

\textbf{Step 2: Elimination of $\rN_{12}$.}
In this step, our goal is to find a similarity transformation that eliminates the upper off-diagonal block $\rN_{12} = \rM_{12}\rG$. The following proposition guarantees the existence of such a transformation. 

\begin{prop}
There exists a $2\times 2$ matrix $\rQ$ such that
\begin{align}\label{sectransform}
\tilde{\rN}_{\xi}^{\vep}\coloneqq 
\bpm \Id & -\vep \rQ \\
0 & \Id \epm  {\rN}_{\xi}^{\vep}  \bpm \Id & \vep \rQ \\
0 & \Id \epm = \bpm \rN_{11}-\vep^2 \rQ \rN_{21} & 0\\ \vep \rN_{21} & \rN_{22}+\vep^2 \rN_{21} \rQ \epm 
\end{align}
\end{prop}
\begin{proof}
    By straightforward computation, it suffices to find $\rQ$, such that 
    \beq\label{idtosolve}
    \begin{aligned}
     \big(\rN_{11} \rQ-\rQ\rN_{22}+\rN_{12}\big)-\vep^2\rQ \rN_{21}\rQ=0\, .
    \end{aligned}
    \eeq
Define $\rB_{\pm}\coloneqq\rN_{11}-X_2^{\pm}\,\mathrm{Id}_{2\times 2}$ and
write $\rQ=(\rV_1, \rV_2)$, $\rN_{12}=(\rU_1, \rU_2)$ where $\rV_1, \rV_2; \rU_1, \rU_2$ denote the column vectors of $\rQ$ and $\rN_{12}$, then \eqref{idtosolve} is equivalent to 
\beq
\begin{aligned}\label{eq-IFT}
   F_{+}(\rV_1, \rV_2)\coloneqq \rB_{+} \rV_1+ \rU_1-\vep^2 \rQ \rN_{21} \rV_1=0,\\
    F_{-}(\rV_1, \rV_2): = \rB_{-} \rV_2+ \rU_2-\vep^2 \rQ \rN_{21} \rV_2=0.
\end{aligned}
\eeq
By Remark \ref{rmk-separation-eigen}, the real-valued matrices $\rB_{\pm}$ are invertible. Consequently, we apply the inverse function theorem to the map $F\colon \mR^4\rightarrow \mR^4$
defined by
\beqs 
F(\rV)=\bpm F_{+}(\rV_1, \rV_2) \\ F_{-} (\rV_1, \rV_2)\epm, \qquad \rV= \bpm \rV_1\\ \rV_2 \epm.
\eeqs
Then, for $\vep$ sufficiently small, say $|\vep|\leq \ep_3$, there exist vectors  $\rV_1,\rV_2$ such that 
\eqref{eq-IFT} holds and 
\begin{align*}
\rV_1=-\rB_{+}^{-1}\rU_1+\cO_{\mR^2}(\vep^2), \qquad \rV_2=-\rB_{-}^{-1}\rU_2+\cO_{\mR^2}(\vep^2).
\end{align*}
Here, $\cO_{\mathbb{R}^2}(\varepsilon^2)$ denotes terms valued in $\mathbb{R}^2$ whose absolute value is bounded by a constant multiple of $\varepsilon^2$.
 To summarize, we find the matrix $\rQ=-\bpm \rB_{+}^{-1}\rU_1 & \rB_{-}^{-1}\rU_2\epm +\cO_{\cM_{2\times 2}}(\vep^2)$
such that \eqref{idtosolve} holds.
\end{proof}

By \eqref{sectransform}, the eigenvalues of $\tilde{\rN}_{\xi}^{\vep}$ (thus also the  spectrum of $\rM_{\xi}^{\vep}$) are composed of the eigenvalues of the  $2\times 2$ matrices 
 $\rN_{11}-\vep^2 \rQ \rN_{21}$ and $\rN_{22}+\vep^2 \rN_{21} \rQ$. %
 By  definition \eqref{def-entry-N}
 it is easily seen that the eigenvalues of $\rN_{22}+\vep^2 \rN_{21} \rQ$ are 
 \begin{align}\label{tX2pm}
 \tilde{X}_2^{+} =
 =1+\sqrt{\alpha}+\cO(|\vep|+\xi^2), \quad \tilde{X}_2^{-} 
 =1-\sqrt{\alpha}+\cO(|\vep|+\xi^2).
 \end{align}
It suffices to compute the spectrum of the matrix  $\rN_{11}-\vep^2 \rQ \rN_{21}$. The key task is to identify the off-diagonal element in the upper right corner of the matrix $\rQ \rN_{21}$.
A straightforward calculation yields
\begin{align*}
  &  \rB_{\pm}^{-1}=\bpm \f{1}{m_{11}-X_2^{\pm}} & 0\\ \f{m_{21}}{(m_{11}-X_2^{\pm})(m_{22}-X_2^{\pm})} & \f{1}{m_{22}-X_2^{\pm}} \epm +\cO(\xi^2+\vep^2),\\
   & \rN_{12}=\bpm \rU_1 & \rU_2 \epm=\rM_{12}\rG = \f{1}{\vep} \bpm m_{13}-\sqrt{\alpha}\,m_{14} &  m_{13}+\sqrt{\alpha}\,m_{14}\\ m_{23}-\sqrt{\alpha}\,m_{24} &  m_{23}+\sqrt{\alpha}\,m_{24}\epm +\cO(|\vep|+\xi^2),\\
   & \rN_{21}=\rG ^{-1}\rM_{21}=\f{1}{2\vep}
   \bpm  m_{31}-\f{1}{\sqrt{\alpha}}m_{41}& m_{32}-\f{1}{\sqrt{\alpha}}m_{42} \\  m_{31}+\f{1}{\sqrt{\alpha}}m_{41} &   m_{32}+\f{1}{\sqrt{\alpha}}m_{42} \epm +\cO(|\vep|+\xi^2)\,.
\end{align*}
This implies that
\[
-\rQ \rN_{21}=\bpm n_{11} & n_{12} \\ n_{21} & n_{22} \epm
\]
with  the real numbers $n_{11}, n_{21}, n_{22}=\cO(1)$, and
\beqs
n_{12}=\f{(m_{13}-\sqrt{\alpha}m_{14})(m_{32}-\f{1}{\sqrt{\alpha}}m_{42})}{2\vep^2\,(m_{11}-X_2^{+})}+\f{(m_{13}+\sqrt{\alpha}m_{14})(m_{32}+\f{1}{\sqrt{\alpha}}m_{42})}{ 2\vep^2\,(m_{11}-X_2^{-})}.
\eeqs
It then follows from \eqref{diag-offdiag} and \eqref{entrysep} that 
\begin{align*}
n_{12}&=\f{m_{13}m_{42}}{2\sqrt{\alpha}\,\vep^2} \bigg(\f{1}{m_{11}-X_2^{-}}-\f{1}{m_{11}-X_2^{+}}\bigg)+\cO(\xi^2)\, \\
   &= -\f{m_{13}m_{42}}{\,\vep^2(e_{*}^2-\alpha)}+\cO(|\vep|+\xi^2)\,  \\
   &=-\f{\alpha \kpa^2 \sh(2\kpa)}{\,2(e_{*}^2-\alpha)}\big(1+\f{\kpa e_{*}}{\alpha \sh(2\kpa)}\big)^2+\cO(|\vep|+\xi^2).
\end{align*}
Consequently, the eigenvalues of 
 $\rN_{11}-\vep^2 \rQ \rN_{21}=\big(m_{j\ell}+\vep^2 n_{j\ell}\big)_{1\leq j,\ell\leq 2}$ read
\begin{align}\label{realeigen-up}
\tilde{X}_1^{\pm}=\f12\bigg(m_{11}+m_{22}+\vep^2(n_{11}+n_{22})+\sqrt{\tilde{\Delta}_1}\bigg) 
 \end{align}
with  
\begin{align*}
  \tilde{\Delta}_1= \big(m_{11}-m_{22}+\vep^2(n_{11}-n_{22})\big)^2 +4 \big(m_{21}+\vep^2 n_{21}\big)\big(m_{12}+\vep^2 n_{12}\big).
\end{align*}
By virtue of \eqref{diff-digentry} and \eqref{m12-21} it holds that 
\beq\label{criteria}
\begin{aligned}
   \tilde{\Delta}_1&=4\, m_{21}\big(m_{12}+\vep^2 n_{12}\big)+\cO(\xi^4+\vep^4) \\
   &= \kpa^{-2}w_1''(0) \bigg( w_1''(0)\xi^2+2C\vep^2\bigg)+\cO\big((\xi^2+\vep^2)(|\xi|+|\vep|)\big),
\end{aligned}
\eeq
where the constant 
\begin{align*}
C=C(\alpha,\beta)= \f{\alpha \kpa^2 \sh(2\kpa)}{\,2(e_{*}^2-\alpha)}\Big(1+\f{\kpa e_{*}}{\alpha \sh(2\kpa)}\Big)^2-\f{2\kpa \,k_2}{c \p_c \kpa}.
\end{align*}
Note that the wave number $\kpa$ is determined 
 by $(\alpha, \beta)$ from  the dispersion relation \eqref{disp-relation} and $k_2$ is defined in \eqref{def-k2}.

 By \eqref{realeigen-up} and \eqref{criteria}
 we see that if $\tilde{C}=\tilde{C}(\alpha,\beta)=w_1''(0)C(\alpha,\beta)>0$, then for $|(\vep,\xi)|$ sufficiently small,   $\tilde{\Delta}_1>0$ and thus $X_1^{\pm}\in \mR$. On the other hand, if $\tilde{C}<0,$
 then as long as $|\xi|\leq A |\vep| \ll 1$ with $A=-\f{ 2\tilde{C}}{w_1''(0)}$ then  $\tilde{\Delta}_1<0$ and 
$ |\Im \tilde{X}_1^{\pm}|=\cO(|\vep|)$.
This completes  the proof of Theorem \ref{thm-modulation-matrix}.
\end{proof}

\begin{rmk}\label{rmk-specurve}
It follows from the above analysis and \eqref{spectrumrelation}, \eqref{tX2pm}, \eqref{realeigen-up} that,
the spectrum of $\mathrm{D}_{\xi,\vep}$ is composed by, for $|\vep|,|\xi|$ small enough, 
\begin{align}
  &   \i \,\kappa\,\xi\, \tilde{X}_{2}^{\pm} 
    =  \i \,\kappa\,\xi \big(1\pm\sqrt{\alpha}+\cO(|\vep|+\xi^2)\big)\in \i \mR, \label{eigen1}\\
    &    \i \,\kappa\,\xi\tilde{X}_{1}^{\pm}=
\f{ \i \,\kappa\,\xi}{2} \bigg(\underbrace{m_{11}+m_{22}+\vep^2(n_{11}+n_{22})}_{\in\,\mR}+\sqrt{\tilde{\Delta}_1}\bigg) \label{eigen2}
\end{align}
where ${\tilde{\Delta}_1}=\cA_1  (\xi,\vep) \big(\cA_2(\xi,\vep)\xi^2-\cA_3(\xi,\vep)\vep^2\big)$ and 
\begin{align*}
     \cA_1= \kappa^{-2} w_1''(0)+\cO(\xi^2+\vep^2); \quad \,
  \cA_2= w_1''(0)+\cO(\xi^2+\vep^2); \quad \,
 \cA_3= -2C+\cO(\xi^2+\vep^2).
\end{align*}
Therefore, for a fixed but small amplitude $|\vep|$ as the Floquet parameter $\xi$ varies, the eigenvalues in \eqref{eigen1} remain on and move along the imaginary axis, while the eigenvalues in \eqref{eigen2} stay on the imaginary axis if $\tilde{C}=w_1''(0)C(\alpha,\beta)>0$ and
trace out an  “8”-shaped curve if $\tilde{C}<0.$
\end{rmk}

\appendix
\section{Wave profiles}

We summarize here some useful facts for the wave profiles $({\zeta}_{\vep}, {\vp}_{\vep})$ and the wave number which have essentially been proven in \cite[Appendix A]{Mariana-Tien-Erik}.
\begin{lem}\label{lem-waveprofile}
The wave number $k_{\vep}$ has the expansion
\beq
 k_{\vep}=k_0+\vep^2 k_2+\cO(\vep^4)
\eeq
where $k_0=\kappa$ is the unique simple root of the dispersion relation \eqref{disp-relation-1}
and 
\begin{align}\label{def-k2}
    k_2=
\frac{\kappa^3}{d(\kappa)}\tilde{\chi}
\end{align}
with 
\beq\label{def-tchi}
\begin{aligned}
  \tilde{\chi}&= (9\alpha\beta + 16)\kappa 
- 12\alpha\beta \kappa \ch(2\kappa) 
+ 3\alpha\beta \kappa \ch(4\kappa) \\
&\qquad - 8\alpha(2c(\kappa) - 1)\sh(2\kappa) 
- 4\alpha(c(\kappa) + 2)\sh(4\kappa)
\end{aligned}
\eeq
and
\begin{align*}
    c(\kappa)=-1-\frac{  \kappa\left( \ch(2\kappa) + 2 \right)}{D(2\kappa)}; \quad d(\kappa)=32\alpha \bigg( 2\beta\kappa\left( \ch(2\kappa) - 1 \right) + 2\kappa - \sh(2\kappa) \bigg).
\end{align*}
It holds that $k_2>0$ when $\alpha\in (0,1), \beta>0$ or $\alpha=1, \beta\in (0,\f13)$.
Moreover, the following expansion holds for the wave profiles:
\begin{align} \label{expan-profile}
& \zeta_{\vep}=\vep\zeta_1+\vep^2\zeta_2+\cO(\vep^3), \qquad \vp_{\vep}=\vep\vp_1+\vep^2\vp_2+\cO(\vep^3),
\end{align}
with 
\beq\label{def-expzetavp}
\begin{aligned}
    \zeta_1=\sh(\kappa)\cos (x), \quad \zeta_2=\frac{\kappa}{4}\big(1+c(\kappa)\big)\sh (2\kappa) \cos(2x)-\frac{\kappa^2}{4\alpha}, \\
    \vp_1=\ch(\kappa)\sin (x), \quad \vp_2=\frac{\kappa}{4}\big(c(\kappa)\ch(2\kappa)+2\sh^2(\kappa)\big)\sin (2x).
\end{aligned}
\eeq
\end{lem}

\section{Construction of the basis for the generalized kernel of $L_0^{\vep}$}\label{appen-kernel}

Recall the operator $L_0^{\vep}$ is defined as 
\beqs 
 L_0^{\vep}\coloneqq\left( \begin{array}{cc}
   \p_x(d_{\vep}\cdot)  &  G_0[\zeta_{\vep}]   \\[5pt]
 P_0[\zeta_{\vep}] -  w_{\vep}& d_{\vep}\p_x
\end{array}\right),
\eeqs
where $G_0[\zeta_{\vep}]$ and $P_0[\zeta_{\vep}]$ are defined in \eqref{def-GP}.
In this section, we prove Proposition \ref{prop-kernel}, concerning the construction of the basis for the generalized kernel of $L_0^{\vep}\,$.

\subsection{Determination of ${q}_1^{\vep}(0,\cdot), {q}_2^{\vep}(0,\cdot)$}

 To find ${q}_1^{\vep}(0,\cdot), {q}_2^{\vep}(0,\cdot)$, we will first find 
 two generalized eigenfunctions of the original linear operator (before using the good unknown) 
 \beq \label{relationLtL}
{\cL}_0^{\vep}\coloneqq\left( \begin{array}{cc}
   1   & 0 \\
   -Z_{\vep}   & 1
 \end{array}\right)^{-1}L_0^{\vep}\left( \begin{array}{cc}
   1   & 0 \\
   -Z_{\vep}   & 1
 \end{array}\right),
 \eeq
 and then transform to $L_0^\vep$.

 By definition, the background 1-d waves $(\zeta_{\vep}, \vp_{\vep})^t$ solve the equation
\beqs \label{profile eq}
    \left\{ \begin{array}{l}
      k_{\vep} \p_x \zeta_{\vep}+ G_0[\zeta_{\vep}]\varphi_{\vep}=0, \\
       k_{\vep}  \p_x\vp_{\vep}-\f12 (k_{\vep}\p_x\varphi_{\vep})^2+\f12 \f{(G_0[\zeta_{\vep}]\vp_\vep+k_{\vep}^2\p_x\vp_{\vep}\cdot\p_x\zeta_{\vep})^2}{1+|k_{\vep}\p_x\zeta_{\vep}|^2}+\beta k_{\vep}\p_x \big(\f{k_{\vep}\p_x\zeta_{\vep}}{\sqrt{1+(k_{\vep}\p_x\zeta_{\vep})^2}}\big)-\alpha\zeta_{\vep}=0\, . 
    \end{array}
    \right. 
\eeqs
First, by the translational invariance, we get by differentiating the above profile equations that 
\beqs 
{\cL}_0^{\vep}\begin{pmatrix}
     \p_x \zeta_{\vep}  \\
    \p_x \vp_{\vep} 
\end{pmatrix}
=0\, .
\eeqs
Note that $\zeta_\vep$, $\varphi_\vep$ also depend smoothly on the parameters $\alpha$, $\beta$.
To find the associated generalized kernel it is useful to think of $\alpha$, $\beta$ as functions of $h$, $c$ defined through the relation \eqref{eq-Froude-Weber}, and to use the change of variable 
\beqs 
X= hx, \quad Z=hz, \quad \tilde{\zeta}_{\vep}(X)=h\,\zeta_{\vep}(x), \quad \tilde{\vp}_{\vep}(X)=c\, h\,\vp_{\vep}(x)
\eeqs
and study the resulting  system where the wave speed $c$ appears explicitly: 
 \beqs 
    \left\{ \begin{array}{l}
     c k_{\vep} \p_X \tilde{\zeta}_{\vep}+ \tilde{G}_0[\tilde{\zeta}_{\vep}]\tilde{\varphi}_{\vep}=0, \\
      c k_{\vep}  \p_X\tilde{\vp}_{\vep}-\f12 (k_{\vep}\p_X\tilde{\varphi}_{\vep})^2+\f12 \f{(\tilde{G}_0[\tilde{\zeta}_{\vep}]\vp_\vep+k_{\vep}^2\p_X\tilde{\vp}_{\vep}\cdot\p_X\tilde{\zeta}_{\vep})^2}{1+|k_{\vep}\p_X\tilde{\zeta}_{\vep}|^2}+T k_{\vep}\p_X \big(\f{k_{\vep}\p_X\tilde{\zeta}_{\vep}}{\sqrt{1+(k_{\vep}\p_X\tilde{\zeta}_{\vep})^2}}\big)-g\tilde{\zeta}_{\vep}=0\, . 
    \end{array}
    \right. 
\eeqs
Denote by $\tilde{\cL}^{\vep}$ the linearized operator 
around $(\tilde{\zeta}_{\vep}, \tilde{\vp}_{\vep})$, with the expansion
$$\tilde{\cL}_{\xi}^{\vep}=\tilde{\cL}^{\vep}(\p_x+\i\,\xi)=\tilde{\cL}_0^{\vep}(\p_x)+\i\, k_{\vep}\xi\tilde{\cL}_{1}^{\vep}(\p_x)+\cO_{B(Y_{\per})}(\xi^2).$$
 Since $\tilde{\zeta}_{\vep}$ and $\tilde{\vp}_{\vep}$ depend smoothly on the wave speed $c$ and the amplitude $\vep$, we can
differentiate the above equation with respect to $c$ 
to find that 
\beq\label{eq-pc}
\begin{aligned}
&\tilde{\cL}_0^{\vep} \begin{pmatrix}
\p_{c} \tilde{\zeta}_{\vep}\\
\p_{c} \tilde{\vp}_{\vep}
\end{pmatrix}\\
 &\quad=\begin{pmatrix}
    ck_{\vep}\p_X\p_{c} \tilde{\zeta}_{\vep}+\tilde{G}_0[\tilde{\zeta}_{\vep}]\p_{c}\tilde{\varphi}-\tilde{G}_0[\tilde{\zeta}_{\vep}]\big(\tilde{Z}_{\vep}\,\p_{c}\tilde{\zeta}_{\vep}\big)-k_{\vep}\p_X\big(\tilde{v}_{\vep}\p_{c}\tilde{\zeta}_{\vep}\big)\\
    ck_{\vep}\p_X\p_{c} \tilde{\vp}_{\vep}-\tilde{v}_{\vep}k_{\vep}\p_X\p_{c} \tilde{\vp}_{\vep}+ \tilde{Z}_{\vep} \, \tilde{G}_0[\tilde{\zeta}_{\vep}]\big(\p_{c}\tilde{\vp}_{\vep}-\tilde{Z}_{\vep}\p_{c} \tilde{\zeta}_{\vep}\big)-%
\tilde
{Z_{\vep}}k_{\vep}\p_X \tilde{v}_{\vep} \p_{c} \tilde{\zeta}_{\vep}+(\tilde{P}_0[\tilde{\zeta}_{\vep}]-g) \p_{c} \tilde{\zeta}_{\vep}
\end{pmatrix}\\
& \quad =-k_{\vep}\p_X \begin{pmatrix}
 \tilde{\zeta}_{\vep}\\
 \tilde{\vp}_{\vep}
\end{pmatrix}-\p_c k_{\vep} \, \tilde{\cL}_1^{\vep}\p_X\begin{pmatrix}
 \tilde{\zeta}_{\vep}\\
 \tilde{\vp}_{\vep}
\end{pmatrix},
\end{aligned}
\eeq
where 
\begin{align*}
 \tilde{P}_0[\tilde{\zeta}_{\vep}]\coloneqq T k_{\vep}\p_X\bigg(
\f{k_{\vep}\p_X}{(1+(k_{\vep}\p_X\tilde{\zeta}_{\vep})^2)^{\f32}}\bigg)
\end{align*}
and 
\beqs
\tilde{Z}_{\vep}=\f{\tilde{G}_0[\tilde{\zeta}_{\vep}]\tilde{\varphi}_{\vep}+k_{\vep}^2\p_X\tilde{\vp}_{\vep}\cdot\p_X\tilde{\zeta}_{\vep}}{{1+(k_{\vep}\p_X\tilde{\zeta}_{\vep})^2}}, \qquad \tilde{v}_{\vep}=k_{\vep}\p_X \tilde{\vp}_{\vep}-\tilde{Z}_{\vep}k_{\vep}\p_X \tilde{\zeta}_{\vep}\, .
\eeqs
Similarly, it holds that 
\begin{align}\label{eq-pep}
    \tilde{\cL}_0^{\vep} \begin{pmatrix} 
  \p_{\vep} \tilde{\zeta}_{\vep} \\
   \p_{\vep} \tilde{\vp}_{\vep}
    \end{pmatrix}= -\p_{\vep}k_{\vep} \, \tilde{\cL}_1^{\vep}\p_X \begin{pmatrix} 
 \tilde{\zeta}_{\vep} \\
 \tilde{\vp}_{\vep}
    \end{pmatrix}.
\end{align}
Combining \eqref{eq-pc} and \eqref{eq-pep}, we  readily get that 
\begin{align*}
    \tilde{\cL}_0^{\vep}\, \tilde p_2^{\vep}= \frac{\vep  k_{\vep} \p_{\vep}k_{\vep}}{\p_c k_{\vep}} \tilde p_1^{\vep} \quad \quad  \text{ with \,} \tilde p_1^{\vep}=\vep^{-1}\begin{pmatrix}
     \p_X \tilde{\zeta}_{\vep}  \\
    \p_X \tilde{\vp}_{\vep}  
\end{pmatrix}, \quad \tilde p_2^{\vep}=\p_{\vep}\begin{pmatrix}
    \tilde{\zeta}_{\vep}  \\
 \tilde{\vp}_{\vep} 
\end{pmatrix}-\frac{\p_{\vep} k_{\vep}}{\p_c k_{\vep}}\p_c \begin{pmatrix}
    \tilde{\zeta}_{\vep}  \\
 \tilde{\vp}_{\vep}  
\end{pmatrix}.
\end{align*}
Changing back to the original variable $x$, we find that 
\begin{align*}
    \cL_0^{\vep}\, 
    p_2^{\vep}= \frac{\vep  k_{\vep} \p_{\vep}k_{\vep}}{c\,\p_c k_{\vep}}  p_1^{\vep} \quad \quad  \text{ with \,}  p_1^{\vep}=\vep^{-1}\begin{pmatrix}
     \p_x {\zeta}_{\vep}  \\
    \p_x{\vp}_{\vep}  
\end{pmatrix}, \quad  p_2^{\vep}=\p_{\vep}\begin{pmatrix}
   {\zeta}_{\vep}  \\
{\vp}_{\vep} 
\end{pmatrix}-\frac{\p_{\vep} k_{\vep}}{\p_c k_{\vep}} \begin{pmatrix}
   \p_c{\zeta}_{\vep}  \\
 \p_c(c\vp_{\vep} )/c
\end{pmatrix}.
\end{align*}
Consequently, we find the two eigenfunctions of $L_0^{\vep}$ through the relation \eqref{relationLtL}:
\begin{align*}
    q_j^{\vep}(0,\cdot)=\left( \begin{array}{cc}
   1   & 0 \\
   -Z_{\vep}   & 1
 \end{array}\right) p_j^{\vep}, \qquad j=1,2.
\end{align*}

\subsection{Construction of $q_3^{\vep}(0,\cdot), q_4^{\vep}(0,\cdot)$} We will prove the existence of $q_3^{\vep}(0,\cdot), q_4^{\vep}(0,\cdot)$ in the form of \eqref{q3q4} by the Lyapunov--Schmidt method. It suffices to find a constant $\Lambda_0$, a smooth function $ \Lambda_1(\vep)$
and a function  $u^{\vep}(x)  \perp \ker(L_0^0)$ 
in the space $\tilde{Y}_{\per}\coloneqq\{f\in {Y}_{\per} : f=\text{(even, odd)}^t\},$
such that 
\begin{align}\label{id-uep}
    L_0^{\vep} u^{\vep}=\begin{pmatrix}
      0 \\
   1
   \end{pmatrix}+\big(\vep \Lambda_0\, +\vep^2
   \Lambda_1(\vep)\big)q_1^{\vep}(0,\cdot).
\end{align}
We search for $u^{\vep}$ of the form
\beq\label{form-uep}
u^{\vep}(x)=u^0(x)+ \vep u^1(x)+\vep^2 u^2(\vep,x), \qquad (u^j\in \tilde{Y}_{\per}\,).
\eeq 
Let $ L^{\vep}_0=L^{0}_0+\vep L^{1}_0 +
\vep^2 \frac{L^{\vep}_0-L^{0}_0-\vep L^{1}_0}{\vep^2}$. Substituting  \eqref{form-uep}  into \eqref{id-uep}
 and collecting terms of order $\vep^0, \vep^1,\vep^2$, we find 
\begin{align}
L_0^0 u^0&=\begin{pmatrix}
      0 \\
   1
   \end{pmatrix}, \quad \text{ so that } u^0=-\begin{pmatrix}
    1/\alpha   \\
   0
   \end{pmatrix},\notag\\
   L_0^{0} u^1&=-L_0^1 u^0+\Lambda_0\, q_1^0, \label{eq-u1} \\
 L_0^{\vep} u^2&= -\vep^{-2}\big(L_0^{\vep}-L_0^0-\vep L_0^1\big)\,u^0-\vep^{-1}(L_0^{\vep}-L_0^0) u^1+\Lambda_0 \vep^{-1}\big(q_1^{\vep}-q_1^0\big)+\Lambda_1(\vep)q_1^{\vep}\coloneqq\cH^{\vep}.\label{eq-u2} 
\end{align} 
From the Fredholm alternative we see that the solvability of \eqref{eq-u1} is equivalent to 
the orthogonality condition 
\begin{align}\label{orthocond}
  \big\langle Jq_j^0(0,\cdot),\, -L_0^1 u^0+\Lambda_0\, q_1^0(0,\cdot) \big\rangle=0, \quad j=1,2,3,
\end{align}
where $\{Jq_j^0(0,\cdot)\}_{j=1,2,3}$ form a basis of the  kernel of $(L_0^0)^{*}$ and are defined by 
\begin{align*}
    Jq_1^0(0,\cdot)=\bpm \ch(\kappa) \cos (x)\\ \sh(\kappa)\sin (x)\epm, \quad Jq_2^0(0,\cdot)=\bpm \ch(\kappa)\sin (x) \\ -\sh(\kappa)\cos (x)\epm,\quad
     Jq_3^0(0,\cdot)=\bpm 1\\ 0\epm. \quad 
\end{align*}
Note that  $Jq_4=(0, -1/{\alpha})^t$ is a generalized eigenfunction of $(L_0^0)^{*}$.
However, one can readily check that 
\begin{align*}
L^1 u^0&=-\f{\kappa}{\alpha}\bpm -\p_x v_{\vep}\\\p_x Z_{\vep}\epm\bigg|_{\vep=0} = -\f{\kappa^2}{\alpha}
\bpm  \ch (\kappa)\sin (x)\\ \sh (\kappa) \cos (x) \epm= \bpm \textrm{ odd } \\ \textrm{ even } \epm, \\
q_1^0(0,\cdot) &=
\bpm  -\sh (\kappa)\sin (x)\\ \ch (\kappa) \cos( x )\epm= \bpm \textrm{ odd } \\ \textrm{ even } \epm,
\end{align*}
and thus \eqref{orthocond} holds for $j=1,3$ automatically. On the other hand, one can determine $\Lambda_0$ by the condition for $j=2$:
\begin{align*}
    \Lambda_0=\f{\big\langle Jq_2^0(0,\cdot),\, L_0^1 u^0  \big\rangle}{\big\langle Jq_2^0(0,\cdot),\,  q_1^0(0,\cdot) \big\rangle}=\frac{\kappa^2}{\alpha \sh (2\kappa)}.
\end{align*}
With this in hand, and using the fact $G_0[0](e^{\pm\i x})=\kappa |D| \tah (\kappa |D|)(e^{\pm \i x})=\kappa \tah (\kappa )(e^{\pm\i x}),$
we can solve $u^1$ explicitly by 
\eqref{eq-u1}:
\begin{align*}
    u^1=- \frac{\kappa}{2\alpha}  \begin{pmatrix}
   \ch (\kappa) \cos (x)\\ -\sh (\kappa) \sin (x)
   \end{pmatrix}.
\end{align*}

Let us now proceed to prove the existence of $(\Lambda_1(\vep), u^2(\vep,\cdot))$ that
satisfies \eqref{eq-u2}
which will be provided by the Implicit Function Theorem.  Denote by $\mathbb{Q}_0$ the orthogonal projection onto the orthogonal complement in $Y_{\per}$ of 
$\ker (L_0^0)$ and by $\mathbb{P}_0$ the orthogonal projection onto the orthogonal complement in $Y_{\per}$ of $\ker (L_0^0)^*$. Then \eqref{eq-u2} is equivalent to
\begin{align}
    &\tilde{L}_0^{\vep} u^2= \mathbb{P}_0 \cH^{\vep}, \quad \text{ with } \,\,\tilde{L}_0^{\vep}=\mathbb{P}_0  L_0^{\vep} \mathbb{Q}_0\, ,\\
    & F_j(\vep,\Lambda_1)=\langle Jq_j^0(0,\cdot), (L_0^{\vep}-L_0^0)u^2-\cH^{\vep}\rangle=0, \qquad j=1,2,3, \label{orthcond-123}
\end{align}
 where we have abused notation by considering $\Lambda_1$ as a parameter rather than a function of $\vep$.
Note that we have used the fact $(L_0^0)^{*} (Jq_j^0(0,\cdot))=0,\, j=1,2,3$.

On the one hand,  the formula 
$(\tilde{L}_0^{\vep})^{-1}=(\tilde{L}_0^{0})^{-1}\big(\Id + (\tilde{L}_0^{\vep}-\tilde{L}_0^{0})(\tilde{L}_0^{0})^{-1}\big)^{-1}$ defines an inverse $(\tilde{L}_0^{\vep})^{-1}\colon \tilde{\tilde{Y}}_{\per}\to {\tilde{Y}}_{\per} \cap {\rm Dom}(L_0^0)$  to $\tilde L_0^\vep$ with
\begin{align*}
    \|(\tilde{L}_0^{\vep})^{-1}-(\tilde{L}_0^{0})^{-1}\|_{B(\tilde{\tilde{Y}}_{\per},\tilde{Y}_{\per})}\lesssim \vep, 
\end{align*}
where ${\tilde{\tilde{Y}}}_{\per}\coloneqq \{f\in {Y}_{\per} : f=\text{(odd, even)}^t\}$.
 The existence of the function $u^2(\vep,\cdot)\in \tilde{Y}_{\per}$ is thus ensured, once $\Lambda_1(\vep)$ is known.
On the other hand, it is observed that 
\[
\cH^{\vep}=\bpm \textrm{ odd } \\ \textrm{ even } \epm
\quad \text{ and }
\quad
(L_0^{\vep}-L_0^0)u^2=\bpm \textrm{ odd } \\ \textrm{ even } \epm.
\]
Thus \eqref{orthcond-123} holds directly 
for $j=1,3$ and it suffices to find a smooth function $\Lambda_1(\vep)$ such that
\begin{align*}
    F_2(\vep, \Lambda_1)=\langle Jq_2^0(0,\cdot), (L_0^{\vep}-L_0^0)u^2-\cH^{\vep} \rangle =0. 
\end{align*}
Setting $\vep=0$, one finds a constant $\Lambda_{10}$ such that
 $F(0,  \Lambda_{10})=0$. Moreover,  
 \begin{align*}
 (\p_{\Lambda_1}  F_2)(0, \Lambda_1)=-\langle Jq_2^0(0,\cdot), q_1^0(0,\cdot) \rangle =\pi \sh(2\kpa)\neq 0\,.
 \end{align*}
Consequently, by the Implicit Function Theorem, there exists ${\vep}_3>0$ 
and a $C^1$ curve $\Lambda_1({\vep}): [-{\vep}_3, {\vep}_3]\rightarrow \mR$, such that $\Lambda_{1}(0)=\Lambda_{10}$ and 
\begin{align*}
    F_2(\vep, \Lambda_1(\vep))=0, \quad \forall\, \vep\in [-{\vep}_3, {\vep}_3].
\end{align*}

\section{Expansions of the Dirichlet--Neumann operator}
This appendix gathers several useful results concerning the transformed Dirichlet–Neumann operator $G_{\xi}[\zeta_{\vep}]$, which are used in the main body of the paper. 
\begin{prop}\label{prop-DN}
Let $\vep_0>0$ be small enough and $s\geq 0$.
The map 
\begin{align*}
    G_{\xi}[\zeta_{\vep}]: (-1/2,1/2]\times& [-\vep_0, \vep_0] \rightarrow B\big (H^{s+1}(\mathbb{T}_{2\pi}),H^{s}(\mathbb{T}_{2\pi})\big) \\
   & (\xi,\vep) \mapsto (\vp \mapsto G_{\xi}[\zeta_{\vep}]\vp)
\end{align*}
is smooth. Moreover, denoting 
\begin{align*}
    G_{\ell}^k= \big(\f{1}{\i\kpa}\p_{\xi}\big)^{\ell} \p_{\vep}^k\big(G_{\xi}[\zeta_{\vep}]\big)\big|_{\xi=\vep=0},
\end{align*}
it holds that 
\beq\label{expan-DN}
\begin{aligned}
G_0^0=|\kpa D|\tah |\kpa D|\,, & \qquad G_1^0=-\i \tah(\kpa D)-\f{\kpa\p_x}{\ch^2(\kpa D)} ; \\ 
 G_0^1=-G_0^0\,\zeta^1G_0^0-\kpa \p_x(\zeta^1 \kpa\p_x), &\qquad G_1^1=-\big(G_0^0\,\zeta^1G_1^0+G_1^0\,\zeta^1 G_0^0\big)-\kpa \big(\p_x(\zeta^1\cdot)+\zeta^1\p_x\big),
\end{aligned}
\eeq
where $D={-\i \p_x}$ and $\zeta^1=\p_{\vep}\zeta_{\vep}|_{\vep=0}$.
\end{prop}
\begin{proof}
To prove the smoothness of $ G_{\xi}[\zeta_{\vep}]$ in terms of $\xi$ and $\vep$, we use the following equivalent definition of $G_{\xi}[\zeta_{\vep}]$: 
\begin{align*}
G_{\xi}[\zeta_{\vep}]\vp=
\p_z{\Psi}|_{z=1+\zeta_{\vep}}-\p_x\zeta_{\vep}(\p_x+\i \xi)\Psi|_{z=1+\zeta_{\vep}},
\end{align*}
where $\Psi=\Psi(\vep,\xi,x,z)$ solves the following elliptic problem in $\mathbb{T}_{2\pi}\times [0,1+\zeta_{\vep}]$:
\begin{align*}
\big((\p_x+\i \xi)^2+\p_z^2\big)\Psi=0\, , \qquad \p_z\Psi|_{z=0}=0, \qquad \Psi|_{z=1+\zeta_{\vep}}=\vp.
\end{align*}
To get properties for $\Psi$ and thus for $G_{\xi}[\zeta_{\vep}]\vp,$
one can use various changes of variables to straighten the domain. Since $\zeta_{\vep}\in C^{\infty}(\mathbb{T}_{2\pi})$, we can neglect the loss of half a derivative on $\zeta_{\vep}$ and choose a simple one: 
\begin{align*}
    (x,z)\rightarrow (x, \tilde{z}\coloneqq{z}/(1+\zeta_{\vep})), \qquad \Phi(\vep,\xi, x,\tilde{z})=\Psi(\vep,\xi, x,z).
\end{align*}
Then the elliptic problem transforms to (dropping the tilde for simplicity)
\begin{align}\label{ellipticpb}
\f{1}{(1+\zeta_{\vep})^2}\p_z^2 \Phi+\Big((\p_x+\i \xi)-\f{z\zeta_{\vep}}{1+\zeta_{\vep}}\p_z \Big)\Big((\p_x+\i \xi)-\f{z\zeta_{\vep}}{1+\zeta_{\vep}}\p_z \Big)\Phi=0\, \,\, \, \text{in }\Omega=\mathbb{T}_{2\pi}\times [0,1],
\end{align}
with the boundary conditions
\beqs
 \p_z\Phi|_{z=0}=0, \quad \Phi|_{z=1}=\vp\,.
\eeqs
Moreover, it holds that
\begin{align}\label{def-DN-new}
G_{\xi}[\zeta_{\vep}]\vp=\f{1+(\p_x\zeta_{\vep})^2}{1+\zeta_{\vep}}\p_z{\Phi}|_{z=1}- \p_x\zeta_{\vep}(\p_x+\i \xi)\Phi|_{z=1}.
\end{align}

Since $\zeta_{\vep}\in C^{\infty}(\mathbb{T}_{2\pi})$, one can follow the same proof as in Section A.1.2 of \cite{Lannes} to show that for any fixed $\xi_0\in (-\f12,\f12]$, 
the map 
\begin{align*}
    G_{\xi_0}[\zeta_{\vep}]:\, & H^s(\mathbb{T}_{2\pi})\rightarrow 
    B\big (H^{s+1}(\mathbb{T}_{2\pi}),H^{s}(\mathbb{T}_{2\pi})\big) \\
   & \zeta_{\vep} \mapsto (\vp \mapsto G_{\xi_0}[\zeta_{\vep}]\vp) 
\end{align*}
is analytic. Since $G_{\xi_0}[\zeta_{\vep}]$
depends on $\vep$ only through $\zeta_{\vep}$ and 
$\zeta_{\vep}$ is analytic in $\vep$, we see that $G_{\xi_0}[\zeta_{\vep}]$ is analytic as a map from $[-\vep_0,\vep_0]$ to  $B\big (H^{s+1}(\mathbb{T}_{2\pi}),H^{s}(\mathbb{T}_{2\pi})\big)$. On the other hand, one can check during the proof that these norms of mapping are smooth in terms of $\xi$. This is indeed expected since the elliptic problem involves polynomials in terms of $\xi$ and thus 
the map 
\begin{align*}
   & (-1/2,1/2] \rightarrow B\big(H^{s+1}(\mathbb{T}_{2\pi}), H^{s+\f32}(\Omega)\big) \\
 & \qquad \qquad \xi \mapsto \big(\vp\rightarrow \Phi \big)
\end{align*}
is smooth. Therefore, as an element of $B\big (H^{s+1}(\mathbb{T}_{2\pi}),H^{s}(\mathbb{T}_{2\pi})\big),$
$G_{\xi}[\zeta_{\vep}]$ is smooth in $\xi\in (-1/2,1/2]$.

  Given the smooth dependence of  $G_{\xi}[\zeta_{\vep}]$ on both  $\xi$ and $\vep$, we first expand  
  $G_{\xi}[\zeta_{\vep}]$ in terms of $\vep$  treating $\xi$ as a parameter for which the operator remains smooth. Since $G_{\xi}[\zeta_{\vep}]$ depends on $\vep$ only through $\zeta_{\vep}$, the matter reduces to expanding 
  $G_{\xi}[\zeta_{\vep}]$ in terms of $\zeta_{\vep}$.
  Following the same  approach\footnote{The only difference in the present setting is that we consider the elliptic problem on $\mathbb{T}_{2\pi}\times [0,1]$, with a smoothly varying parameter $\xi$.} as in Theorem 3.21
  and Proposition 3.44
  in \cite{Lannes}, we obtain 
\begin{align*}
G_{\xi}[\zeta_{\vep}]=G_{\xi}[0]+\bigg( \underbrace{ G_{\xi}[0]\zeta_{\vep}G_{\xi}[0]-\kpa(\p_x+\i \xi)\big(\zeta_{\vep}\kpa(\p_x+\i \xi)\big) }_{\coloneqq G_{\xi}^1}\bigg)+\cO_{B\big (H^{s+1}(\mathbb{T}_{2\pi}),H^{s}(\mathbb{T}_{2\pi})\big)}(\vep^2),
\end{align*}
which yields the expressions for $G_0^0$ and $G_0^1$ in \eqref{expan-DN}. Using the smoothness of  
$G_{\xi}[0]$ and $G_{\xi}^{1}$ with respect to $\xi,$
 we further expand them in powers of
 $\xi$ to get the formulas for $G_1^0$ and $G_1^1$. 
\end{proof}

\section{Some properties of the extension operators } \label{appen-extension} 

In this section, we collect several useful properties for the extension operators 
$\cU_{\xi}(\vep)$ and
$\cU^{\vep}(\xi)$ defined 
 respectively in \eqref{def-extop-ep} and \eqref{def-extop-xi}. 

\begin{prop}[Properties of the extension operators]\label{prop-extop}
\,

 \textnormal{(1).} The extension operators $\cU_{\xi}({\vep})$ and 
 $\cU^{\vep}(\xi)$
  preserve the symplectic inner product 
    \beq\label{preseve-syminner}
   \mathcal{S}(f,g)= \mathcal{S}(\mathcal{U}_{\xi}(\vep)f, \mathcal{U}_{\xi}(\vep)g),   \qquad   \mathcal{S}(f,g)= \mathcal{S}(\mathcal{U}^{\vep}(\xi)f, \mathcal{U}^{\vep}(\xi)g)  
     \eeq 
    where $\mathcal{S}(f,g)\coloneqq\langle Jf, g \rangle$.
    
 \textnormal{(2).} The extension operator  $\cU^{\vep}(\xi)$ commutes with the reversibility operator: 
\beq\label{reversibilitycom}
\mathcal{U}^{\vep}(\xi)\mathcal{I}=\mathcal{I}\,\mathcal{U}^{\vep}(\xi),
\eeq
where 
\beq \label{reversibilityop}
\mathcal{I} g =\mathrm{diag}\, (1,-1) \overline{g(-x)}.
\eeq
\end{prop}
Similar results have already been established in \cite{NRS-EP}. However, for the sake of self-containedness, we briefly sketch their proofs.
\begin{proof}

(1). 
Let us sketch the proof of the second identity; the first one follows from the same arguments. Define the dual extension operator 
\begin{align*}
\p_\xi\mathcal{V}^{\vep}(\xi)
=[\p_\xi(\Pi_{\xi}^{\vep})^{*}, (\Pi_{\xi}^{\vep})^{*}]\ \mathcal{V}^{\vep}(\xi)\,, \qquad
\mathcal{V}^{\vep}(0)=\Id\,,
\end{align*}
where $(\Pi_{\xi}^{\vep})^{*}\coloneqq\oint_{\p B_{\ep_1}(0)} \big(\lambda I-(L_{\xi}^{\vep})^{*}\big)^{-1}\, \d \lambda$.
Let $\xi$ belong to the interval of existence for the 
ODE \eqref{def-extop-xi}.
One proves, on the one hand, $(\mathcal{V}^{\vep})^{*}\cU^{\vep}(\xi)=\Id$ and on the other hand, from the Hamiltonian symmetry 
$(\Pi_\xi^\vep)^*=-J^{-1}\,\Pi_\xi^\vep\,J$, that
$\mathcal{V}^{\vep}(\xi)\,=\,J^{-1}\,\mathcal{U}^{\vep}(\xi)\,J$.

(2). 
From the parity of the background waves, we readily find that $L_{\xi}^{\vep}\mathcal{I}=-\mathcal{I}\,L_{\xi}^{\vep},$
which implies that $\Pi_{\xi}^{\vep}\mathcal{I}=-\mathcal{I}\,\Pi_{\xi}^{\vep}$ and thus \eqref{reversibilitycom}.
\end{proof}

\begin{lem}\label{lem-U1-real}
Let $\cU_1^{\vep}\coloneqq \f{1}{\i \kpa} \p_{\xi}\cU^{\vep}|_{\xi=0}$, then \, $\overline{\cU_1^{\vep}}=\cU_1^{\vep}$, so 
$\cU_1^{\vep}$ maps real-valued functions to real-valued functions. 
\begin{proof}
By the definition \eqref{def-extop-xi},  we have
\begin{align*}
  \cU_1^{\vep}=[\Pi_1^{\vep}, \Pi_0^{\vep}]\,, \qquad \text{ with \,}  \,\Pi_1^{\vep}\coloneqq \f{1}{\i \kpa} \p_{\xi}\Pi_{\xi}^{\vep}|_{\xi=0},\,\,  \Pi_0^{\vep}= \Pi_{\xi}^{\vep}|_{\xi=0} .
\end{align*}
where
\begin{align*}
    \Pi_{\xi}^{\vep}= \oint_{\p B_{\ep_0}(0)} (\lambda I -L_{\xi}^{\vep})^{-1}\, \d\lambda\,.
\end{align*}
On the one hand, $L_0^{\vep}$ has real valued coefficient
and $\overline{G_0^{\vep}[\zeta_{\vep}]}=G_0^{\vep}[\zeta_{\vep}]$, it follows that
$\overline{\Pi_0^{\vep}}=\Pi_0^{\vep}$. On the other hand, using the formula
\begin{align*}
    \Pi_{1}^{\vep}= \oint_{\p B_{\ep_0}(0)} (\lambda I -L_{\xi}^{0})^{-1} L_1^{\vep}\,(\lambda I -L_{\xi}^{\vep})^{-1}\, \d\lambda\,,
\end{align*}
and noting that $L_1^{\vep}=\f{1}{\i \kpa} \p_{\xi}L_{\xi}^{\vep}|_{\xi=0}$ also has real valued coefficients we conclude that $\overline{\Pi_1^{\vep}}=\Pi_1^{\vep}$. Therefore,     $\overline{\cU_1^{\vep}}=\cU_1^{\vep}$, and the result follows.
\end{proof}
\end{lem}

In the next lemma, we derive expansions of $\cU^0({\xi})$
in $\xi$ up to second order
\begin{lem}\label{lem-exp-extopxi}

\textnormal{(1).} For any $f\in Y_{\per},$
\beq \label{p1-extop}
\cU^0({\xi}) f=\sum_{j\in \{0,\pm 1\}}\langle \,\cU_{\xi}^0 f, e^{\i j x}  \rangle e^{\i j x} .
\eeq

 \textnormal{(2).}  If $V$ is a constant vector in $\mR^2$, then 
  \begin{align}\label{exp-extopxi}
   \cU^0(\xi) V=V, \qquad \cU^0(\xi)(e^{\pm \i x} V)=e^{\pm \i x} \sqrt{\f{\omega_{\pm1}(0)}{\omega_{\pm1}(\xi)}} \bpm 1  & 0\\ 0 & \f{\omega_{\pm1}(\xi)}{\omega_{\pm1}(0)} \epm  V.
  \end{align}
\end{lem}
\begin{proof}

Since the operator $L_{\xi}^0$ has constant coefficient, 
these can be  computed via the Fourier transform. Let $V$ be a constant vector in $\mR^2$. Then, by definition
\begin{align*}
\Pi_{\xi}^0 (e^{\i j x} V)&=\oint_{B_{\ep}(0)} (\lambda I-L_{\xi}^0)^{-1} (e^{\i j x}V)\, \d \lambda\, \\
&=e^{\i j x} \oint_{B_{\ep}(0)} \f{1}{(\lambda-\lambda_j^{+}(\xi))(\lambda-\lambda_j^{-}(\xi))} \bpm \lambda-\i (j+\xi) & (j+\xi)\tah(j+\xi)\\ -(\alpha+\beta(j+\xi)^2)& \lambda-\i(j+\xi)\epm V \d \lambda
\end{align*}
Since for any $j\notin\{ 0,\pm 1\}$, the eigenvalues $\lambda_j(\xi)$  lie outside of $B_{\ep}(0)$ for sufficiently small $\ep,\xi$, it follows that $\Pi_{\xi}^0(e^{\i j x} V)=0$ for any $j\notin\{ 0,\pm 1\}$. Together with the definition of \eqref{def-extop-xi}, this yields \eqref{p1-extop} and implies
\begin{align}\label{project-id}
\Pi_{\xi}^0 (V)=V, \qquad \Pi_{\xi}^0(e^{\pm \i x}V)=e^{\pm \i x}\cM_{\pm}(\xi) V,
\end{align}
where 
\begin{align*}
    \cM_{\pm}(\xi)\coloneqq\f12 \bpm 1 &\pm  \i \,\omega_{\pm 1}^{-1} \\\mp\i \,\omega_{\pm 1} & 1\epm (\xi)\, .
\end{align*}
Differentiating \eqref{project-id} with respect to $\xi$ gives:
\begin{align*}
 (\p_{\xi}\Pi_{\xi}^0)V=0, \qquad (\p_{\xi}\Pi_{\xi}^0)(e^{\pm \i x}V)=e^{\pm \i x}  \cM_{\pm}'(\xi)V.
\end{align*}
Consequently, we obtain $\cU_{\xi}^0(V)=V$ and
\begin{align*}
[\p_{\xi}\Pi_{\xi}^0, \Pi_{\xi}^0](e^{\pm\i x} V)= [\cM_{\pm}'(\xi), \cM_{\pm}(\xi)] (V e^{\pm\i x})=
e^{\pm\i x} \f{\omega_{\pm 1}'}{2\,\omega_{\pm 1}}(\xi) \bpm -1 & 0\\0& 1 \epm V,
\end{align*}
this, together with the definition of $\cU^0({\xi})$ in \eqref{def-extop-xi}, yields that
\begin{align*}
    \cU^0({\xi}) (e^{\pm\i x} V)=e^{\pm\i x} \exp \bigg(\int_0^{\xi} \f{\omega_{\pm 1}'}{2\,\omega_{\pm 1}}(\xi) \bpm -1 & 0\\0& 1 \epm \bigg) V=e^{\pm\i x}  \bpm \sqrt{\f{\omega_{\pm1}(0)}{\omega_{\pm1}(\xi)}}  & 0\\ 0 & \sqrt{\f{\omega_{\pm1}(\xi)}{\omega_{\pm1}(0)}} \epm  V .
\end{align*}
\end{proof}

\section{Relation between the index function $\tilde{C}(\alpha,\beta)$ and those in \cite{A-Segur,D-R-packets}}\label{appen-indices} 

As we mentioned in the introduction,  a modulational criterion has been formally derived by 
Djordjevic \& Redekopp \cite{D-R-packets} and Ablowitz \& Segur \cite{A-Segur}, by approximating 
the water waves system by the following nonlinear Schr\"odinger equation:
\begin{align*}
    \i\,\pt u+\lambda\, \p_x^2 u=\nu |u|^2 u,
\end{align*}
where $\lambda, \nu$ are defined 
in (2.24b, 2.24h) of \cite{A-Segur}. If $\lambda \nu>0$ 
the above equation is defocusing and it yields stability while if $\lambda \nu<0$ the equation is focusing and it yields instability.
In this section, we perform additional computations to establish the relationship between the index function $\tilde{C}(\alpha,\beta)$ found in Theorem \ref{thm-modulation} and $\lambda \nu$ defined in \cite{A-Segur}. 

For any $(\alpha,\beta)\in \rm I\cup II\cup III$,  
let $\kpa>0$ 
satisfies the dispersion relation \eqref{disp-relation}.
 Denote 
$$\sigma=\tah(\kpa),\, \quad \tilde{T}=\f{\beta}{\alpha}\kpa^2\qquad   e_{*}=\f12\Big(1+\f{2\kpa}{\sh(2\kpa)}\Big)+\beta\kpa\tah(\kpa)
\, $$
and 
\begin{align}
 & \chi=\f{(1-\sigma^2)(9-\sigma^2)+\tT(3-\sigma^2)(7-\sigma^2)}{\sigma^2-\tT(3-\sigma^2)} +8\sigma^2-2(1-\sigma^2)^2(1+\tT)-\f{3\sigma^2 \tT}{1+\tT}\,, \label{def-chi}\\ 
  & \chi_1=1+\f{e_{*}}{2}(1-\sigma^2)(1+\tilde{T}),\qquad  \mu_1=\f{e_{*}^{2}}{\alpha}-1>0\,. \nonumber
\end{align}
Then it can be checked that $\lambda, \nu$ defined in  \cite[(2.24b), (2.24h)]{A-Segur} reads\footnote{Here we transform the coefficients after the normalization of the system. For instance, $C_g^2/(gh)$ for the original system corresponds to $e_{*}^2/\alpha$ here.}
\begin{align*}
\lambda=\frac{w_1''(0)}{2\sqrt{\alpha\kpa}}\,,\qquad\quad  
\nu=\nu(\alpha,\beta)=\sqrt{\f{\alpha}{16 \kpa}}\bigg(
\f{8}{\alpha(1+\tT)}\f{\chi_1^2}{\mu_1}+\chi\bigg),
\end{align*}
where $w_1=w_1(\xi)$ is defined in \eqref{def-w1''}.
From the dispersion relation \eqref{disp-relation} it follows that  $\alpha=\kpa/(\sigma (1+\tT))$. Therefore, 
\begin{align}\label{rewrite-nu}
  \nu=\sqrt{\f{\alpha}{16\kpa}} \bigg( \f{\tah(\kpa)}{\kpa}\f{8\alpha}{e_{*}^2-\alpha} \Big(1+\f{\kpa e_{*}}{\alpha \sh(2\kpa)}\Big)^2+\chi\bigg).    
\end{align}
On the other hand,  by \eqref{def-k2}, it holds that  
\begin{align*}
    k_2=\f{\kpa^3}{64\alpha  \sh(2\kpa) (e_{*}-1)}\tilde{\chi},
\end{align*}
with $\tilde{\chi}$ defined in \eqref{def-tchi}. It follows from straightforward but lengthy calculations that (see also (C.1) \cite{Mariana-Tien-Erik})
\beqs 
\tilde{\chi}=-\f{8\alpha\kpa \ch^4(\kpa)}{\alpha+\beta\kpa^2}\chi=-4\alpha \sh(2\kpa)\ch^2(\kpa)\chi\,.
\eeqs
Moreover, differentiating the dispersion relation 
$\alpha+\beta\kpa^2=\kpa/\tah(\kpa)$
with respect to  the wave speed $c$, we get 
\begin{align*}
    \p_c\alpha+\kpa^2 \p_c\beta+2\beta\kpa\p_c\kpa=\Big(\f{1}{\tah(\kpa)}-\f{\kpa}{\sh^2(\kpa)}\Big)\p_c\kpa,
\end{align*}
which implies
\begin{align*}
    c\p_c\kpa=-2\bigg(\f{1}{\tah(\kpa)}-\f{\kpa}{\sh^2(\kpa)}-2\beta\kpa\bigg)^{-1}(\alpha+\beta\kpa^2)=\f{\kpa}{e_{*}-1}.
\end{align*}
Therefore, we find that
\begin{align*}
    -\f{2\kpa \,k_2}{c \p_c \kpa}=\f{\kpa^3 \ch^2(\kpa)}{8}\chi,
\end{align*}
substituting which into  definition \eqref{def-index}, leads to 
\begin{align}\label{rewrite-C}
    C(\alpha,\beta)= \f{\alpha \kpa^2 \sh(2\kpa)}{\,2(e_{*}^2-\alpha)}\Big(1+\f{\kpa e_{*}}{\alpha \sh(2\kpa)}\Big)^2+\f{\kpa^3 \ch^2(\kpa)}{8}\chi\,.
\end{align}
Combining \eqref{rewrite-nu} and \eqref{rewrite-C}, we find easily that 
\begin{align}\label{relation-Cnu}
C(\alpha,\beta)=
\sqrt{\f{\kpa^7}{4\alpha}}
\ch^2(\kpa) \, \nu(\alpha,\beta)\,. 
\end{align}

\section{Brief discussion of the high-frequency stability when $(\alpha,\beta)\in \rm II$ }\label{app-regionII-highcrossing}

While Theorem \ref{thm-fullmodulation} gives the  criterion for the modulational stability when 
$(\alpha,\beta)\in \rm II$,  to get a global (in-)stability of periodic waves when $(\alpha,\beta)$ belongs to this regime, we need to examine the spectrum near other spectral crossings (of $L_{\xi}^0$) far away from $0$.  
Let $\kappa_1, \kappa_2 $ be two positive roots of the dispersion relation \eqref{disp-relation} with $\kpa_1<\kpa_2$. 

On the one hand, if $(\alpha,\beta)$ is such that any other spectral crossing arises from the same spectral branch, that is, there is no $j, j'\in \mathbb{Z}$, $\xi\neq 0$, such that $\lambda^{-}_j(\xi)=\lambda_{j'}^{+}(\xi)$, where 
$\lambda^{\pm}_j(\xi)$ is defined in \eqref{defspec-0}
with $\kappa$ replaced by $\kpa_1$ or $\kpa_2$, then it follows from the same arguments as in Section 3 that all the spectrum of $L^{\vep}$ near these spectral crossings lies on the imaginary axis, since the Krein signatures are always identical. 

On the other hand, if $(\alpha,\beta)$ is such that there exists a spectral crossing that arises from  different spectral curves\footnote{This indeed happens for many $(\alpha,\beta)$. We refer to cases (6,7,9,12,14,15) in \cite[pp. 44--50]{Hur-Yang-capillary} where $\lambda^{\pm}_j(\xi)= \lambda_{j'}^{+}(\xi)$ with $|j,j'|\leq 2, |j-j'|=1$ when $\kpa=\kpa_2$ and $|j-j'|=2$ for $\kpa=\kpa_1$. However, according to the analysis in Section 4 of \cite{Hur-Yang-capillary}, there exists a curve $\kappa=\kappa_0(\tilde{T})$,
which lies above the green curve (i.e., in the regime $\alpha>1$), and intersects the blue curve in Figure \ref{figure1}. All the relevant cases occur on the left-hand side of this curve $\kappa_0(\tilde{T})$. Consequently, the region enclosed by this curve together with the blue and green curves is also one in which the periodic waves are spectrally stable.}, then generically, 
$L^{\vep}$ would have unstable spectrum near these crossings. 
Let $j,j',\xi_0$ be such that $\lambda^{-}_j(\xi_0)=\lambda_{j'}^{+}(\xi_0)\coloneqq\i \sigma_0$. To study the spectrum of $L_{\xi}^{\vep}$ for $|(\xi-\xi_0,\vep)|$ small enough, the strategy is still to find a two dimensional representation matrix $D_{\xi}^{\vep}$ whose eigenvalues coincide with the spectrum of $L_{\xi}^{\vep}$ near $\i \sigma_0$. The matter reduces to finding the basis and dual basis.
As in Section 3, we start from the basis of $L_{\xi}^0$
 \begin{align*}
  q_1^0(\xi,x)=  e^{\i j'  x} \begin{pmatrix}1\\\,\i\,\omega_{j'}(\xi)\end{pmatrix} , \qquad  q_2^0(\xi,x)=  e^{\i j  x} \begin{pmatrix}1\\\,-\i\,\omega_{j}(\xi)\end{pmatrix} ,
\end{align*}
and then obtain a basis of eigenspaces of $L_{\xi}^{\vep}$ associated to the spectrum of $L_{\xi}^{\vep}$ near $\i \sigma_0$ by 
making use of the Kato expansion operator $\cU_{\xi}(\vep)$: 
\begin{align*}
    q_j^{\vep}(\xi,x)= \cU_{\xi}(\vep)  q_j^{0}(\xi,x).
\end{align*}
The dual basis is then found through 
\begin{align*}
   \tilde{q}_1^{\vep}(\xi,\cdot)&=-\f{\i\,Jq_{1}^{\vep}(\xi, \cdot)}{4\pi \omega_{j'}(\xi)}\,,&
\tilde{q}_{2}^{\vep}(\xi, \cdot)=\f{\i\,Jq_{2}^{\vep}(\xi, \cdot)}{4\pi\omega_{j}(\xi)}.
\end{align*}
With these in hand, we then find the representation matrix: 
\begin{align*}
   D_{\xi}^{\vep}&=
   \f{\i}{4\pi}\left( \begin{array}{cc}
\f{b_1(\xi,\vep)}{\omega_{j'}(\xi)} & \f{a(\xi,\vep)}{\omega_{j'}(\xi)}\\[5pt]
 -\f{\overline{a(\xi, \vep)}}{{\omega_{j}(\xi)}} & -\f{b_2(\xi, \vep)}{{\omega_{j}(\xi)}}
    \end{array}\right), 
\end{align*}
where $ a(\xi,\vep)=\langle {q}_{1}^{\vep}(\xi, \cdot), A_{\xi}^{\vep}  {q}_{2}^{\vep}(\xi, \cdot)\rangle,\,\,\, b_k(\xi,\vep)=\langle {q}_{k}^{\vep}(\xi, \cdot), A_{\xi}^{\vep}  {q}_{k}^{\vep}(\xi, \cdot)\rangle, k=1,2$.
The eigenvalues of $D_{\xi}^{\vep}$ reads 
\begin{align*}
    \lambda_{\pm}=\f{\i}{8\pi} \bigg( \f{b_1(\xi,\vep)}{\omega_{j'}(\xi)}-\f{b_2(\xi,\vep)}{\omega_{j}(\xi)} \pm \sqrt{\Delta} \bigg)
\end{align*}
with 
\begin{align*}
    \Delta=\f{4|a|^2(\xi,\vep)}{\, \omega_{j'}\omega_{j}(\xi)}-\bigg(\f{b_1(\xi,\vep)}{\omega_{j'}(\xi)}+\f{b_2(\xi,\vep)}{\omega_{j}(\xi)}  \bigg)^2\,.
\end{align*}
Since 
\begin{align*}
\f{b_1(\xi_0,0)}{\omega_{j'}(\xi_0)}+\f{b_2(\xi_0,0)}{\omega_{j}(\xi_0)}=0, \qquad \p_{\xi}\bigg(\f{b_1(\xi_0,0)}{\omega_{j'}(\xi_0)}+\f{b_2(\xi_0,0)}{\omega_{j}(\xi_0)}\bigg)=0,
\end{align*}
it follows from the implicit function theorem that there exists a $C^1$ curve $\xi=\mathrm{\Xi}(\vep)$ such that 
\begin{align*}
    \f{b_1(\mathrm{\Xi}(\vep),\vep)}{\omega_{j'}(\mathrm{\Xi}(\vep))}+\f{b_2(\mathrm{\Xi}(\vep),\vep)}{\omega_{j}(\mathrm{\Xi}(\vep))}\equiv 0.
\end{align*}
  Consequently, the instability of $L^{\vep}$ around $\i \sigma_0$ reduces to the non-vanishing of the parameter $a(\mathrm{\Xi}(\vep),\vep)$. 
Straightforward computations show that 
  \begin{align*}
      a(\mathrm{\Xi}(\vep),\vep)=\Gamma(\kpa) \vep^{|j-j'|}+\cO(\vep^{|j-j'|+1}).
  \end{align*}
Since $\Gamma(\kpa)$ is a real analytic function in  $\kpa$, for a given $\kpa$ it is generically non-vanishing.
If this is the case, we find unstable spectrum of $L^{\vep}$ whose positive real part is of order $\cO(|\vep|^{|j-j'|})$. Let us remark that such a procedure for finding unstable spectrum around a spectral crossing with opposite Krein signatures has been employed in \cite{NRS-EP,JRSY-cmp}. \\[3pt]
\qquad 

\textbf{Data Availability Statement:}
Some data are generated from numerical simulations used to produce Figures  \ref{figure1}, \ref{fig:figure2} and \ref{fig:figure3}.
The MATLAB scripts used for data generation are available from the corresponding author upon reasonable request.

\textbf{Conflicts of Interest Statement:}
The authors declare that they have no conflicts of interest.

\section*{Acknowledgements}
C. Sun is partly supported by the ANR project ANR-24-CE40-3260, and E. Wahlén is supported by the Swedish Research Council (grant no. 2020-00440). Part of this work was carried out during C. Sun’s visits to Lund University, whose warm hospitality is gratefully appreciated. The authors also sincerely thank Mariana Haragus and Zhao Yang for their valuable comments and suggestions. 
\bibliographystyle{siam}
\bibliography{pwwref}
\end{document}